\documentclass{ws-ijbc}

\usepackage{balance}
\usepackage{amsfonts, amsmath, amssymb, xfrac}
\usepackage{enumitem}
\setlist{leftmargin=*,align=left,noitemsep,topsep=0pt,parsep=0pt,partopsep=0pt}

\usepackage{array}
\usepackage[english]{babel}

\usepackage{graphics}
\usepackage[config, labelfont={sf,bf}, textfont=sf,caption=false]{subfig}
\graphicspath{{./figs/}}
\usepackage[suffix=, outdir=./figs/]{epstopdf}
\usepackage{array}
\usepackage{eepic}
\usepackage{epsfig}
\usepackage{calc}
\usepackage{amstext}

\usepackage{bm}
\usepackage{lineno,hyperref}
\modulolinenumbers[5]

\usepackage[usenames,dvipsnames]{color}
\definecolor{MyDarkGreen}{rgb}{0.0,0.4,0.0}

\usepackage{listings}
\begin{document}

\catchline{}{}{}{}{} 

\markboth{G. Leonov et al.}{Homoclinic Bifurcations in Lorenz-like Systems}

\title{Homoclinic Bifurcations of the Merging Strange Attractors
in the Lorenz-like System}

\author{G. A. Leonov}
\address{St. Petersburg State University, Russia}

\author{R. N. Mokaev}
\address{St. Petersburg State University, Russia\\
University of Jyv\"{a}skyl\"{a}, Finland}

\author{N. V. Kuznetsov}
\address{St. Petersburg State University, Russia\\
University of Jyv\"{a}skyl\"{a}, Finland \\
Institute of Problems of Mechanical Engineering RAS, Russia\\
nkuznetsov239@gmail.com}

\author{T. N. Mokaev}
\address{St. Petersburg State University, Russia}

\maketitle

\begin{history}
\received{(to be inserted by publisher)}
\end{history}

\begin{abstract}
In this article we construct the parameter region where the existence of
a homoclinic orbit to a zero equilibrium state of saddle type in the Lorenz-like system 
will be analytically proved in the case of a nonnegative saddle value.
Then, for a qualitative description of the different types of homoclinic bifurcations,
a numerical analysis of the detected parameter region is carried out
to discover several new interesting bifurcation scenarios.
\end{abstract}


\section{\label{sec:intro} Introduction}

In 1963, E. Lorenz \cite {Lorenz-1963} discovered a strange attractor and described
a homoclinic bifurcation of the change in attraction of separatrices of a saddle in a
three-mode model of two-dimensional convection
\begin{equation}\label{eq:lorenz}
    \begin{cases}
        \dot{x} = -\sigma (x - y), \\
        \dot{y} = r x - dy - x z, \\
        \dot{z} = - b z + x y,
    \end{cases}
\end{equation}
where $d = 1$, $\sigma > 0$ is a Prandtl number, $r > 0$ is a Rayleigh number,
$b > 0$ is a parameter that determines the ratio of the vertical and
horizontal dimensions of the convection cell.
Equations \eqref{sys:lorenz_like} are also encountered in other mechanical and physical problems,
for example, in the problem of fluid convection in a closed annular tube~\cite{RubenfeldS-1977},
for describing the mechanical model of a chaotic water wheel~\cite{TelG-2006},
the model of a dissipative oscillator with an inertial nonlinearity~\cite{NeimarkL-1992},
and the dynamics of a single-mode laser~\cite{Oraevsky-1981}.

Later on, for $d \neq 1$ it was suggested various Lorenz-like systems,
such as Chen system~\cite{ChenU-1999} ($d = -c$, $c > \tfrac{\sigma}{2}$, $r = c - \sigma$),
Lu system~\cite{LuChen-2002} ($d = -c$, $c > 0$, $r = 0$),
and Tigan-Yang systems~\cite{Tigan-2008,YangChen-2008} ($d = 0$),
which have interesting non-regular dynamics
differing in certain aspect form the Lorenz system dynamics~\cite{LeonovK-2015-AMC}.

Using the following smooth change of variables (see, e.g.~\cite{Leonov-2015-ND,LeonovAM-2017-VestSPSU}):
\begin{equation}
    \eta := \sigma (y - x), \quad \xi := z - \tfrac{x^2}{b}
\end{equation}
one can reduce system \eqref{eq:lorenz} to the form
\begin{equation}\label{eq:lorenz_intermed}
    \begin{cases}
        \dot{x} = \eta, \\
        \dot{\eta} = -(\sigma+d) \eta + \sigma \xi x + \sigma (r-d)x - \tfrac{\sigma}{b} x^3, \\
        \dot{\xi} = -b \xi - \tfrac{(2\sigma-b)}{b\sigma} x \eta,
    \end{cases}
\end{equation}
Then, by changing
$$\begin{aligned}
& t := \sqrt{\sigma(r-d)} t,\quad
x := \frac{x}{\sqrt{b(r-d)}},
& \vartheta := \frac{\eta}{\sqrt{b\sigma}(r-d)},\quad
u := \frac{\xi}{r-d}
\end{aligned}
$$
system \eqref{eq:lorenz_intermed} can be reduced to the form
\begin{equation}\label{sys:lorenz_like}
    \begin{cases}
        \dot{x} = \vartheta, \\
        \dot{\vartheta} = -\lambda \vartheta - x u + x - x^3, \\
        \dot{u} = - \alpha u - \beta x \vartheta,
    \end{cases}
\end{equation}
$$
\lambda=\frac{(\sigma+d)}{\sqrt{\sigma(r-d)}},\quad
\alpha=\frac{b}{\sqrt{\sigma(r-d)}},\quad\beta=\frac{2\sigma-b}{\sigma}.
$$
Using the following change of variables
(see, e.g. \cite{Leonov-2012-DRAN,Leonov-2013-IJBC,Leonov-2014-DAN,Leonov-2014-ND})
\[
    \nu := y, \quad u := z - x^2
\]
the well-known Shimizu-Morioka system~\cite{Shimizu1980201,LeonovAK-2015}
\begin{equation}\label{eq:shimizu_morioka}
    \begin{cases}
        \dot{x} = y, \\
        \dot{y} = (1 - z) x -\lambda y, \\
        \dot{z} = - \alpha (z - x^2)
    \end{cases}
\end{equation}
with $\beta = 2$ can be also transformed to the form \eqref{sys:lorenz_like}.

The following Lorenz-like system from \cite{OvsyannikovT-2017}
\begin{equation}\label{eq:ovsyannikov}
    \begin{cases}
        \dot{X} = Y, \\
        \dot{Y} = X -\lambda Y - XZ - X^3, \\
        \dot{Z} = - \alpha Z + B X^2.
    \end{cases}
\end{equation}
can be also reduced to the system of form \eqref{sys:lorenz_like} using
by changing the variables
\begin{equation}\label{eq:transform:leonov_ovsyannikov}
    X := \sqrt{\frac{\alpha}{B+\alpha}} \, x, \quad
    Y := \sqrt{\frac{\alpha}{B+\alpha}} \, y, \quad
    Z := z + \frac{B}{B+\alpha} x^2,
\end{equation}
and if $\beta = \frac{2B}{B+\alpha} < 2$.

Thus, in this article it is convenient
for us to consider and study system~\eqref{sys:lorenz_like}.
Its equilibria have the following form:
\begin{equation}\label{equib}
    S_0 = (0, \, 0, \, 0), \quad S_{\pm} = (\pm 1, \, 0, \, 0).
\end{equation}
It is easy to show that for positive $\alpha$, $\beta$, $\lambda$
the equilibrium state $S_0$ is always a saddle, and
$S_\pm$ are stable equilibria if $\beta < \frac{\lambda(\lambda \alpha + \alpha^2 + 2)}{(\lambda+\alpha)}$.

The seminal work \cite{Lorenz-1963} initiated the development of chaotic dynamics
and, in particular, the description of scenarios of transition to chaos.
An important role in such scenarios plays
a homoclinic bifurcation.
They are associated with global changes in dynamics in the phase space
of the system such as changes in attractors basins of attraction
and the emergence of chaotic dynamics
\cite{Wiggins-1988,ShilnikovTCh-1998,ShilnikovTCh-2001,HomburgS-2010,AfraimovichGLShT-2014}
and are applied in mechanics, theory of population and chemistry
(see, e.g, \cite{KuznetsovMR-1992,Champneys-1998,ArgoulAR-1987}).
The high complexity of studying the motions in the vicinity of
a homoclinic trajectory and the homoclinic trajectory itself
was noted by Poincar\'{e}~\cite{Poincare-1892}.
In this paper for the Lorenz-like system~\eqref{sys:lorenz_like}
we analytically prove the existence of a homoclinic trajectory and
make an attempt to study the various scenarios of homoclinic bifurcation
numerically.

\section{Existence problem of homoclinic orbit. Analytical method.}

\begin{definition}
The homoclinic trajectory ${\rm x}(t)$ of an autonomous system of differential equations
\begin{equation}\label{eq:ode}
    \dot{{\rm x}} = f({\bf \rm x},q), \quad t \in \mathbb{R}, \quad
    {\bf \rm x} \in \mathbb{R}^n
\end{equation}
for a given value of parameter $q \in \mathbb{R}^m$
is a phase trajectory that is doubly asymptotic
to a saddle equilibrium ${\rm x}_0 \in \mathbb{R}^n$, i.e.
$$
    \lim_{t \to +\infty} {\rm x}(t) = \lim_{t \to -\infty} {\rm x}(t) = {\rm x}_0.
$$
\end{definition}
Here $f({\bf \rm x},q)$ is a smooth vector-function,
$\mathbb{R}^n = \{{\rm x}\}$ is a phase space of system \eqref{eq:ode}.
Let $\gamma(s), s \in [0,1]$ be a smooth path in the space
of the parameter $\{q\}=\mathbb{R}^m$.
Consider the following Tricomi problem \cite{Tricomi-1933,Leonov-2014-DAN} for system \eqref{eq:ode} and
the path $\gamma(s)$: {\it is there a point $q_0 \in \gamma(s)$
for which system \eqref{eq:ode} with $q_0$ has a homoclinic trajectory?}

Consider system \eqref{eq:ode} with $q=\gamma(s)$ and introduce the following notions.
Let ${\rm x}(t,s)^+$ be an outgoing separatrix of the saddle point ${\rm x}_0$
(i.e. $\lim\limits_{t\to-\infty}{\rm x}(t,s)^+={\rm x}_0$)
with a one-dimensional unstable manifold.
Define by ${\rm x}_{\Omega}(s)^+$
the point of the first crossing of separatrix ${\rm x}(t,s)^+$ with the closed set $\Omega$:
\[
    {\rm x}(t,s)^+ \, \not \in \, \Omega, \quad t\in(-\infty,T),
\]
\[
    {\rm x}(T,s)^+ = {\rm x}_{\Omega}(s)^+ \in \Omega.
\]
If there is no such crossing, we assume that ${\rm x}_{\Omega}(s)^+=\emptyset$ (the empty set).

Now let us formulate a general method for proving the existence of homoclinic trajectories
for systems \eqref{eq:ode} called the {\it Fishing principle}
\cite{Leonov-2012-PLA,Leonov-2013-IJBC,Leonov-2014-ND,LeonovKM-2015-EPJST}.
\begin{theorem}
\label{thm:fish_princ}
Suppose that for the path $\gamma(s)$ there is an $(n-1)$-dimensional bounded manifold $\Omega$ with
a piecewise-smooth edge $\partial\Omega$
that possesses the following properties:
\begin{enumerate}[label=(\roman*)]
    \item for any ${\rm x}\in\Omega\setminus\partial\Omega$ and $s\in[0,1]$,
        the vector $f({\rm x},\gamma(s))$ is transversal
        to the manifold $\Omega\setminus\partial\Omega$; \label{fish_princ:cond1}

    \item for any $s\in[0,1]$, $f({\rm x}_0,\gamma(s))=0$,
        the point ${\rm x}_0\in\partial\Omega$ is a saddle; \label{fish_princ:cond2}

    \item for $s=0$ the inclusion ${\rm x}_{\Omega}(0)^+\in\Omega\setminus\partial\Omega$
        is valid; \label{fish_princ:cond3}

    \item for $s=1$ the relation ${\rm x}_{\Omega}(1)^+=\emptyset$ is valid
        (i.e. ${\rm x}_{\Omega}(1)^+$ is an empty set); \label{fish_princ:cond4}

    \item for any $s\in[0,1]$ and ${\rm y}\in\partial\Omega\setminus {\rm x}_0$
        there exists a neighborhood $U({\rm y},\,\delta)=\{{\rm x} \in \mathbb{R}^n~\vert~\, |{\rm x}-{\rm y}|<\delta\}$
        such that ${\rm x}_{\Omega}(s)^+ \, \not\in \, U({\rm y},\delta)$. \label{fish_princ:cond5}
\end{enumerate}

If conditions \ref{fish_princ:cond1}--\ref{fish_princ:cond5} are satisfied,
then there exists $s_0\in[0,1]$ such that ${\rm x}(t,s_0)^+$ is a homoclinic trajectory
of the saddle point ${\rm x}_0$.
\end{theorem}
The proof and a geometric interpretation of Theorem~\ref{thm:fish_princ} are given,
e.g. in \cite{Leonov-2012-PLA}.

For the further investigation of system~\eqref{sys:lorenz_like}
we prove several auxiliary statements using the Lyapunov function
\begin{equation}\label{eq:lyap-func}
    V(x,\vartheta,u)=\vartheta^2-\frac{u^2}{\beta}-x^2+\frac{x^4}{2}
\end{equation}
which has the following derivative along the solutions of system~\eqref{sys:lorenz_like}
\begin{equation}\label{eq:lyap-func-der}
    \frac{dV}{dt} = ({\rm grad} V, f) = 2\left(-\lambda\vartheta(t)^2+
    \frac{\alpha}{\beta}u(t)^2\right).
\end{equation}

\begin{lemma}\label{lemma:lem1}
Let $\lambda=0$ and $\beta>0$. Then the separatrix
\[
    \lim\limits_{t\to -\infty}x(t)=\lim\limits_{t\to -\infty}\vartheta(t)=\lim\limits_{t\to -\infty} u (t)=0
\]
starting from the saddle $x = \vartheta = u = 0$ tends to infinity as $t \to +\infty$.
\end{lemma}
\begin{proof}
Assume the contrary.
Then in this case the separatrix has an $\omega$-limit point $x_0, \vartheta_0, u_0$.
From \eqref{eq:lyap-func-der} we can obtain that the arc of trajectory
$\tilde x(t)$, $\tilde\vartheta(t)$, $\tilde u(t)$, $t\in[0,T]$
with initial data $\tilde x(0) = x_0$, $\tilde\vartheta(0)=\vartheta_0$, $\tilde u(0)=u_0$
also consists of $\omega$-limit points and satisfies the relation
$\tilde u(t)=0$, $\forall\, t\in [0,T]$.
Then from the third equation of \eqref{sys:lorenz_like} we can obtain that
$\tilde\vartheta(t) \tilde x(t) = 0$, $\forall t\in[0,T]$.
This implies
$$
    (\tilde x(t)^2)^\bullet = 2 \, \tilde x(t) \, \tilde\vartheta(t) = 0, \quad \forall t \in [0,T].
$$
Thus, $\tilde x(t) = \rm{const}$, $\tilde\vartheta(t)=0$, $\tilde u(t)=0$, $\forall \, t \in [0,T]$.
Then it is easy to see that $\tilde x(t)$, $\tilde\vartheta(t)$, $\tilde u(t)$
are an equilibrium point.
From \eqref{eq:lyap-func-der} and the relation
$V(0,0,0) = 0 > -1/2 = V(\pm 1,0,0)$ it follow that
$\tilde x(t) = \tilde\vartheta(t) = \tilde u(t)\equiv 0$.
But in this case the trajectory $x(t),\vartheta(t),u(t)$
is a homoclinic one and $V(x(t),\vartheta(t),u(t))\equiv 0$.

Then from \eqref{eq:lyap-func-der} it follows that $u(t)\equiv 0$.
Repeating the arguments that we held earlier for
$\tilde x(t), \tilde\vartheta(t), \tilde u(t)$,
we get that $x(t)=\vartheta(t)=u(t)\equiv 0$.
The latter contradicts the assumption that $x(t),\vartheta(t),u(t)$ is a separatrix of the saddle
$x = \vartheta = u = 0$.

Thus, the separatrix $x(t),\vartheta(t),u(t)$ has no $\omega$-limit points and
tends to infinity as $t\to+\infty$.
\end{proof}

Consider system \eqref{sys:lorenz_like} with $\lambda \geq 0$, $\beta > 0$, and assume that
\begin{equation}\label{ineq:params_1}
    \alpha(\sqrt{\lambda^2+4}+\lambda)>2(\beta-2).
\end{equation}
Inequality \eqref{ineq:params_1} implies that there exists a number $L > 0$, such that
\begin{equation}\label{ineq:params_2}
    L > \frac{\sqrt{\lambda^2 + 4} - \lambda}{2}, \qquad \frac{\beta L}{\alpha + 2 L} < 1.
\end{equation}
Introduce the notions $K = \frac{\beta L}{\alpha + 2 L} < 1$ and $M = 1 - K$.

Consider the separatrix $x^+(t)$, $\vartheta^+(t)$, $u^+(t)$ of the zero saddle
point of system \eqref{sys:lorenz_like},
where $x(t)^+>0$, $\forall t\in(-\infty,\tau)$, $\tau$ is a number,
and $\lim\limits_{t \to -\infty}x(t)^+ = 0$ (i.e. positive outgoing separatrix is considered).

\begin{lemma}\label{lemma:lem2}
    Let the following inequality holds
    \begin{equation}\label{ineq:lem1}
        x^+(t) \geq 0, \qquad \forall\,t \in (-\infty, \tau]
    \end{equation}
    and $M > 0$.
    Then there exists a number $R > 0$ (independent of parameter $\tau$) such that
    $x^+(t) \leq R$, $|\vartheta^+(t)| \leq R$, $|u^+(t)| \leq R$ for all $t \in (-\infty, \tau]$.
\end{lemma}
\begin{proof}
    Define the manifold $\Phi$ as
    $$
        \Phi = \left\{ x \in [0, \, x_0], \, \vartheta \leq
        \min\left\{L x, \, \sqrt{\vartheta_0^2 + x^2 - \tfrac{M}{2} x^4}\right\}, \, u \geq - K x^2 \right\}.
    $$
    Here $\vartheta_0$ is an arbitrary positive number (e.g., $\vartheta_0 = 1$), and
    $x_0$ is a positive root of the equation
    \begin{equation}\label{eq:x0}
        \vartheta_0^2 + x^2 - \frac{M}{2} x^4 = 0.
    \end{equation}

    Inequalities \eqref{ineq:params_2}, $K > 0$ and $\vartheta \leq Lx$
    in a small vicinity of $x = \vartheta = 0$ implies that
    at a certain time interval $(-\infty,\tau_1)$, $\tau_1 < \tau$
    the separatrix $x^+(t)$, $\vartheta^+(t)$, $u^+(t)$ belongs to $\Phi$.
    In order to prove that the separatrix belongs to $\Phi$
    for all $t \in (-\infty,\,\tau]$ consider the parts of the boundary of $\Phi \cap \{ x > 0\}$
    and show that they transversal.
    These boundaries are the following surfaces or the parts of surfaces
    \begin{align*}
        &\delta_1\Phi = \left\{ (x,\vartheta,u) \in \mathbb{R}^3 ~\big|~
        x \in (0,\, x_0), \, \vartheta = L x, \, u \geq - K x^2 \right\},& \\
        &\delta_2\Phi = \left\{ (x,\vartheta,u) \in \mathbb{R}^3 ~\big|~
        x \in (0,\, x_0), \, \vartheta^2 = \vartheta_0^2 +  x^2 - \tfrac{M}{2} x^4, \, u \geq - K x^2 \right\},& \\
        &\delta_3\Phi = \left\{ (x,\vartheta,u) \in \mathbb{R}^3 ~\big|~
        x \in (0,\, x_0), \, \vartheta < L x, \, u = - K x^2 \right\},& \\
        &\delta_4\Phi = \left\{ (x,\vartheta,u) \in \mathbb{R}^3 ~\big|~
        x = x_0, \, \, \vartheta < 0 \right\}.&
    \end{align*}

    Consider a solution $x(t)$, $\vartheta(t)$, $u(t)$ of system \eqref{sys:lorenz_like},
    which at the point $t$ is on the surface $\delta_1\Phi$.
    From \eqref{ineq:params_2} it follows that
    $$
    \frac{d \vartheta}{d x} = - \lambda + \frac{1 - x^2 - u}{L}
    < - \lambda + \frac{1 - M x^2}{L} < - \lambda + \frac{1}{L},
            \qquad \forall \, x \in (0,\,x_0].
    $$
    It follows that
    $$
    \frac{d \vartheta}{d x} < L, \qquad \forall\, x \in (0,\,x_0], \quad
        v = L x, \quad u \geq - K x^2.
    $$
    Thus the surface $\delta_1\Phi$ is transversal and if $x(t)$, $\vartheta(t)$, $u(t)$
    is on the surface $\delta_1\Phi$, then for this solution there exists a number $\varepsilon(t)$
    such that $\vartheta(\tau) - L x(\tau) < 0$, $\forall \tau \in (t , t + \varepsilon(t))$.

    Now consider a solution $x(t)$, $\vartheta(t)$, $u(t)$ of system \eqref{sys:lorenz_like},
    which at the point $t$ is on the surface $\delta_2\Phi$ and consider the function
    $V(x, \vartheta) = \vartheta^2 - x^2 + \frac{M}{2} x^4$.
    On the set $\delta_2\Phi$ the following relations hold
    $$
        V = 0, \qquad \dot{V}(x,\vartheta) =
        -2 \lambda \vartheta^2(t) - 2 \vartheta(t) x(t) \big(u(t) + K x^2(t)\big) < 0.
    $$
    This implies transversality of $\delta_2\Phi$ and if $x(t)$, $\vartheta(t)$, $u(t)$
    is on the surface $\delta_2\Omega$, then for this solution there exists a number $\varepsilon(t)$
    such that $V(x(\tau),\vartheta(\tau)) < 0$, $\forall \tau \in (t , t + \varepsilon(t))$.

    Consider a solution $x(t)$, $\vartheta(t)$, $u(t)$ of system \eqref{sys:lorenz_like},
    which at the point $t$ is on the surface $\delta_3\Phi$.
    Then
    $$
    (u + K x^2)^\bullet = -\alpha u - \beta x \vartheta + 2 K x \vartheta =
    x ((-\beta + 2 K) v + \alpha K x) \, = \,
    \frac{\alpha \beta x}{\alpha + 2 L}(L x - v) > 0.
    $$
    This implies transversality of $\delta_3\Phi$ and if $x(t)$, $\vartheta(t)$, $u(t)$
    is on the surface $\delta_2\Phi$, then for this solution there exists a number $\varepsilon(t)$
    such that $u(\tau) + K x^2(\tau) > 0$, $\forall \tau \in (t, t+\varepsilon(t))$.
    Transversality of $\delta_4\Phi$ is obvious.

    From the relations proved above and the obvious inequality $\dot{x}(t) < 0$ for
    $x(t) = x_0$, $\vartheta(t) < 0$ it follows that the separatrix
    $\left(x^+(t),\,\vartheta^+(t),\,u^+(t)\right)$
    belongs to $\Phi$ for all $t \in (-\infty, \, \tau]$.

    Notice that the third equation of system \eqref{sys:lorenz_like} yields the relations
    $$
    (u + \frac{\beta}{2} x^2)^\bullet + \alpha (u + \frac{\beta}{2} x^2) \,=\,
        \frac{\alpha \beta}{2} x^2.
    $$
    Taking into account the boundedness of $x^+(t)$, i.e. $x^+(t) \in (0,\,x_0)$
    for all $t \in (-\infty,\,\tau]$, it follows the boundedness of $u^+(t)$ on $(-\infty,\,\tau]$:
    \begin{equation*}
        u^+(t) + \frac{\beta}{2}(x^+(t))^2 \leq \frac{\beta}{2}x_0^2,
        \quad \forall t \in (-\infty,\,\tau].
    \end{equation*}
    Hence, we have the estimate
    \begin{equation}\label{ineq:boundedness_u0}
        u^+(t) \leq \frac{\beta}{2}x_0^2,
        \quad \forall t \in (-\infty,\,\tau].
    \end{equation}

    The second equation of system \eqref{sys:lorenz_like} and boundedness of $x^+(t)$ and $u^+(t)$
    on $(-\infty,\,\tau]$ yields the boundedness of $\vartheta^+(t)$ for $\lambda > 0$ and boundedness of
    $\dot{\vartheta}^+(t)$ for $\lambda = 0$.
    From the first equation of the system \eqref{sys:lorenz_like} and from the boundedness
    of $x^+(t)$ and $\dot{\vartheta}^+(t)$ it follows the boundedness of
    $\vartheta^+(t)$ on $(-\infty,\,\tau]$.
    This implies the assertion of the lemma.
\end{proof}

\begin{lemma}\label{lemma:lem3}
    Suppose inequality \eqref{ineq:params_1} and the following inequality
    \begin{equation}\label{ineq:lem3:cond}
        \lambda^2 > 4 \left[\left(1 + \frac{\beta}{2}\right) \, x_0^2 - 1\right]
    \end{equation}
    hold, where $x_0$ -- is the positive root of equation \eqref{eq:x0}.
    Then $x^+(t) > 0$, $\forall \, t \in (-\infty, +\infty)$.
\end{lemma}
\begin{proof}
    Here the conditions of Lemma~\ref{lemma:lem2}.
    Therefore, if $x^+(t) > 0$, $\forall \, t \in (-\infty, \tau)$, then $x^+(t) \in \Omega$,
    $\forall \, t \in (-\infty, \tau)$ and relation \eqref{ineq:boundedness_u0} is satisfies.
    If $x^+(\tau) = 0$, then there exists a time moment $T < \tau$ such that
    for any $P > 0$
    \begin{equation}\label{eq:lem3:eq1}
        \vartheta^+(T) = - P x^+(T), \quad
        \vartheta^+(t) > - P x^+(t), \quad \forall \, t \in (-\infty, \, T).
    \end{equation}
    For the relation $\vartheta(T) = - P x(T)$ we have the following
    $$
        \frac{d \vartheta}{d x} \, > \, - \lambda + \frac{D}{P}, \qquad
        D = \left(1 + \frac{\beta}{2}\right) \, x_0^2 - 1.
    $$

    It is clear that if $P \,=\, \frac{\lambda}{2}
    + \sqrt{\frac{\lambda^2}{4} - D}$, then
    \begin{equation}\label{ineq:lem3:ineq1}
        \frac{d}{d x}(\vartheta + P x) > P - \lambda + \frac{D}{P} = 0
    \end{equation}
    on $\Omega$.
    Here we use condition \eqref{ineq:lem3:cond}.

    From \eqref{ineq:lem3:ineq1} it follows that
    \begin{equation*}
        (\vartheta(T)^+)^\bullet+P(x(T)^+)^\bullet>0.
    \end{equation*}
    It follows that there is no $T < \tau$, such that $v^+(T) = -P \, x^+(T)$,
    which contradicts relations \eqref{eq:lem3:eq1}.
    This implies Lemma~\ref{lemma:lem3}.
\end{proof}

The obtained lemmas and the Fishing principle (see Theorem~\ref{thm:fish_princ})
allows us to formulate for system \eqref{sys:lorenz_like}
the main analytical result of the article.
\begin{theorem}\label{thm:leonov}
    Consider a smooth path $\lambda(s)$, $\alpha(s)$, $\beta(s)$, $s \in [0,~1)$ in
    the parameter space of system \eqref{sys:lorenz_like}.
    Let
    \begin{equation}\label{cond:leonov:alpha_lambda}
        \begin{aligned}
             \lambda(0) = 0,& \quad \lim\limits_{s\to 1}\lambda(s)=+\infty,\\
             \limsup\limits_{s\to 1}\alpha(s)<+\infty,&\quad
        \limsup\limits_{s\to 1}\beta(s)<+\infty
        \end{aligned}
    \end{equation}
    and the following condition holds
    \begin{equation}\label{cond:leonov:beta}
        \alpha(s)(\sqrt{\lambda(s)^2+4}+\lambda(s))>2(\beta(s)-2),\,\,\,\forall s\in[0,1].
    \end{equation}
    Then there exists $s_0 \in (0,1)$ such that system \eqref{sys:lorenz_like} with
    $\alpha(s_0)$, $\beta(s_0)$, $\lambda(s_0)$ has a homoclinic trajectory.
\end{theorem}
\begin{proof}
    Here we present the sketch of the proof using the Fishing principle (Theorem~\ref{thm:fish_princ}),
    and Lemmas \ref{lemma:lem1}, \ref{lemma:lem2}, \ref{lemma:lem3}.
    We choose the set $\Omega$ as follows
    $$
        \Omega=\left\{ (x,\vartheta,u) \in \mathbb{R}^3 ~\big|~ x=0, \vartheta\le 0,
        \vartheta^2+u^2\le R^2\right\},
    $$
    where $R$ is a sufficiently large positive number.
    Conditions \ref{fish_princ:cond1} and \ref{fish_princ:cond2}
    in Theorem~\ref{thm:fish_princ} are satisfied
    for any $s\in[0,1)$.

    Lemmas~\ref{lemma:lem1} and~\ref{lemma:lem2} imply that,
    for $s=0$ condition \ref{fish_princ:cond3} in Theorem~\ref{thm:fish_princ} holds,
    while Lemmas~\ref{lemma:lem1} and~\ref{lemma:lem3}
    imply that, for $s = s_1$ sufficiently close to $1$
    condition \ref{fish_princ:cond4} in Theorem~\ref{thm:fish_princ} is satisfied.

    Condition \ref{fish_princ:cond5} holds, since system \eqref{sys:lorenz_like} has the solution
    \[
        x(t) \equiv \vartheta(t) \equiv 0, \quad u(t) = u(0)\exp(-\alpha t),
    \]
    which satisfies
    \[
        \lim_{t \to -\infty} u(t) = \infty.
    \]
    Consequently, for large $|t|$, $t < 0$, the solutions with initial data from
    a small neighborhood of the point $x = \vartheta = 0$, $u = u_0$ leave the cylinder
    $\left\{ (x,\vartheta,u) \in \mathbb{R}^3 ~\big|~ \vartheta^2 + u^2 \leq R^2\right\}$,
    where $R$ is a sufficiently large positive number.
    Therefore, by Lemma~\ref{lemma:lem1}, condition~\ref{fish_princ:cond5}
    in Theorem~\ref{thm:fish_princ} holds.

    Hence, a path with $s \in [0, s_1]$ satisfies the conditions of Theorem~\ref{thm:fish_princ},
    and therefore there exists $s_0$,
    for which the assertion of the Theorem~\ref{thm:leonov} holds.
\end{proof}

\begin{corollary}
If $\beta(s) \in (0, 2]$ and conditions \eqref{cond:leonov:alpha_lambda}
for any $s \in [0,1]$ then there exists $s_0 \in (0,1)$
such that system \eqref{sys:lorenz_like} with
$\alpha(s_0)$, $\beta(s_0)$, $\lambda(s_0)$ has a homoclinic trajectory.
\end{corollary}
This result was obtained and proved previously in \cite{Leonov-2015-ND}.
\begin{corollary}\label{corollary:leonov}
Of particular interest to this study is the following path
\begin{equation}\label{eq:path-param}
    \lambda(s) = \frac{s}{\sqrt{1 - s}}, \qquad
    \alpha(s) = \delta \, \sqrt{1 - s}, \qquad \beta(s) \equiv \beta \in (0,~2 + \delta), \qquad
    s \in [0,~1), \qquad \delta \in (0,~1].
\end{equation}
This path satisfies all conditions of Theorem~\ref{thm:leonov}, and therefore
there exists a number $s_0 \in (0,~1)$ such that system \eqref{sys:lorenz_like}
with parameters \eqref{eq:path-param}
and $s = s_0$ has a homoclinic orbit.
\end{corollary}
In this case, conditions \eqref{eq:path-param} describe
the region of the parameters
$\mathcal{B}_{\delta,\beta} = \{~(\delta,\beta)~\big\vert~\delta \in (0,\,1.1],~\beta \in (0, 2 + \delta)\}$
in the parameter plane $(\delta, \beta)$ (see Fig.~\ref{fig:delta_beta_region}).
\begin{figure}[h!]
    \centering
    \includegraphics[width=0.5\textwidth]{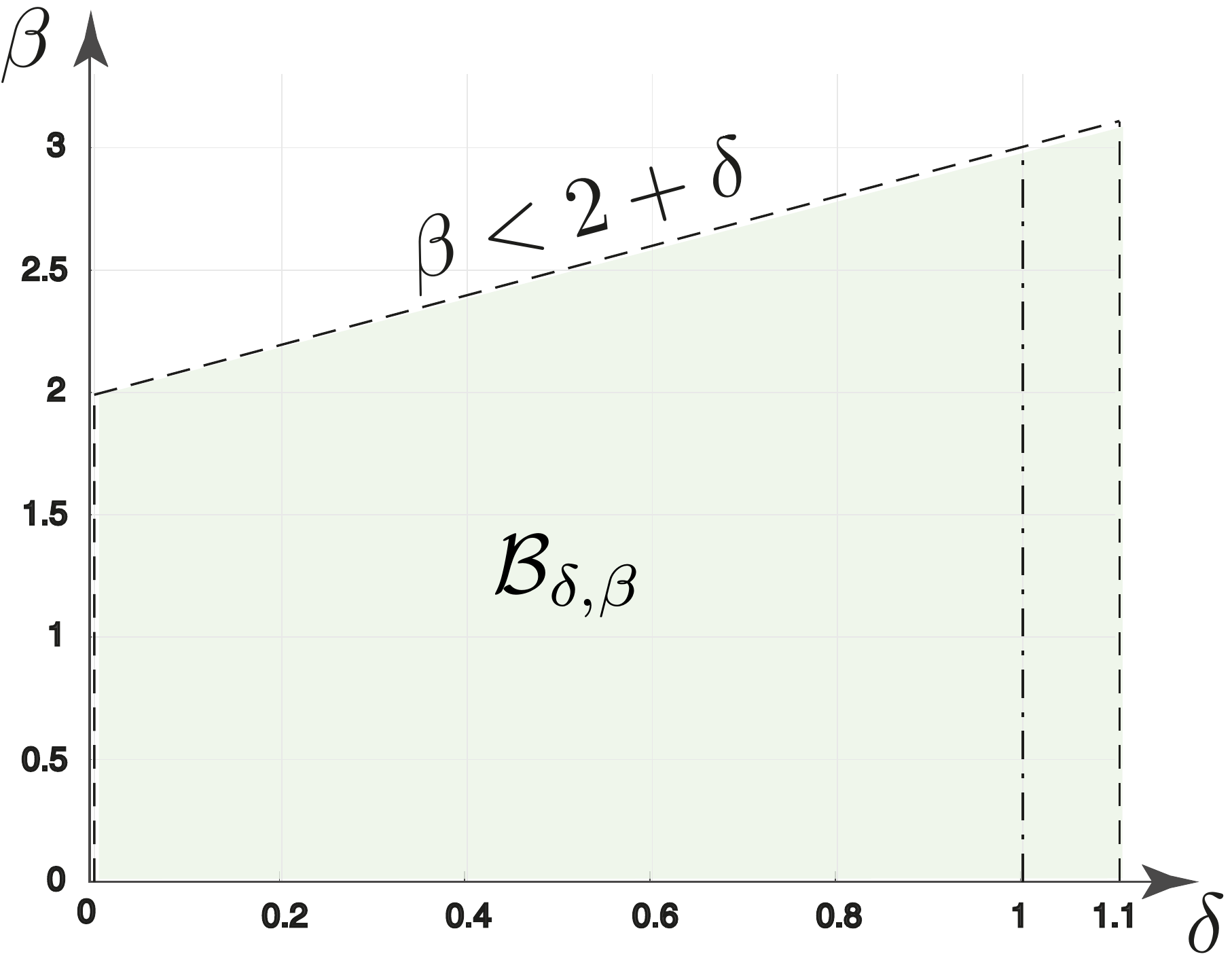}
    \caption{Region of parameters $\mathcal{B}_{\delta,\beta}$ (light green)
    in the plane $(\delta, \beta)$,
    for which there exists a homoclinic trajectory
    in system~\eqref{sys:lorenz_like}.}
  \label{fig:delta_beta_region}
\end{figure}
The eigenvalues and eigenvectors of the matrix of the linear part of \eqref{sys:lorenz_like}
at the saddle $S_0$ have the following form:
\begin{equation}
    \begin{aligned}
        & \lambda^{\rm s} = -\alpha = -\delta\sqrt{1-s},
        &\quad &{\rm v}^{\rm s} = (0,\ 0,\ 1), \\
        & \lambda^{\rm ss} = \tfrac{1}{2}\big(-\sqrt{\lambda^2 + 4} - \lambda\big) = - \tfrac{1}{\sqrt{1-s}},
        &\quad &{\rm v}^{\rm ss} =
        \big(-\!\sqrt{1-s},\ 1,\ 0\big), \\
        & \lambda^{\rm u} = \tfrac{1}{2}\big(\sqrt{\lambda^2 + 4} - \lambda\big) = \sqrt{1-s},
        &\quad &{\rm v}^{\rm u} = \big(\tfrac{1}{\sqrt{1-s}},\ 1,\ 0\big),
    \end{aligned}
\end{equation}
where ${\rm v}^{\rm s}$, ${\rm v}^{\rm ss}$, ${\rm v}^{\rm u}$ are
mutually perpendicular and the saddle value
$\sigma_0 = \lambda^{\rm u} + \lambda^{\rm s} = (1 - \delta) \sqrt{1-s} \geq 0$
is zero if $\delta = 1$ and positive if $\delta \in (0,1)$.
The equilibrium $S_0$ has stable and unstable local invariant manifolds
$W_{\rm loc}^{\rm s}$ и $W_{\rm loc}^{\rm u}$ of dimension
$\dim W_{\rm loc}^{\rm s} = 2$ and $\dim W_{\rm loc}^{\rm u} = 1$, respectively,
intersecting at $S_0$.

\begin{remark}
The result of Corollary~\ref{corollary:leonov}
for path \eqref{eq:path-param} with $\beta \in (0,2)$ and $\delta = 1$
was proved in \cite{Leonov-2015-DAN-Homo,Leonov-2015-ND} and repeated in \cite{OvsyannikovT-2017}
taking into account transformation \eqref{eq:transform:leonov_ovsyannikov}.
\end{remark}

All the homoclinic bifurcations considered in
\cite{Leonov-2012-PLA,Leonov-2012-DM,Leonov-2012-DRAN,Leonov-2013-DAN-ChenLu,Leonov-2013-IJBC,
Leonov-2015-DAN-Homo,Leonov-2015-ND}
are described by the following two scenarios:
either the change of attracting equilibria
for separatrices of the saddle,
or the merging of two stable limit cycles
into one stable limit cycle.
Here, using numerical simulations, we describe for $\delta < 1$
several new homoclinic bifurcations.

\section{Numerical analysis of stability/instability of homoclinic butterfly
in the Lorentz-like system}

Homoclinic bifurcation phenomena is related to the mathematical description of the transition to chaos
know in literature as Shilnikov chaos.
Numerical analysis and visualization of Shilnikov chaos is a difficult task,
since it requires the study of unstable structures that are sensitive to errors in numerical methods.

In this article, to analyze the scenarios of homoclinic bifurcations and
possible onset of chaotic behavior
we perform a simple numerical scanning of the parameter region
$\mathcal{B}_{\delta,\beta} = \{(\delta,\beta)\,\big\vert\,\delta \in (0,\,1.1],\,\beta \in (0, 2 + \delta)\}$
(see Fig.~\ref{fig:delta_beta_region})
and for a fixed pair $(\delta, \beta) \in \mathcal{B}_{\delta,\beta}$
we calculate the approximate interval $[\underline{s} \,,\overline{s}] \subset (0,1)$,
such that for $s_0 \in [\underline{s} \,,\overline{s}]$
there exist a homoclinic orbit.
We select the grid of points $B_{\rm grid} \subset \mathcal{B}_{\delta,\beta}$
with the predefined step $\delta_{\rm grid}$ and for each point
$(\delta_{\rm curr}, \beta_{\rm curr}) \in B_{\rm grid}$
we choose the partition
$0 < s_{\rm step}^0 < 2s_{\rm step}^0 < 3s_{\rm step}^0 < \dots, (N-1)s_{\rm step}^0 < 1$
of the interval $(0,1)$ with step $s_{\rm step}^0 = \tfrac{1}{N}$.
For the system \eqref{sys:lorenz_like} with parameters
$\delta_{\rm curr}, \beta_{\rm curr}, \lambda(s_{\rm curr}), \alpha(s_{\rm curr})$
we will simulate separatrix
$(x_{\rm sepa}(t), \vartheta_{\rm sepa}(t), u_{\rm sepa}(t))$
of the saddle $S_0$ of the system \eqref{sys:lorenz_like} by integrating them numerically on the
chosen time interval
$t \in [0, T_{\rm trans}]$ using the implementation of the numerical procedure
\texttt{ode45} for solving differential equations in {\rm MATLAB}.

To determine possible existence of limit sets (stable limit cycles and attractors) we will
also numerically integrate trajectories $x_{\rm lim}(t), \vartheta_{\rm lim}(t), u_{\rm lim}(t)$
with initial data $\left(x_{\rm lim}(0), \vartheta_{\rm lim}(0), u_{\rm lim}(0)\right) =
\left(x_{\rm sepa}(T_{\rm trans}), \vartheta_{\rm sepa}(T_{\rm trans}), u_{\rm sepa}(T_{\rm trans})\right)$
on the chosen interval $t \in [0, T_{\rm lim}]$.
Resulting trajectories $(x_{\rm sepa}(t), \vartheta_{\rm sepa}(t), u_{\rm sepa}(t))$ and
$(x_{\rm lim}(t), \vartheta_{\rm lim}(t), u_{\rm lim}(t))$
will be colored according to the color scale from blue to red,
corresponding to the integration time interval
(this will help us to determine the twisting~/~untwisting of the trajectory).
Note that due to symmetry of the system \eqref{sys:lorenz_like} it is enough to integrate only one separatrix
$\Gamma^+(t) = (x_{\rm sepa}(t), \vartheta_{\rm sepa}(t), u_{\rm sepa}(t))$ and the second one can be expressed as
$\Gamma^-(t) := (-x_{\rm sepa}(t), -\vartheta_{\rm sepa}(t), u_{\rm sepa}(t))$.
When equilibria $S_\pm$ are saddle-foci, we also simulate the separatrix of
the $S_+$ in the described above manner.

In numerical integration of trajectories via \texttt{ode45}
we use the event handler {\rm ODE Event Location}
and handle the following events:
\begin{itemize}[label=$\bullet$]
    \item {\it separatrix $\Gamma^+(t)$ tends to infinity}.
    For the values of parameter $s$ close to $0$ system \eqref{sys:lorenz_like}
    is not dissipative in the sense of Levinson, and
    the separatrix of the saddle $S_0$ will slowly untwist to ''infinity''.
    Therefore, the procedure will be terminated if the separatrix leaves the sphere with
    a big enough radius $R_{\rm inf}$.
    \item {\it separatrix $\Gamma^+(t)$ tends to equilibrium $S_+$ (or $\Gamma^-(t)$ to $S_-$),
    towards which it was released}.
    If for some $s = s_{\rm cr}$ the separatrix tends to nearest equilibrium state,
    then for $s > s_{\rm cr}$ there will be no other bifurcations.
    At this point it is possible to terminate the scanning by parameter $s$
    and start modeling of the next pair $(\delta, \beta)$.
    To examine the attraction to the equilibrium state it was detected the event of
    falling into its small neighborhood of the radius $\varepsilon_{\rm eq}$.
\end{itemize}


If for a fixed pair $(\delta, \beta)$ during the scanning of the interval $(0,1)$
with the step $s_{\rm step}^0$ there are two consecutive values
$\underline{s}, \overline{s} \in (0,1)$ such that
the behavior of the separatrices $\Gamma^\pm(t)$ changes as we go from the parameters
$\lambda(\underline{s})$, $\alpha(\underline{s})$ to the parameters
$\lambda(\overline{s})$, $\alpha(\overline{s})$,
then the segment $[\underline{s} \,,\overline{s}]$ is also
scanned with the $s_{\rm step}^1 = 0.1 s_{\rm step}^0$.
This gradual reduction in the partitioning step allows us to
find the boundary values $\underline{s}$, $\overline{s}$
which specify a bifurcation with a certain accuracy
$\overline{s} - \underline{s} > \varepsilon_{\rm threshold}$.
\begin{table}[h!]
  \tbl{Values of the parameters of the numerical procedure for scanning the region
  $\mathcal{B}_{\delta,\beta}$.}
  {\begin{tabular}{
    |>{\centering}m{1cm}<{\centering}|
    >{\centering}m{1cm}<{\centering}|
    >{\centering}m{1.5cm}<{\centering}|
    >{\centering}m{1cm}<{\centering}|
    >{\centering}m{1cm}<{\centering}|
    >{\centering}m{1cm}<{\centering}|
    >{\centering}m{1cm}<{\centering}|
    >{\centering}m{1.5cm}<{\centering}|
    >{\centering}m{1.5cm}<{\centering}|
    @{}m{0pt}@{}}
    \hline
    $\delta_{\rm grid}$ & $s_{\rm step}^0$ & $\varepsilon_{\rm threshold}$ & $T_{\rm trans}$ & $T_{\rm lim}$ &
    $R_{\rm inf}$ & $\varepsilon_{\rm eq}$ & {\rm RelTol} & {\rm AbsTol} &\\[8pt]
    \hline
    $10^{-2}$ & $10^{-3}$ & $10^{-12}$ & $4 \cdot 10^{3}$ & $10^{3}$ & $100$ & $10^{-1}$ &
    $10^{-15}$ & $10^{-15}$ &\\[8pt]
    \hline
    \end{tabular}}
  \label{table:fitts:sim:regimeTime}
\end{table}

\begin{figure}[h!]
    \centering
    \includegraphics[width=0.5\textwidth]{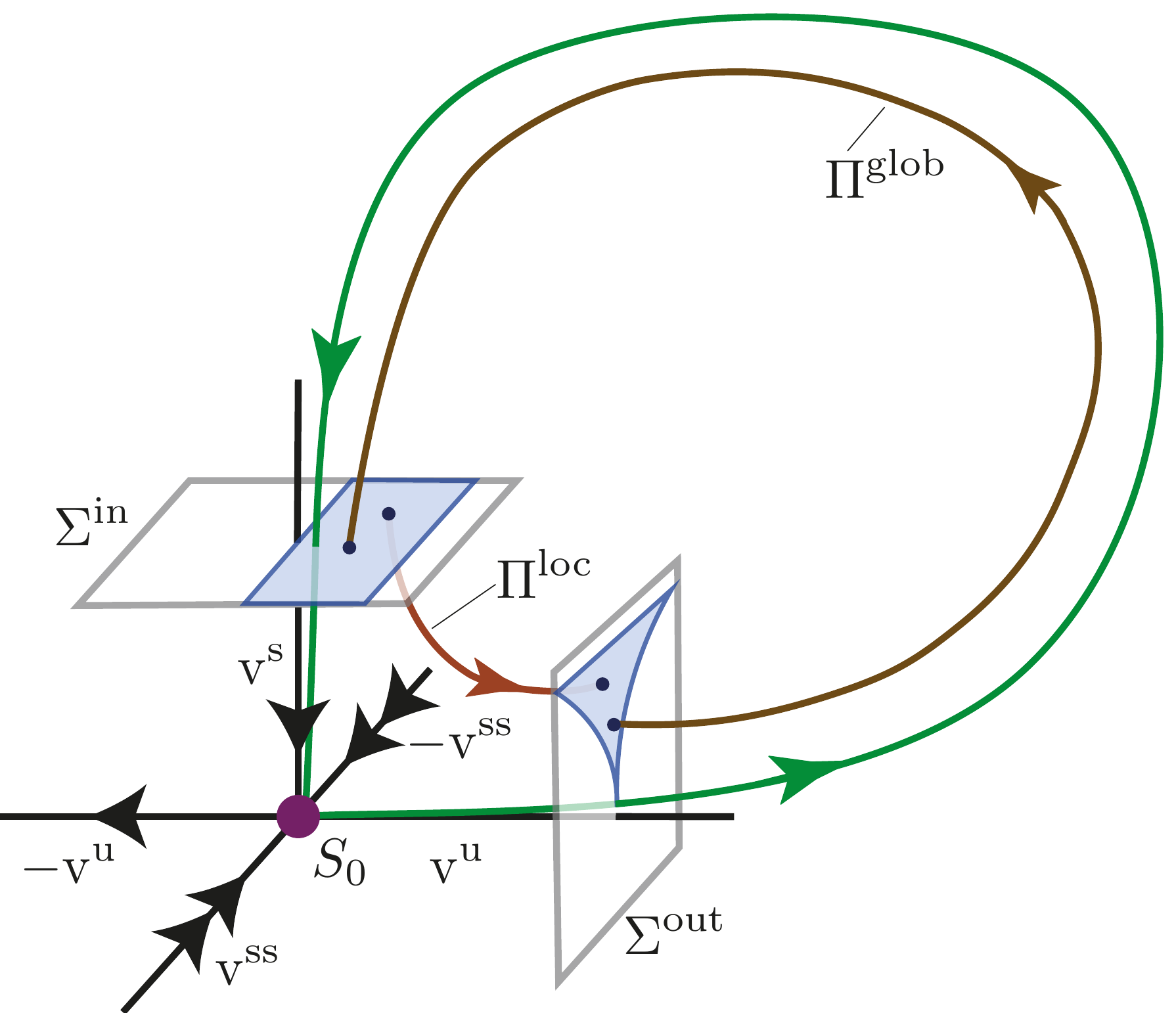}
    \caption{Poincar\'{e} sections $\Sigma^{\rm in}$ and $\Sigma^{\rm out}$
    in the neighborhood of the saddle $S_0 = (0,0,0)$ of the system \eqref{sys:lorenz_like}.}
  \label{fig:poincare_sec}
\end{figure}

After the described scanning of the region $\mathcal{B}_{\delta,\beta}$
for each grid point $(\delta, \beta) \in B_{\rm grid}$
the values $\underline{s}, \overline{s} \in (0,1)$ are found numerically,
such that the change of the parameter $s$ on the interval
$[\underline{s} \,,\overline{s}] \subset (0,1)$
specifies a homoclinic bifurcation.
Further, the type of homoclinic bifurcation was refined by numerical analysis of the behavior
of the Poincar\'{e} map on the corresponding sections $\Sigma^{\rm in}$, $\Sigma^{\rm out}$,
chosen in the neighborhood of the saddle $S_0$ (Fig.~\ref{fig:poincare_sec}).
The section $\Sigma^{\rm in}$ is chosen perpendicular to the vector ${\rm v^{s}}$
at a distance of $\varepsilon^{\rm in}$ from $S_0$,
the section $\Sigma^{\rm out}$ -- is perpendicular to the vector ${\rm v^{u}}$
and is located at a distance of $\varepsilon^{\rm out}$ from $S_0$.
On the section $\Sigma^{\rm in}$ a rectangular grid of points $\Sigma^{\rm in}_{\rm grid}$
with sides collinear to the vectors ${\rm v^{u}}$ and ${\rm v^{ss}}$ is chosen.
We match the color according to the color scale (from blue to red) to each row of grid points,
starting with the row that lies at the intersection of
$\Sigma^{\rm in}$ and the plane $\{{\rm v^{s}},{\rm v^{ss}}\}$, and paint the grid in this way
(Fig.~\ref{fig:poincare_sec_grid}).
Next, the evolution of the Poincare map
$\Pi = \Pi^{\rm glob} \circ \Pi^{\rm loc} : \Sigma^{\rm in} \to \Sigma^{\rm in}$
of the given grid of points $\Sigma^{\rm in}_{\rm grid}$ is numerically studied.
In our experiment, the maps $\Pi^{\rm loc} : \Sigma^{\rm in} \to \Sigma^{\rm out}$
and $\Pi^{\rm glob} : \Sigma^{\rm out} \to \Sigma^{\rm in}$
are simulated in {\rm MATLAB} using numerical procedure \texttt{ode45} and
built-in event handler {\rm ODE Event Location}
to determine the moment of hitting on the corresponding section.

As a result of numerical experiments it was found that
for a sufficiently small rectangle
after the first Poincar\'{e} map its image falls inside its domain.
Therefore, from the rectangular grid of points it is possible to
cut out the middle part and to consider the half frame in the experiment.
The size of the cutted-out part is chosen in such a way that
the intersection point of the separatrix $\Gamma^+(t)$ released from the saddle $S_0$
with the section $\Sigma^{\rm in}$ belongs to it along with its small neighborhood.
This approach allows us to avoid the simulation of
the separatrices close to the homoclinic loop
that requires calculations over large time intervals,
since during the refinement of the bifurcation parameter,
the separatrix becomes close to the homoclinic trajectory.
\begin{figure}[h!]
    \centering
    \subfloat[$\Pi^{\rm loc} : \Sigma^{\rm in} \to \Sigma^{\rm out}$
    maps rectangular grid of points $\Sigma^{\rm in}_{\rm grid}$
    to a ''half bowtie'' shape grid $\Pi^{\rm loc}\big(\Sigma^{\rm in}_{\rm grid}\big)$.
    ]{
        \includegraphics[width=0.45\textwidth]{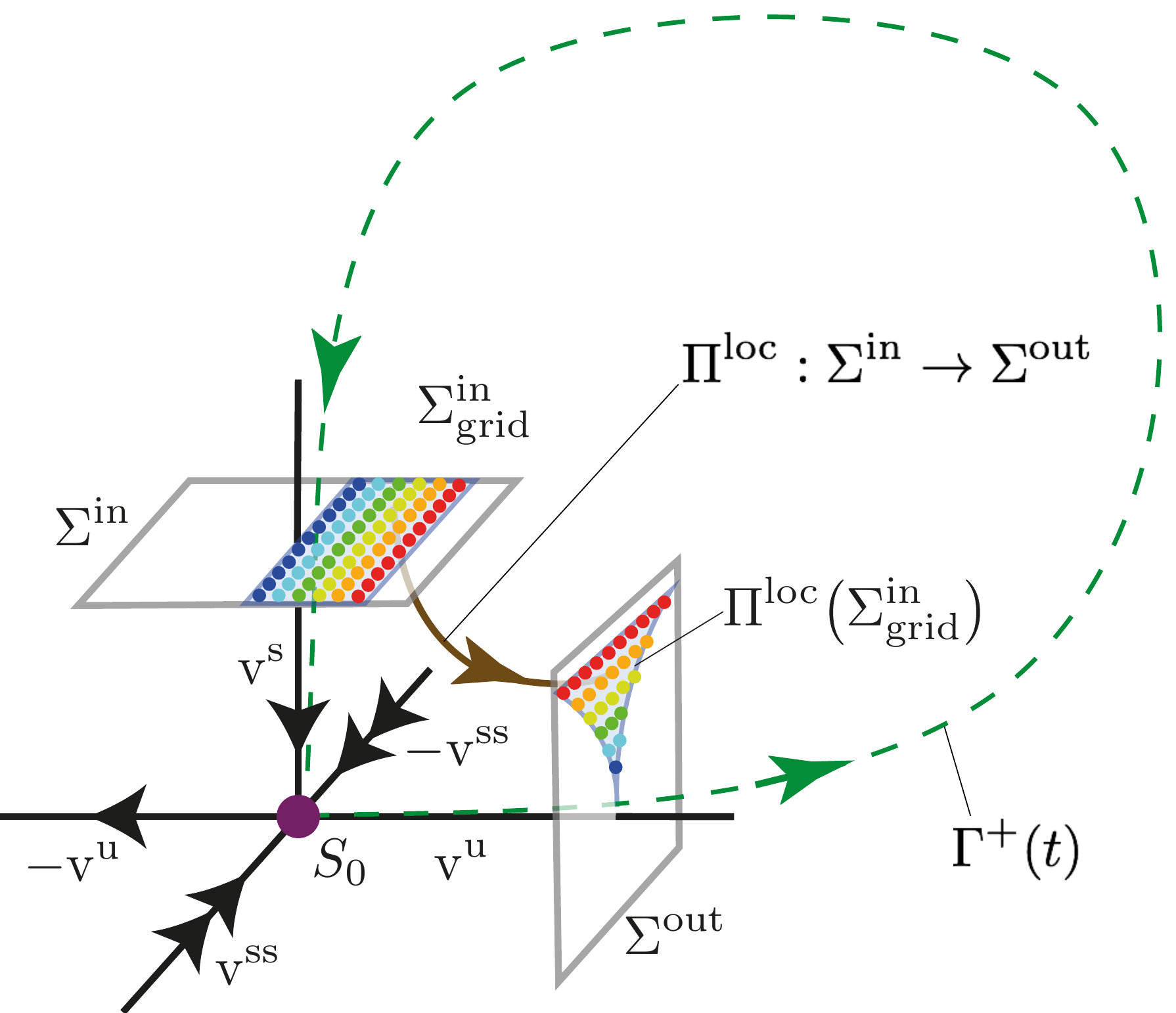}
    }\quad
    \subfloat[$\Pi^{\rm glob} : \Sigma^{\rm out} \to \Sigma^{\rm in}$
    shrinks $\Pi^{\rm loc}\big(\Sigma^{\rm in}_{\rm grid}\big)$
    into a ''stick'' shape grid $\Pi^{\rm glob}\big(\Sigma^{\rm in}_{\rm grid}\big)$.
    ]{
        \includegraphics[width=0.45\textwidth]{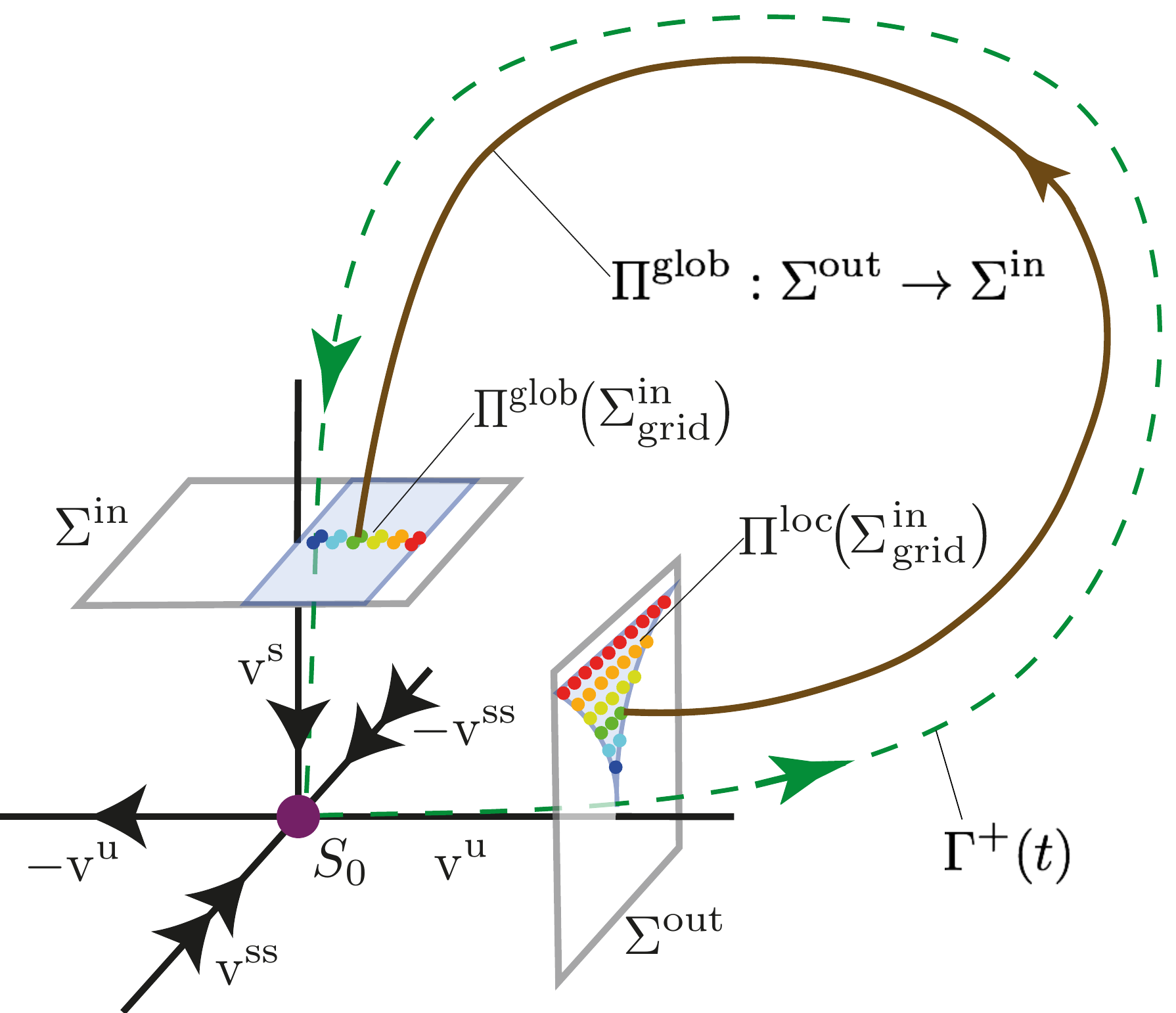}
    }
    \caption{Rectangular grid of points $\Sigma^{\rm in}_{\rm grid}$
    and its image $\Pi^{\rm glob}\big(\Sigma^{\rm in}_{\rm grid}\big)$
    on the Poincare section $\Sigma^{\rm in}$.}
  \label{fig:poincare_sec_grid}
\end{figure}

Numerical studies have shown that in the region covered by the given grid points,
there are $4$ regions with different homoclinic bifurcations (Fig.~\ref{fig:deltaBetaTrap}).
In the green region marked with ($\circ$) before bifurcation separatrices $\Gamma^{\pm}(t)$
were attracted to the opposite equilibria $S_\mp$ and after bifurcation
-- to the nearest ones, i.e. to $S_\pm$.
In this case, during the inverse bifurcation
(i.e. while moving in the direction from $s = 1$ to $s = 0$),
two unstable limit cycles are born from the homoclinic butterfly.
This scenario corresponds to the case of the homoclinic bifurcation in
classical Lorenz system \cite{Lorenz-1963} with parameters
$\sigma = 10$, $b = 8/3$, $r \in [13.9265574075, 13.9265574076]$
(see e.g. \cite{ShilnikovTCh-2001}).
\begin{figure}[h!]
    \centering
    \includegraphics[width=0.5\textwidth]{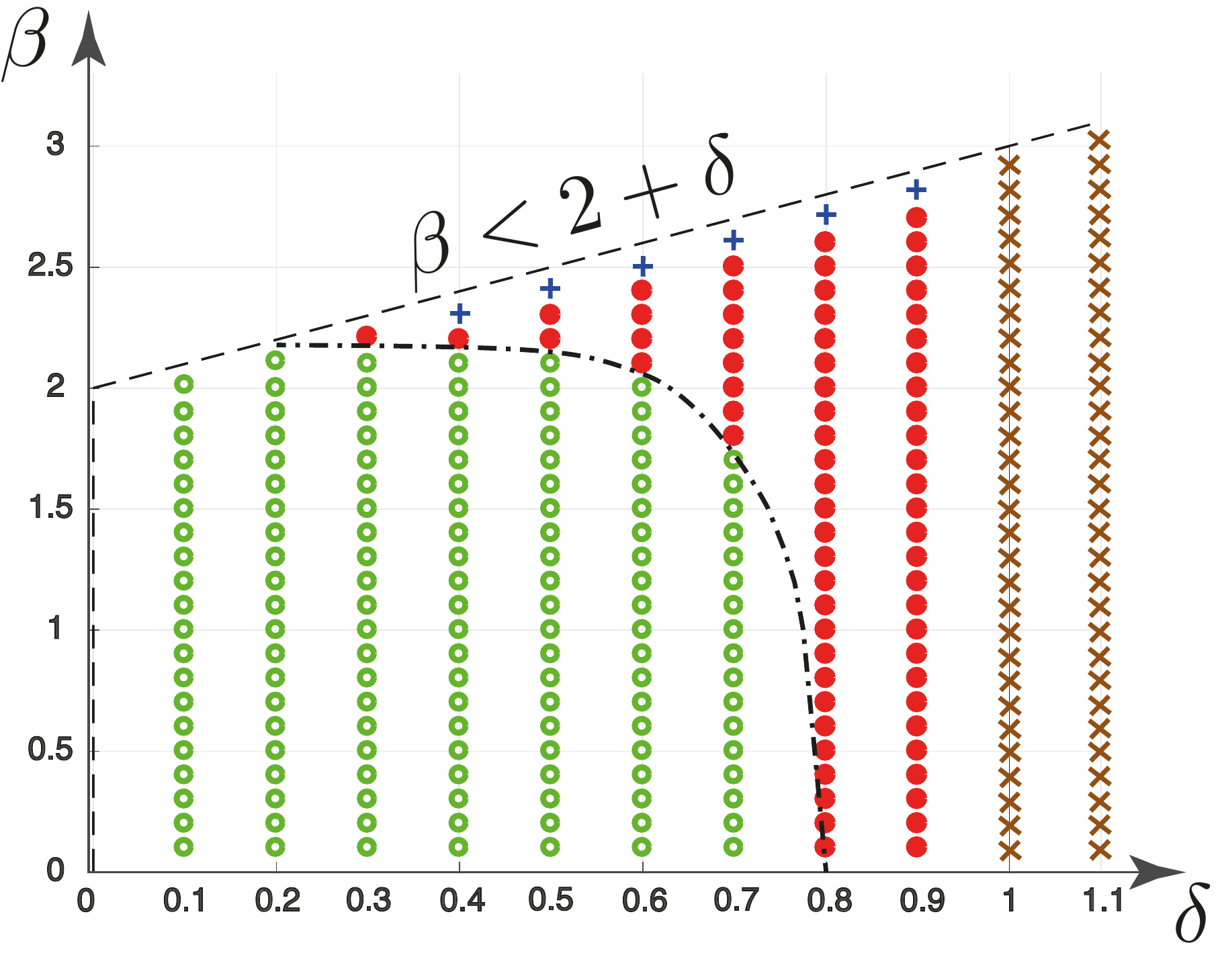}
    \caption{Different types of homoclinic bifurcations in the system \eqref{sys:lorenz_like}.}
  \label{fig:deltaBetaTrap}
\end{figure}

In the orange region marked with ($\times$) 
during the bifurcation, one large stable ''eight''-type limit cycle splits into
two stable limit cycles around $S_\pm$.
Numerical analysis of the separatrices behavior for all $\beta \in (0, 2+\delta)$
within the chosen partition and the dynamics analysis
of the grid points $\Sigma^{\rm in}_{\rm grid}$
on the Poincar\'{e} section $\Sigma^{\rm in}$ under the successive action of Poincar\'{e} map
$\Pi : \Sigma^{\rm in} \to \Sigma^{\rm in}$ give us a reason to state
that there is no chaotic dynamics in the vicinity of the homoclinic bifurcation
in the case of zero and negative saddle values.

\begin{figure}[h!]
 \centering
 \subfloat[$\text{\underline{s}} = 0.060131460578$
 ]{
    \includegraphics[width=0.5\textwidth]{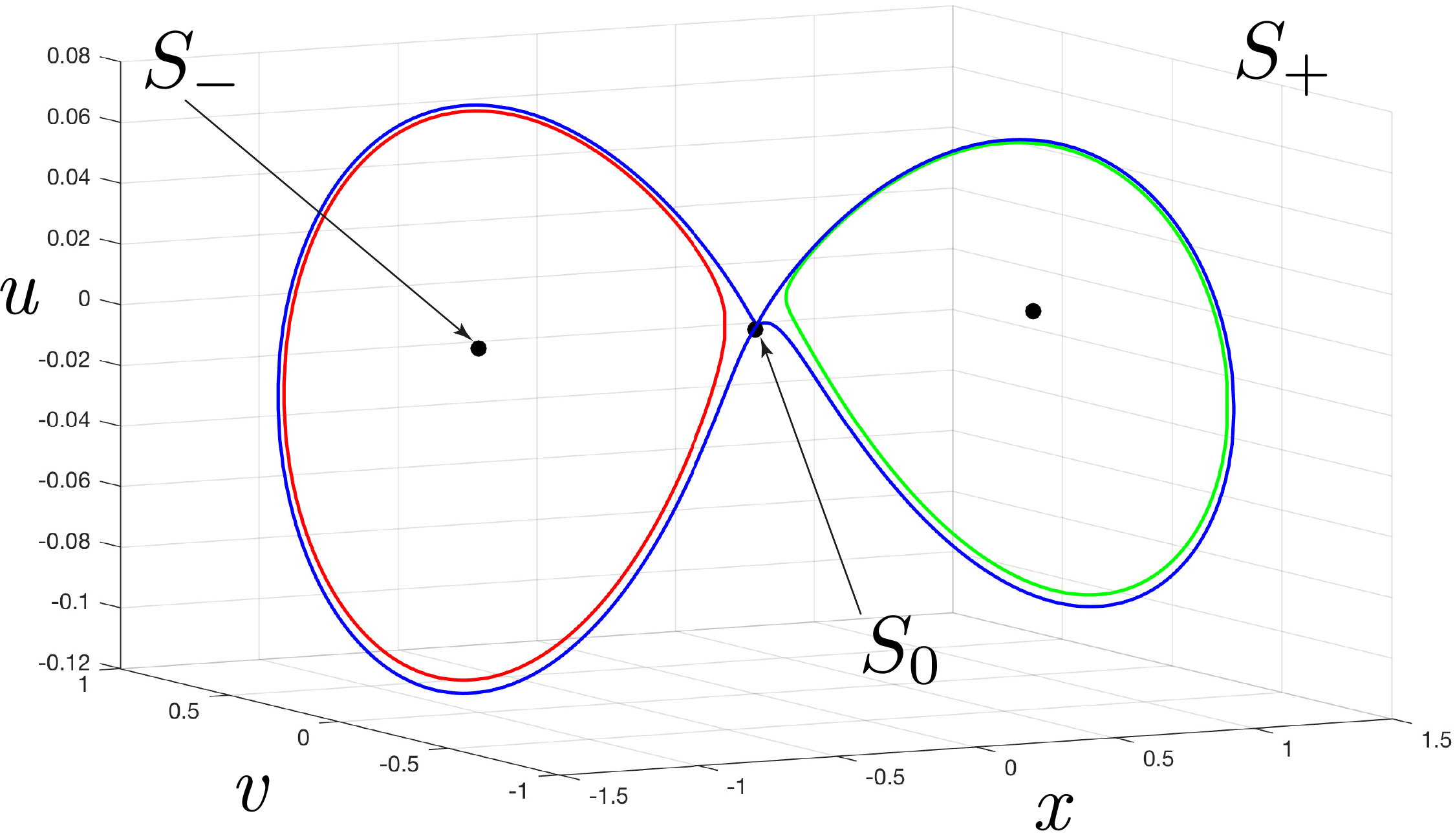}
  }
  \subfloat[$\overline{s} = 0.060131460581$
  ]{
    \includegraphics[width=0.5\textwidth]{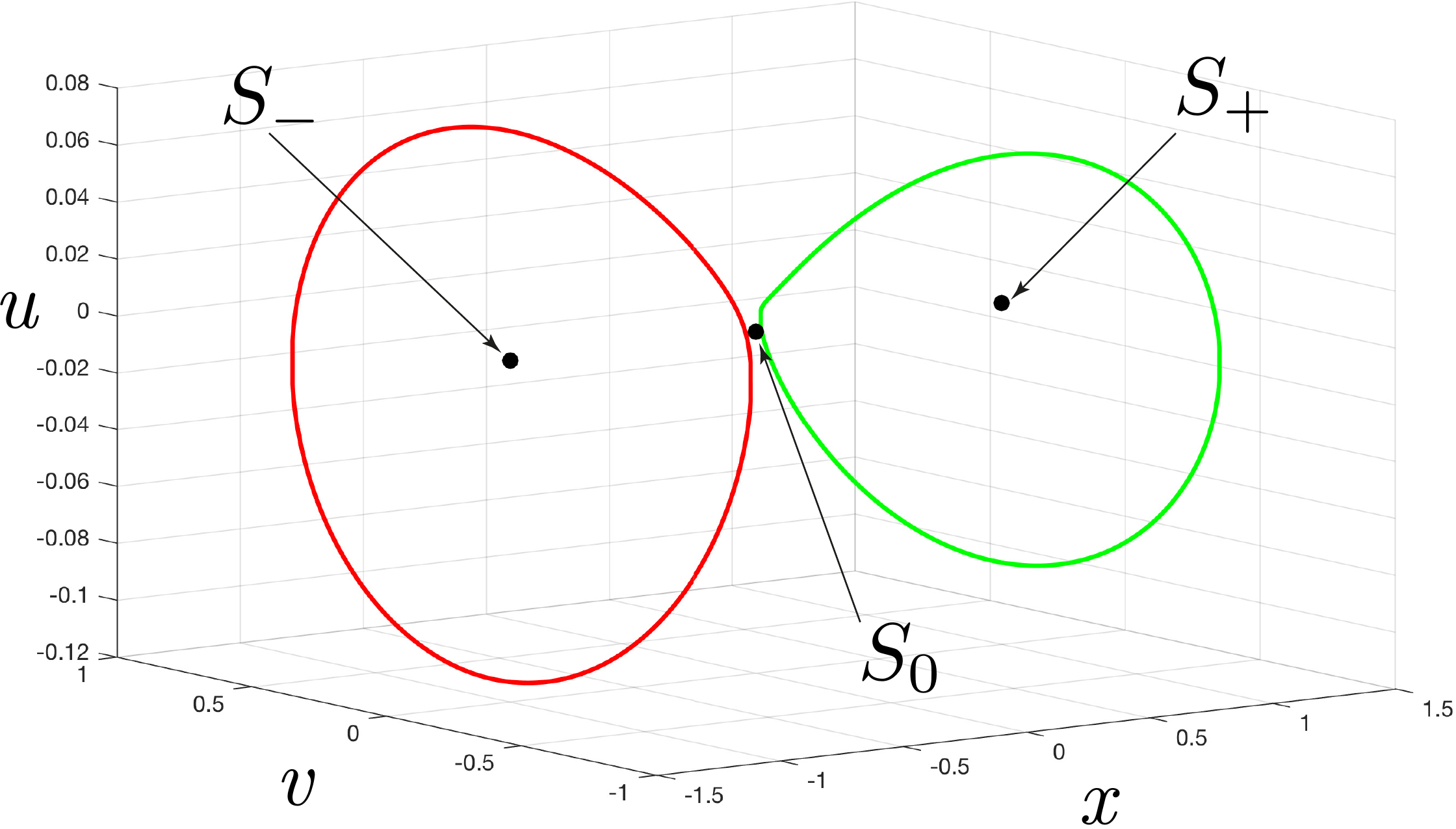}
  }
 \caption{Homoclinic bifurcation for $\delta = 0.9$, $\beta = 0.2$.}
 \label{fig:lorenz_like:lc8t_2lc}
\end{figure}
\begin{figure}[h!]
 \centering
 \subfloat[$\text{\underline{s}}' = 0.7979407438278198$
 ]{
    \includegraphics[width=0.5\textwidth]{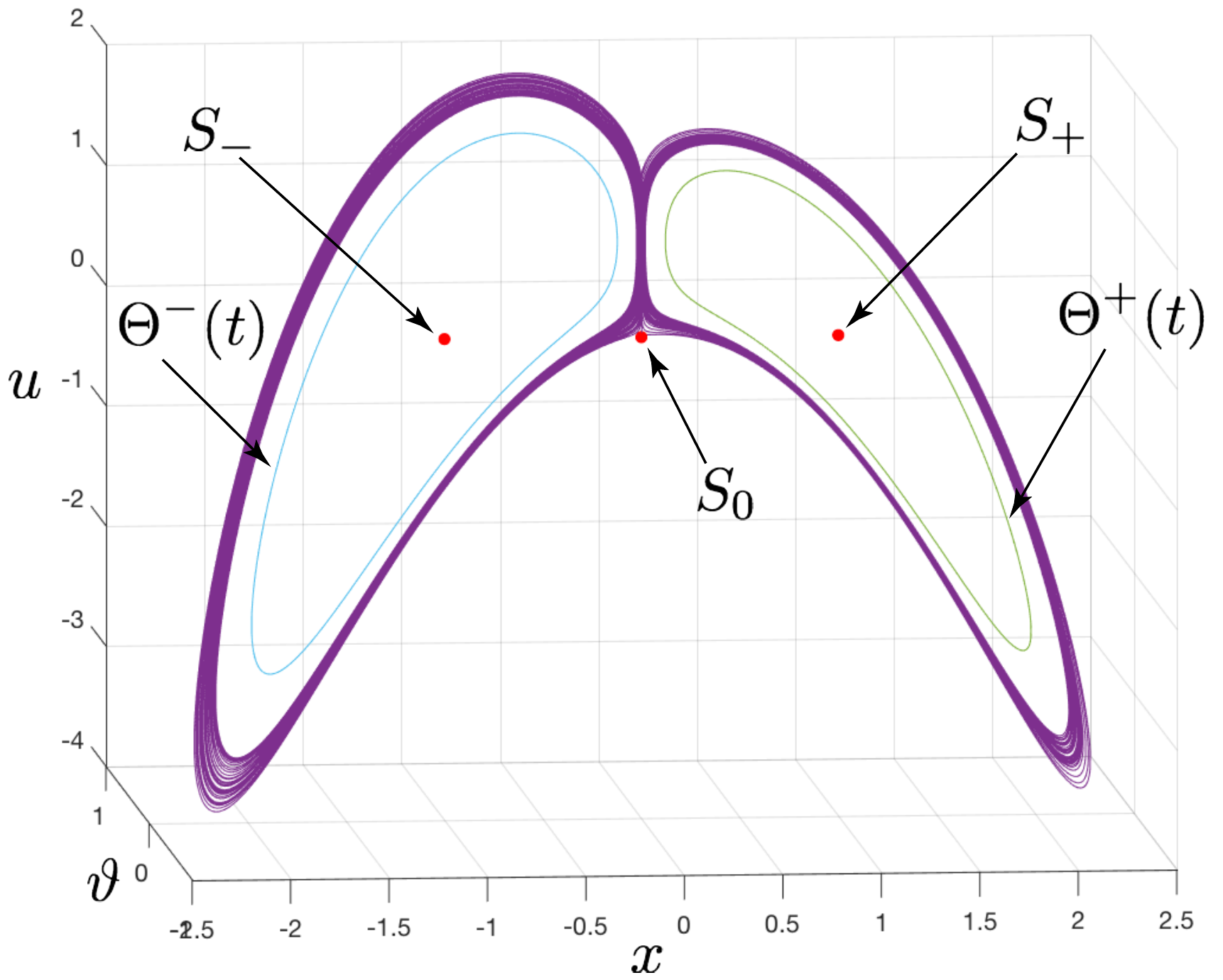}
  }
  \subfloat[$\text{\underline{s}} \in {[0.7979407447814941,0.8059291805341841]}$
  ]{
    \includegraphics[width=0.5\textwidth]{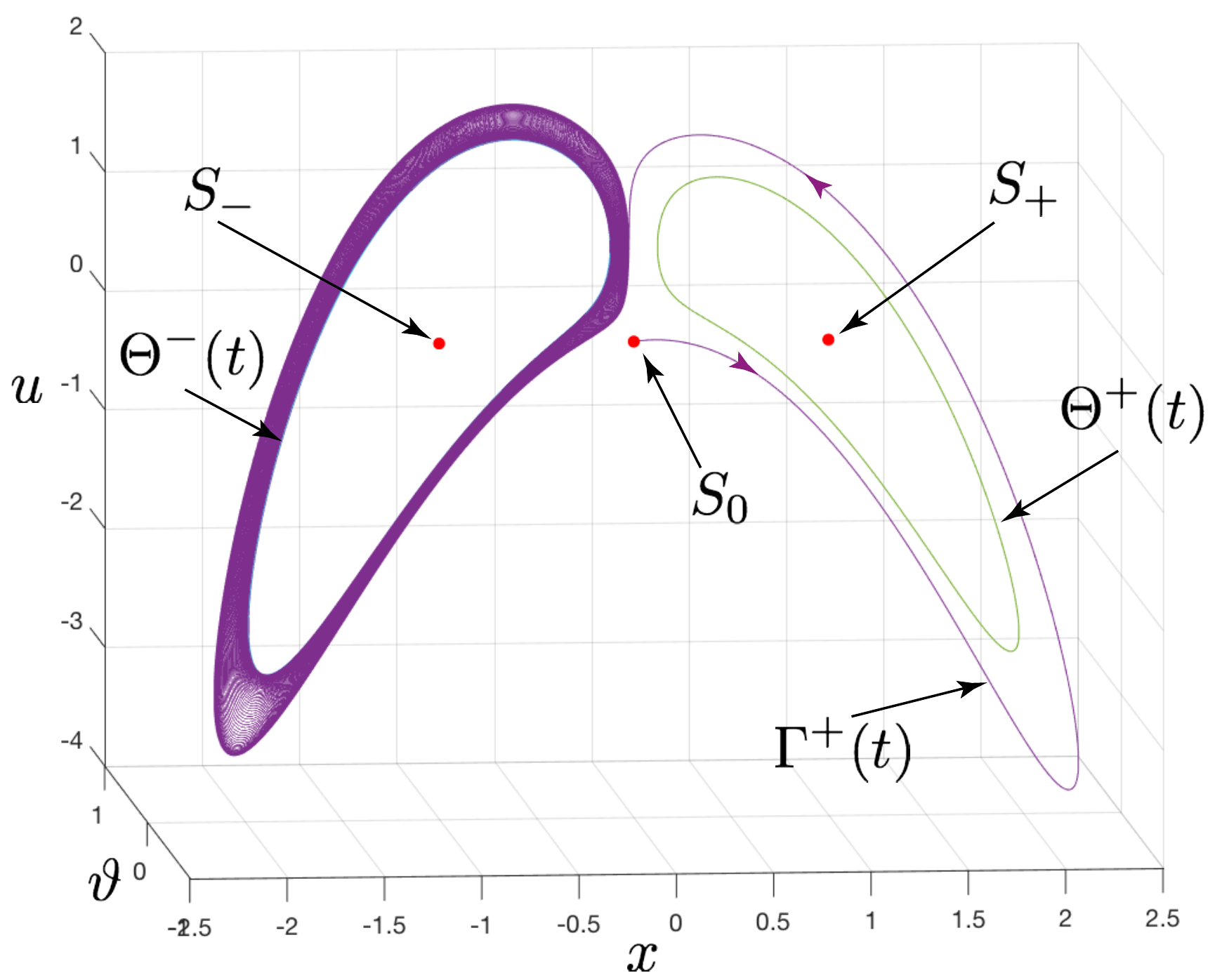}
  }
 \caption{Behavior of separatrix $\Gamma^+(t)$ of saddle $S_0$ and separatrices of
 saddle-foci $S_\pm$ before homoclinic bifurcation for $\delta = 0.5$, $\beta = 2.2$.
 Homoclinic bifurcation occurs on the interval $s \in [\underline{s}, \overline{s}]$,
 where $\overline{s} = 0.8059291805416346$.}
 \label{fig:lorenz_like:ch_2lc}
\end{figure}

Also, two new scenarios of homoclinic bifurcation were found.
In the red area marked with ($\bullet$), depending on values of parameters $\delta$, $\beta$,
two symmetric limit cycles $\Theta^{\pm}$ around $S_\pm$
coexist with either one stable ''eight''-type limit cycle, or a strange attractor
which attract the separatrices $\Gamma^{\pm}(t)$.
Then this attractor (periodic or strange) loses stability and separatrices $\Gamma^{\pm}(t)$
are attracted to the opposite limit cycles $\Theta^{\mp}$.
After the bifurcation the separatrices $\Gamma^{\pm}(t)$
are attracted to the nearest limit cycles $\Theta^{\pm}$.
As in the case of Lorenz system, in this case during the inverse bifurcation,
two unstable limit cycles are born from the homoclinic butterfly,
but here they separetes two stable cycles $\Theta^{\pm}$.
For example, for parameter values $\delta = 0.9$, $\beta = 0.2$,
the dynamics of separatrices in the phase space is shown in
Fig. \ref{fig:lorenz_like:lc8t_2lc}
and the dynamics of the grid of points $\Sigma^{\rm in}_{\rm grid}$
before and after bifurcation is presented in
Fig.~\ref{fig:lorenz_like:lc8t_2lc:PM_before} and
Fig.~\ref{fig:lorenz_like:lc8t_2lc:PM_after}), respectively.
For parameter values $\delta = 0.5$, $\beta = 2.2$ the case of coexistence
of two symmetric limit cycles $\Theta^{\pm}$ around $S_\pm$ with
a strange attractor defined by the separatrix $\Gamma^+(t)$
is presented in Fig.~\ref{fig:lorenz_like:ch_2lc}.
Note that here one could consider two types of vicinities
of the bifurcation point in the parameter space:
$[\underline{s},\overline{s}]$ and $[\underline{s}',\overline{s}]$, where
$[\underline{s},\overline{s}] \subset [\underline{s}',\overline{s}]$.
In the vicinity $[\underline{s},\overline{s}]$ a simple bifurcation
is observed in which, as described just above, there is a change in attracting limit cycles $\Theta^{\pm}$
for the separatrices $\Gamma^{\pm}(t)$ of the saddle $S_0$.
At the same time, on the interval $[\underline{s}',\underline{s})$
the chaotic behavior of the separatrices $\Gamma^{\pm}(t)$ can be observed,
which can make one to think that the
homoclinic bifurcation is embedded in the strange attractor.
\begin{figure}[h!]
 \centering
 \subfloat[{$\overline{s} = 0.7957...$}]{
    \includegraphics[width=0.5\textwidth]{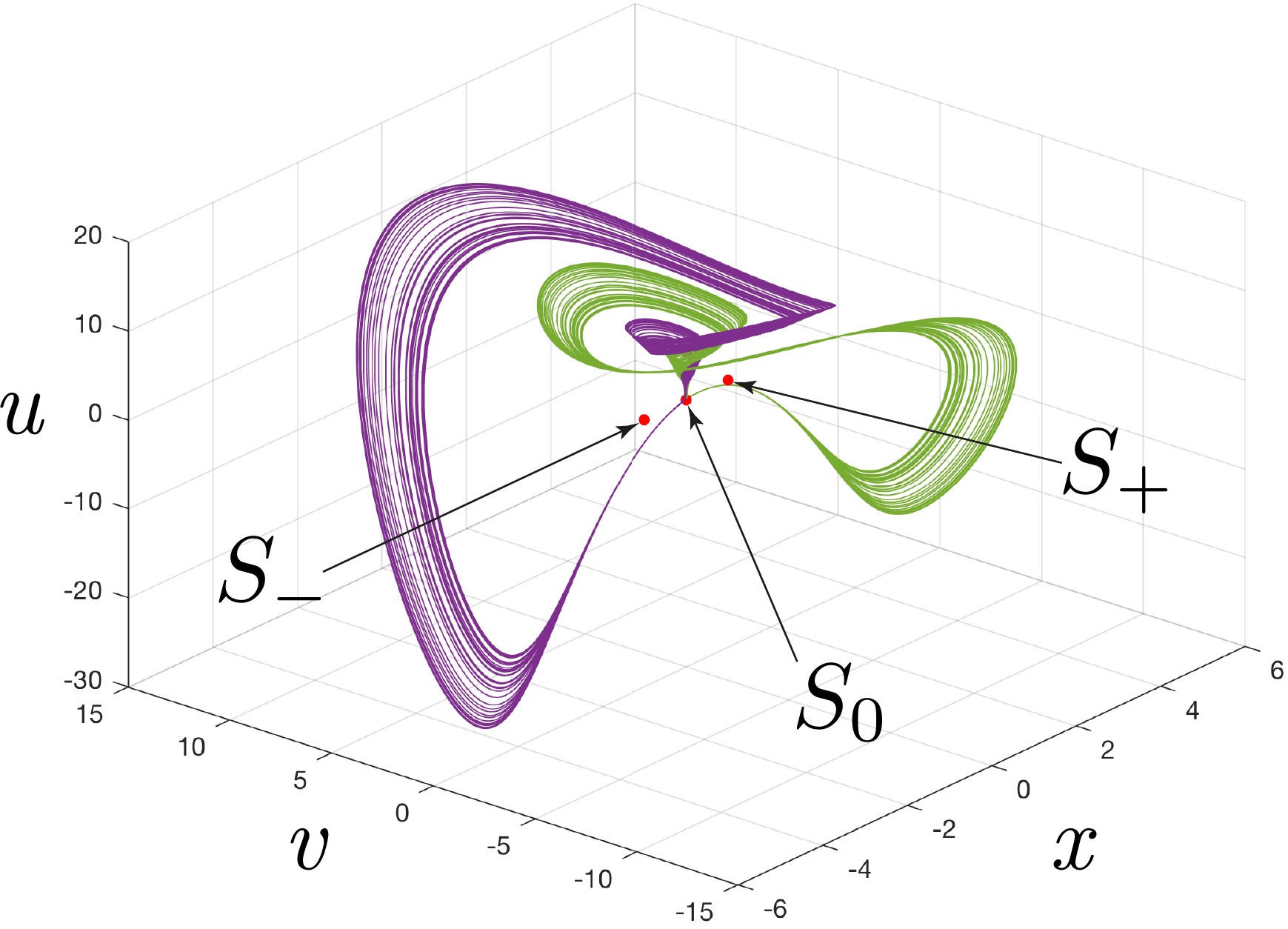}
  }~
  \subfloat[{$\text{\underline{s}} = 0.7955...$}]{
    \includegraphics[width=0.5\textwidth]{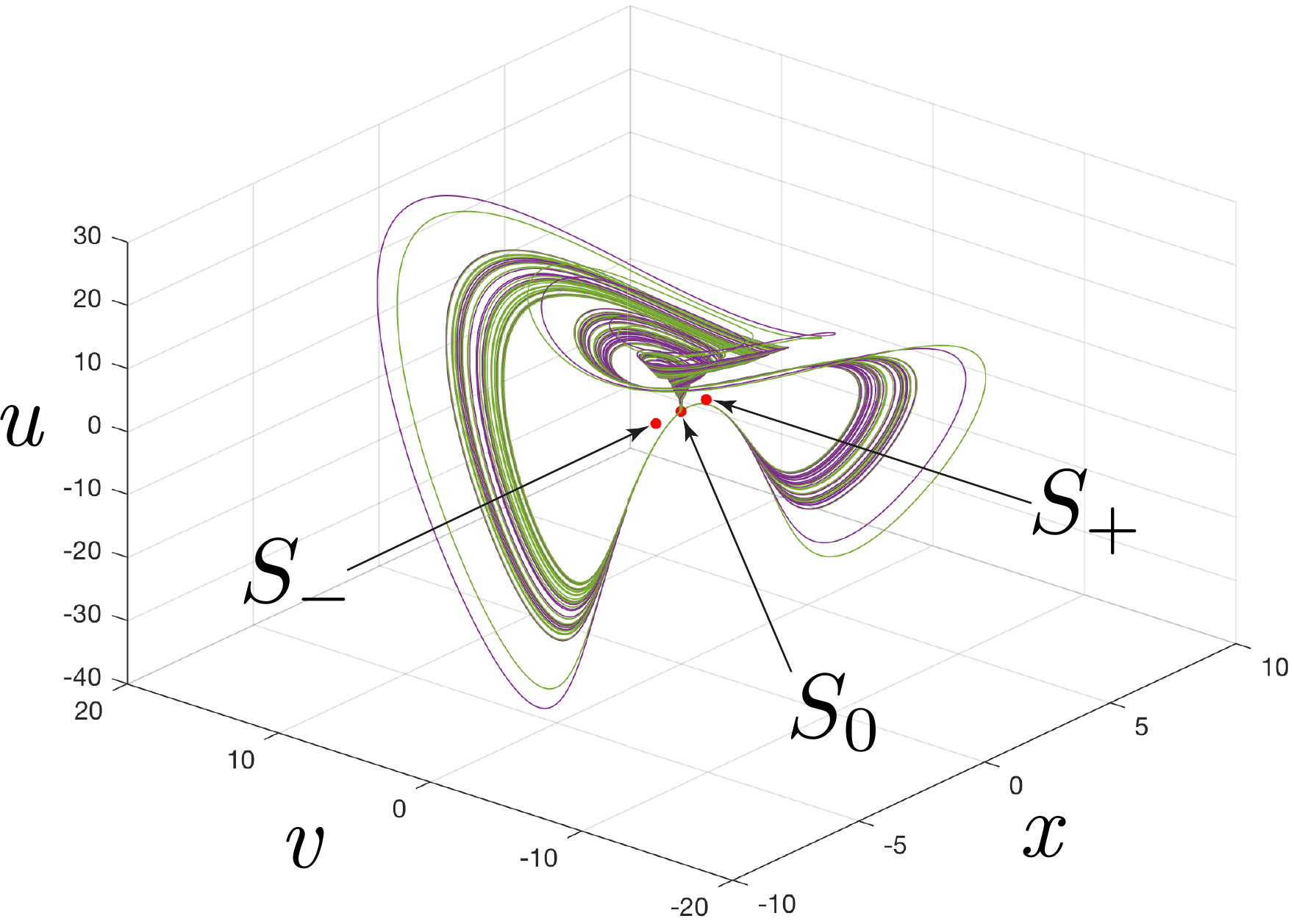}
  }
 \caption{Homoclinic bifurcation of the two merging attractors at $\delta = 0.9$, $\beta = 2.899$.}
 \label{fig:lorenz_like:2chaos_merge}
\end{figure}

In the blue region marked with ($\ast$)
when an unstable homoclinic trajectory occurs,
one strange attractor split into two
(or, if we track the change in the parameter $s$ from $1$ to $0$, then
we can say that two strange attractors merge into one strange attractor).
For example, for parameter values $\delta = 0.9$, $\beta = 2.899$,
the dynamics of separatrices in the phase space is shown in
Fig. \ref{fig:lorenz_like:2chaos_merge} and
the dynamics of the grid of points $\Sigma^{\rm in}_{\rm grid}$
before and after bifurcation is presented in
Fig.~\ref{fig:lorenz_like:unstable_homo:PM_before} and
Fig.~\ref{fig:lorenz_like:unstable_homo:PM_after}), respectively.
\begin{figure}[h!]
    \centering
    \includegraphics[width=0.9\textwidth]{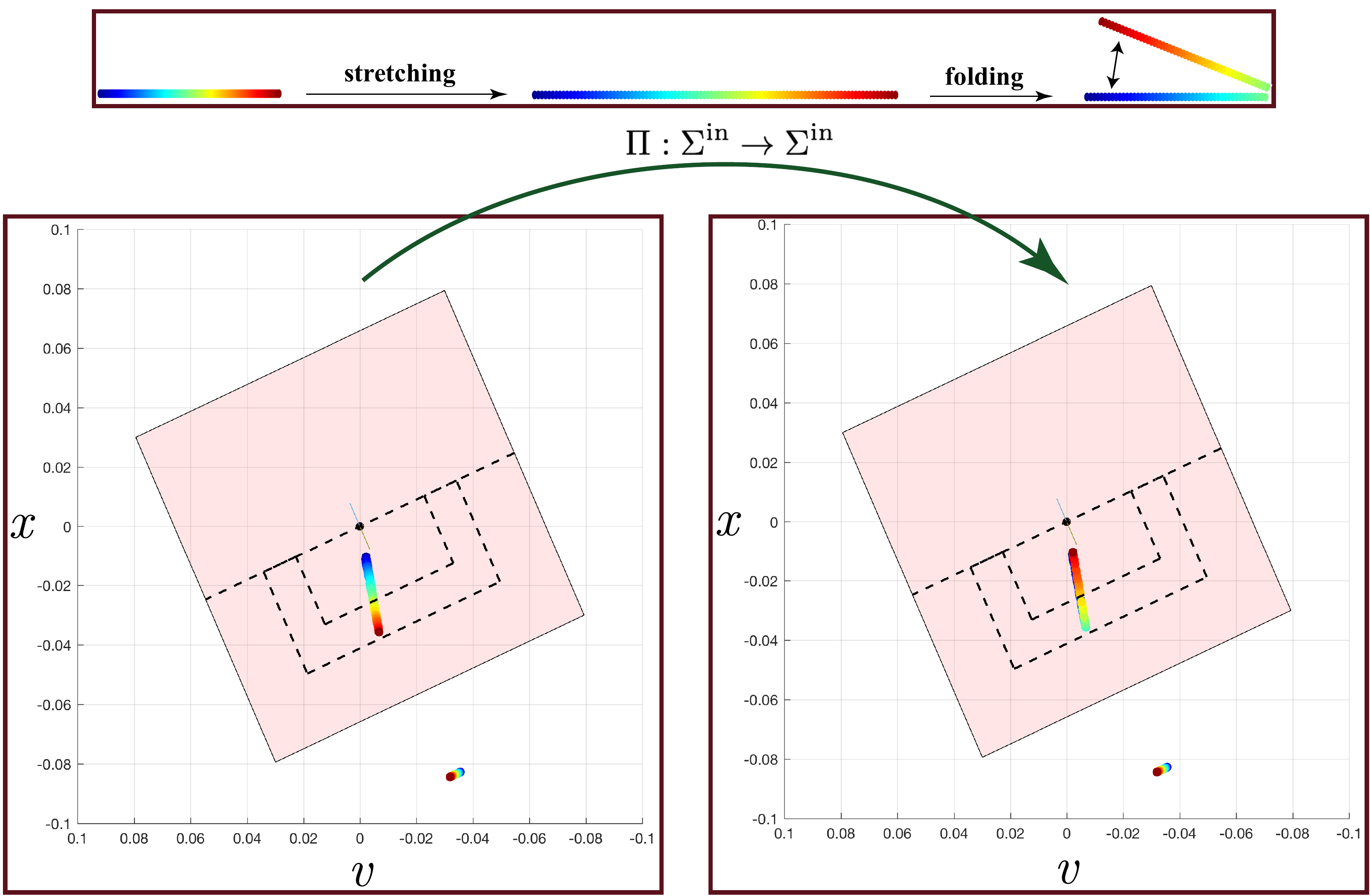}
    \caption{Behavior of the Poincar\'{e} map $\Pi : \Sigma^{\rm in} \to \Sigma^{\rm in}$
    in the case of two splitting attractors of system \eqref{sys:lorenz_like}
    with $\delta = 0.9$, $\beta = 2.899$.}
  \label{fig:testMix}
\end{figure}

For numerical verification of the behavior of the Poincar\'{e} map
$\Pi : \Sigma^{\rm in} \to \Sigma^{\rm in}$
for the case of splitting attractors we perform the following test.
Consider the grid of points $\Sigma^{\rm attr}_{\rm grid}$ corresponding to the
intersection between one of the attractors and the Poincar\'{e} section
$\Sigma^{\rm in}$ and color it according to the scale (from blue to red).
We save the coordinates of grid points assuming
that approximately this grid represents the line segment.
Next, we calculate the image of $\Sigma^{\rm attr}_{\rm grid}$
under the action of the Poincar\'{e} and after that for
each point ${\rm x}_0 \in \Sigma^{\rm attr}_{\rm grid}$ we
compare its coordinate with the coordinate of $\Pi({\rm x}_0)$.
As a result of this experiment, we have obtained that,
under the indicated assumptions, the Poincar\'{e} map
behaves approximately the same way as the know one-dimensional
tent map with parameter $\approx 2$ (Fig.~\ref{fig:testMix}).
Using special methods for finite-time Lyapunov exponents and
finite-time Lyapunov dimension
estimations
(see e.g.~\cite{LeonovKM-2015-EPJST,LeonovKM-2015-CNSNS,LeonovKKK-2016-CNSCS,KuznetsovLMPS-2018}),
we calculate the corresponding
values of the largest finite-time Lyapunov exponent, ${\rm LE}_1(t_{\rm end} {\rm x}_0) = 0.0316 > 0$,
and local finite-time Lyapunov dimension, ${\rm LD}(t_{\rm end}, {\rm x}_0) = 2.0131$
for one of the attractors along the trajectory\footnote{
    Remark that in numerical studying of long-term behavior of trajectories
    of nonlinear systems one usually could face the following problems.
    On the one hand, the result of numerical integration of trajectories via
    approximate methods is strongly influenced by round-off errors
    in the general case accumulate over a large time interval and do not
    allow tracking the ''true'' trajectory without the use of special methods
    and approaches~\cite{Tucker-1999,LiaoW-2014,KehletL-2017}.
    On the other hand, the problem arises of distinguishing between the established behavior
    defined by the {\it sustain} limit sets (periodic orbits, strange attractors)
    from the so-called {\it transient behavior} corresponding to
    a {\it transient set} in phase space, which nevertheless can exist for a long time
    \cite{GrebogiOY-1983,LaiT-2011,ChenKLM-2017-IJBC,KuznetsovLMPS-2018}.
}
with initial data ${\rm x}_0 = (-0.0479075467563750, 8.41428910156156, 13.7220943173008)$ and
time interval $[0,t_{\rm end}]$, $t_{\rm end} = 1000$.
These numerical experiments give us a reason to think that the
considered attractor (and the symmetric one) is strange.

Numerical simulations of separatrices outside the region $\mathcal{B}_{\delta,\beta}$
(i.e. for the case $\beta > 2 + \delta$) show that system \eqref{sys:lorenz_like}
in this region is not dissipative in the sense of Levinson and
separatrices tend to infinity.
Thus, numerically we obtain that outside the region $\mathcal{B}_{\delta,\beta}$
there are no homoclinic bifurcations.
Later on we are going to prove it analytically.

This article is the beginning of a study of this type of homoclinic bifurcation.
Further studies in this direction may require the introduction of
the new mathematical concepts into consideration, and
the development of the new numerical methods with a high performance computing.
Also, the authors plan to take into consideration
recently developed new reliable numerical methods for studying trajectories
of the Lorenz-like systems (see e.g.~\cite{Tucker-1999,LiaoW-2014,LoziP-2015,KehletL-2017})
and the existing approaches for the analysis of homoclinic
bifurcations (see e.g.~\cite{Wiggins-1988,ChampneysKS-1996,Doedel-2007-AUTO,HomburgS-2010}).

\section{Conclusion}

In papers \cite{Leonov-2012-PLA,Leonov-2013-IJBC,LeonovKM-2015-EPJST},
effective analytical and analytical-numerical methods
for studying homoclinic bifurcations and Shilnikov scenarios of system's behavior in its vicinity
were developed.
However, subsequent studies~\cite{LeonovM-2018-arVixV1,LeonovM-2018-PolyahovskieCh,Leonov-2018-UMZh}
have shown the practical difficulties of
the numerical implementation of these methods related to the
calculations with finite accuracy and round-off errors.
In this paper we have tried to overcome these difficulties as much as possible while
remaining within the framework of standard calculations in Matlab.
Thus, in this paper the following results were obtained.
We prove analytically the existence of homoclinic orbit to a saddle zero equilibrium in
the Lorenz-like system \eqref{sys:lorenz_like} 
and perform a numerical scanning of the corresponding parameter region
where the homoclinic bifurcations occur.
As a result, numerical confirmation of possible Shilnikov scenarios and
new scenarios of homoclinic bifurcations
in system \eqref{sys:lorenz_like} were found numerically,
e.g., the unstable homoclinic bifurcation of two merging strange attractors.

Numerical results on the merging attractors bifurcation
emphasize the fact that when studying homoclinic bifurcations it
is not sufficient to investigate the local behavior of the system in a neighborhood of
a saddle equilibrium.
The effects of stretching and compression in the vicinity of the homoclinic trajectory
could be influenced greatly by the global behavior of the system outside the saddle.

\section{Acknowledgment}
This publication was supported
by the grant of the President of Russian Federation
for the Leading Scientific Schools of Russia [NSh-2858.2018.1]
(namely, in Section 2: analytical proof of the existence of homoclinic orbits in system \eqref{sys:lorenz_like},
and in Section 3: initial numerical study of the region $\mathcal{B}_{\delta,\beta}$
of existence of homoclinic orbits), and
by the grant of the Russian Science Foundation [project 14-21-00041]
(namely, in Section 3: clarification of homoclinic bifurcation scenarios using standard Matlab framework,
numerical confirmation of possible Shilnikov scenarios, and
numerical analysis of chaotic attractors arising from the merge bifurcation).

\bibliographystyle{ws-ijbc}

\begin{thebibliography}{49}
\newcommand{\enquote}[1]{``#1''}
\providecommand{\natexlab}[1]{#1}
\providecommand{\url}[1]{\texttt{#1}}
\providecommand{\urlprefix}{URL }
\expandafter\ifx\csname urlstyle\endcsname\relax
  \providecommand{\doi}[1]{doi:\discretionary{}{}{}#1}\else
  \providecommand{\doi}{doi:\discretionary{}{}{}\begingroup
  \urlstyle{rm}\Url}\fi

\bibitem[{{Afraimovich} \emph{et~al.}(2014){Afraimovich}, {Gonchenko},
  {Lerman}, {Shilnikov} \& {Turaev}}]{AfraimovichGLShT-2014}
{Afraimovich}, V.~S., {Gonchenko}, S.~V., {Lerman}, L.~M., {Shilnikov}, A.~L.
  \& {Turaev}, D.~V. [2014] \enquote{Scientific heritage of {LP} {S}hilnikov,}
  \emph{Regular and Chaotic Dynamics} \textbf{19},  435--460.

\bibitem[{Argoul \emph{et~al.}(1987)Argoul, Arneodo \&
  Richetti}]{ArgoulAR-1987}
Argoul, F., Arneodo, A. \& Richetti, P. [1987] \enquote{Experimental evidence
  for homoclinic chaos in the belousov-zhabotinskii reaction,} \emph{Physics
  Letters A} \textbf{120},  269--275.

\bibitem[{Champneys(1998)}]{Champneys-1998}
Champneys, A. [1998] \enquote{Homoclinic orbits in reversible systems and their
  applications in mechanics, fluids and optics,} \emph{Physica D: Nonlinear
  Phenomena} \textbf{112},  158--186.

\bibitem[{Champneys \emph{et~al.}(1996)Champneys, Kuznetsov \&
  Sandstede}]{ChampneysKS-1996}
Champneys, A., Kuznetsov, Y. \& Sandstede, B. [1996] \enquote{A numerical
  toolbox for homoclinic bifurcation analysis,} \emph{International Journal of
  Bifurcation and Chaos} \textbf{6},  867--888.

\bibitem[{Chen \emph{et~al.}(2017)Chen, Kuznetsov, Leonov \&
  Mokaev}]{ChenKLM-2017-IJBC}
Chen, G., Kuznetsov, N., Leonov, G. \& Mokaev, T. [2017] \enquote{Hidden
  attractors on one path: {G}lukhovsky-{D}olzhansky, {L}orenz, and {R}abinovich
  systems,} \emph{International Journal of Bifurcation and Chaos} \textbf{27},
  {a}rt. num. 1750115.

\bibitem[{Chen \& Ueta(1999)}]{ChenU-1999}
Chen, G. \& Ueta, T. [1999] \enquote{Yet another chaotic attractor,}
  \emph{International Journal of Bifurcation and Chaos} \textbf{9},
  1465--1466.

\bibitem[{Doedel \& et. al(2007)}]{Doedel-2007-AUTO}
Doedel, E. \& et. al [2007] \enquote{{AUTO-07P}: {C}ontinuation and bifurcation
  software for ordinary differential equations,}
  \urlprefix\url{http://www.dam.brown.edu/people/sandsted/auto/auto07p.pdf}.

\bibitem[{Grebogi \emph{et~al.}(1983)Grebogi, Ott \& Yorke}]{GrebogiOY-1983}
Grebogi, C., Ott, E. \& Yorke, J. [1983] \enquote{Fractal basin boundaries,
  long-lived chaotic transients, and unstable-unstable pair bifurcation,}
  \emph{Physical Review Letters} \textbf{50},  935--938.

\bibitem[{Homburg \& Sandstede(2010)}]{HomburgS-2010}
Homburg, A.~J. \& Sandstede, B. [2010] \enquote{Homoclinic and heteroclinic
  bifurcations in vector fields,} \emph{Handbook of dynamical systems}
  \textbf{3},  379--524.

\bibitem[{Kehlet \& Logg(2017)}]{KehletL-2017}
Kehlet, B. \& Logg, A. [2017] \enquote{A posteriori error analysis of round-off
  errors in the numerical solution of ordinary differential equations,}
  \emph{Numerical Algorithms} \textbf{76},  191--210.

\bibitem[{Kuznetsov \emph{et~al.}(2018)Kuznetsov, Leonov, Mokaev, Prasad \&
  Shrimali}]{KuznetsovLMPS-2018}
Kuznetsov, N., Leonov, G., Mokaev, T., Prasad, A. \& Shrimali, M. [2018]
  \enquote{Finite-time {L}yapunov dimension and hidden attractor of the
  {R}abinovich system,} \emph{Nonlinear Dynamics} \textbf{92},  267--285,
  \doi{10.1007/s11071-018-4054-z}.

\bibitem[{Kuznetsov \emph{et~al.}(1992)Kuznetsov, Muratori \&
  Rinaldi}]{KuznetsovMR-1992}
Kuznetsov, Y., Muratori, S. \& Rinaldi, S. [1992] \enquote{Bifurcations and
  chaos in a periodic predator-prey model,} \emph{International Journal of
  Bifurcation and Chaos} \textbf{2},  117--128.

\bibitem[{Lai \& Tel(2011)}]{LaiT-2011}
Lai, Y. \& Tel, T. [2011] \emph{Transient Chaos: Complex Dynamics on Finite
  Time Scales} (Springer, New York).

\bibitem[{Leonov(2012{\natexlab{a}})}]{Leonov-2012-PLA}
Leonov, G. [2012{\natexlab{a}}] \enquote{General existence conditions of
  homoclinic trajectories in dissipative systems. {L}orenz,
  {S}himizu-{M}orioka, {L}u and {C}hen systems,} \emph{Physics Letters A}
  \textbf{376},  3045--3050.

\bibitem[{Leonov(2013{\natexlab{a}})}]{Leonov-2013-DAN-ChenLu}
Leonov, G. [2013{\natexlab{a}}] \enquote{Criteria for the existence of
  homoclinic orbits of systems {L}u and {C}hen,} \emph{Doklady Mathematics}
  \textbf{87},  220--223.

\bibitem[{Leonov(2013{\natexlab{b}})}]{Leonov-2013-IJBC}
Leonov, G. [2013{\natexlab{b}}] \enquote{Shilnikov chaos in {L}orenz-like
  systems,} \emph{International Journal of Bifurcation and Chaos} \textbf{23},
  \doi{10.1142/S0218127413500582}, art. num. 1350058.

\bibitem[{Leonov(2014{\natexlab{a}})}]{Leonov-2014-ND}
Leonov, G. [2014{\natexlab{a}}] \enquote{Fishing principle for homoclinic and
  heteroclinic trajectories,} \emph{Nonlinear Dynamics} \textbf{78},
  2751--2758.

\bibitem[{Leonov(2014{\natexlab{b}})}]{Leonov-2014-DAN}
Leonov, G. [2014{\natexlab{b}}] \enquote{R{\"o}ssler systems: estimates for the
  dimension of attractors and homoclinic orbits,} \emph{Doklady Mathematics}
  \textbf{89},  369--371.

\bibitem[{Leonov(2015)}]{Leonov-2015-DAN-Homo}
Leonov, G. [2015] \enquote{Cascade of bifurcations in {L}orenz-like systems:
  {B}irth of a strange attractor, blue sky catastrophe bifurcation, and nine
  homoclinic bifurcations,} \emph{Doklady Mathematics} \textbf{92},  563--567.

\bibitem[{Leonov(2016)}]{Leonov-2015-ND}
Leonov, G. [2016] \enquote{Necessary and sufficient conditions of the existence
  of homoclinic trajectories and cascade of bifurcations in {L}orenz-like
  systems: birth of strange attractor and 9 homoclinic bifurcations,}
  \emph{Nonlinear Dynamics} \textbf{84},  1055--1062.

\bibitem[{Leonov(2018)}]{Leonov-2018-UMZh}
Leonov, G. [2018] \enquote{Lyapunov functions in the global analysis of chaotic
  systems,} \emph{Ukrainian Mathematical Journal} \textbf{70},  42--66.

\bibitem[{Leonov \emph{et~al.}(2015{\natexlab{a}})Leonov, Alexeeva \&
  Kuznetsov}]{LeonovAK-2015}
Leonov, G., Alexeeva, T. \& Kuznetsov, N. [2015{\natexlab{a}}]
  \enquote{Analytic exact upper bound for the {L}yapunov dimension of the
  {S}himizu-{M}orioka system,} \emph{Entropy} \textbf{17},  5101--5116,
  \doi{10.3390/e17075101}.

\bibitem[{Leonov \emph{et~al.}(2017)Leonov, Andrievskiy \&
  Mokaev}]{LeonovAM-2017-VestSPSU}
Leonov, G., Andrievskiy, B. \& Mokaev, R. [2017] \enquote{Asymptotic behavior
  of solutions of {L}orenz-like systems: {A}nalytical results and computer
  error structures,} \emph{Vestnik St. Petersburg University. Mathematics}
  \textbf{50},  15--23.

\bibitem[{Leonov \& Kuznetsov(2015)}]{LeonovK-2015-AMC}
Leonov, G. \& Kuznetsov, N. [2015] \enquote{On differences and similarities in
  the analysis of {L}orenz, {C}hen, and {L}u systems,} \emph{Applied
  Mathematics and Computation} \textbf{256},  334--343,
  \doi{10.1016/j.amc.2014.12.132}.

\bibitem[{Leonov \emph{et~al.}(2016)Leonov, Kuznetsov, Korzhemanova \&
  Kusakin}]{LeonovKKK-2016-CNSCS}
Leonov, G., Kuznetsov, N., Korzhemanova, N. \& Kusakin, D. [2016]
  \enquote{{L}yapunov dimension formula for the global attractor of the
  {L}orenz system,} \emph{Communications in Nonlinear Science and Numerical
  Simulation} \textbf{41},  84--103, \doi{10.1016/j.cnsns.2016.04.032}.

\bibitem[{Leonov \emph{et~al.}(2015{\natexlab{b}})Leonov, Kuznetsov \&
  Mokaev}]{LeonovKM-2015-CNSNS}
Leonov, G., Kuznetsov, N. \& Mokaev, T. [2015{\natexlab{b}}] \enquote{Hidden
  attractor and homoclinic orbit in {L}orenz-like system describing convective
  fluid motion in rotating cavity,} \emph{Communications in Nonlinear Science
  and Numerical Simulation} \textbf{28},  166--174,
  \doi{10.1016/j.cnsns.2015.04.007}.

\bibitem[{Leonov \emph{et~al.}(2015{\natexlab{c}})Leonov, Kuznetsov \&
  Mokaev}]{LeonovKM-2015-EPJST}
Leonov, G., Kuznetsov, N. \& Mokaev, T. [2015{\natexlab{c}}]
  \enquote{Homoclinic orbits, and self-excited and hidden attractors in a
  {L}orenz-like system describing convective fluid motion,} \emph{The European
  Physical Journal Special Topics} \textbf{224},  1421--1458,
  \doi{10.1140/epjst/e2015-02470-3}.

\bibitem[{Leonov \& Mokaev(2018{\natexlab{a}})}]{LeonovM-2018-arVixV1}
Leonov, G. \& Mokaev, R. [2018{\natexlab{a}}] \enquote{Homoclinic bifurcations
  of the merging strange attractors in the {L}orenz-like system,} \emph{arXiv
  preprint arXiv:1802.07694v1} .

\bibitem[{Leonov \& Mokaev(2018{\natexlab{b}})}]{LeonovM-2018-PolyahovskieCh}
Leonov, G. \& Mokaev, R. [2018{\natexlab{b}}] \enquote{Numerical simulations of
  the {L}orenz-like system: Asymptotic behavior of solutions, chaos and
  homoclinic bifurcations,}  \emph{Abstracts of the International Scientific
  Conference on Mechanics ''The Eight Polyakhov’s Reading''}, p. 264.

\bibitem[{Leonov(2012{\natexlab{b}})}]{Leonov-2012-DM}
Leonov, G.~A. [2012{\natexlab{b}}] \enquote{Criteria for the existence of
  homoclinic orbits of systems {L}u and {C}hen,} \emph{Doklady Mathematics}
  \textbf{87},  220--223.

\bibitem[{Leonov(2012{\natexlab{c}})}]{Leonov-2012-DRAN}
Leonov, G.~A. [2012{\natexlab{c}}] \enquote{Tricomi problem for
  {S}himizu-{M}orioka dynamical system,} \emph{Doklady Mathematics}
  \textbf{86},  850--853.

\bibitem[{Liao \& Wang(2014)}]{LiaoW-2014}
Liao, S. \& Wang, P. [2014] \enquote{On the mathematically reliable long-term
  simulation of chaotic solutions of {L}orenz equation in the interval
  [0,10000],} \emph{Science China Physics, Mechanics and Astronomy}
  \textbf{57},  330--335.

\bibitem[{Lorenz(1963)}]{Lorenz-1963}
Lorenz, E. [1963] \enquote{Deterministic nonperiodic flow,} \emph{J. Atmos.
  Sci.} \textbf{20},  130--141.

\bibitem[{Lozi \& Pchelintsev(2015)}]{LoziP-2015}
Lozi, R. \& Pchelintsev, A. [2015] \enquote{A new reliable numerical method for
  computing chaotic solutions of dynamical systems: the {C}hen attractor case,}
  \emph{International Journal of Bifurcation and Chaos} \textbf{25},  1550187.

\bibitem[{Lu \& Chen(2002)}]{LuChen-2002}
Lu, J. \& Chen, G. [2002] \enquote{A new chaotic attractor coined,} \emph{Int.
  J. Bifurcation and Chaos} \textbf{12},  1789--1812.

\bibitem[{Neimark \& Landa(1992)}]{NeimarkL-1992}
Neimark, Y.~I. \& Landa, P.~S. [1992] \emph{Stochastic and Chaotic
  Oscillations} (Kluwer Academic Publishers, Dordrecht, The Netherlands).

\bibitem[{Oraevsky(1981)}]{Oraevsky-1981}
Oraevsky, A.~N. [1981] \enquote{Masers, lasers, and strange attractors,}
  \emph{Quantum Electronics} \textbf{11},  71--78.

\bibitem[{Ovsyannikov \& Turaev(2017)}]{OvsyannikovT-2017}
Ovsyannikov, I. \& Turaev, D. [2017] \enquote{Analytic proof of the existence
  of the {L}orenz attractor in the extended {L}orenz model,}
  \emph{Nonlinearity} \textbf{30},  115.

\bibitem[{Poincare(1892, 1893, 1899)}]{Poincare-1892}
Poincare, H. [1892, 1893, 1899] \emph{Les methodes nouvelles de la mecanique
  celeste. Vol. 1-3} (Gauthiers-Villars, Paris), [English transl. edited by D.
  Goroff: American Institute of Physics, NY, 1993].

\bibitem[{Rubenfeld \& Siegmann(1977)}]{RubenfeldS-1977}
Rubenfeld, L.~A. \& Siegmann, W.~L. [1977] \enquote{Nonlinear dynamic theory
  for a double-diffusive convection model,} \emph{SIAM Journal on Applied
  Mathematics} \textbf{32},  871--894.

\bibitem[{Shilnikov \emph{et~al.}(2001)Shilnikov, Shilnikov, Turaev \&
  Chua}]{ShilnikovTCh-2001}
Shilnikov, L., Shilnikov, A., Turaev, D. \& Chua, L. [2001] \emph{Methods of
  Qualitative Theory in Nonlinear Dynamics: {P}art {2}} (World Scientific).

\bibitem[{Shilnikov \emph{et~al.}(1998)Shilnikov, Shilnikov, Turaev \&
  Chua}]{ShilnikovTCh-1998}
Shilnikov, L.~P., Shilnikov, A.~L., Turaev, D.~V. \& Chua, L. [1998]
  \emph{Methods of Qualitative Theory in Nonlinear Dynamics: {P}art {1}} (World
  Scientific).

\bibitem[{Shimizu \& Morioka(1980)}]{Shimizu1980201}
Shimizu, T. \& Morioka, N. [1980] \enquote{On the bifurcation of a symmetric
  limit cycle to an asymmetric one in a simple model,} \emph{Physics Letters A}
  \textbf{76},  201 -- 204.

\bibitem[{Tel \& Gruiz(2006)}]{TelG-2006}
Tel, T. \& Gruiz, M. [2006] \emph{Chaotic dynamics: {A}n introduction based on
  classical mechanics} (Cambridge University Press).

\bibitem[{Tigan \& Opris(2008)}]{Tigan-2008}
Tigan, G. \& Opris, D. [2008] \enquote{Analysis of a 3{D} chaotic system,}
  \emph{Chaos, Solitons \& Fractals} \textbf{36},  1315--1319.

\bibitem[{Tricomi(1933)}]{Tricomi-1933}
Tricomi, F. [1933] \enquote{Integrazione di unequazione differenziale
  presentatasi in elettrotechnica,} \emph{Annali della R. Shcuola Normale
  Superiore di Pisa} \textbf{2},  1--20.

\bibitem[{Tucker(1999)}]{Tucker-1999}
Tucker, W. [1999] \enquote{The {L}orenz attractor exists,} \emph{Comptes Rendus
  de l'Academie des Sciences - Series I - Mathematics} \textbf{328},  1197 --
  1202.

\bibitem[{Wiggins(1988)}]{Wiggins-1988}
Wiggins, S. [1988] \emph{Global bifurcations and chaos: analytical methods},
  Vol.~73 (Springer-Verlag).

\bibitem[{Yang \& Chen(2008)}]{YangChen-2008}
Yang, Q. \& Chen, G. [2008] \enquote{A chaotic system with one saddle and two
  stable node-foci,} \emph{International Journal of Bifurcation and Chaos}
  \textbf{18},  1393--1414, \doi{10.1142/S0218127408021063}.

\end{thebibliography}

\begin{figure}[h!]
 \centering
 \subfloat[$i = 0$
 ]{
    \includegraphics[width=0.38\textwidth]{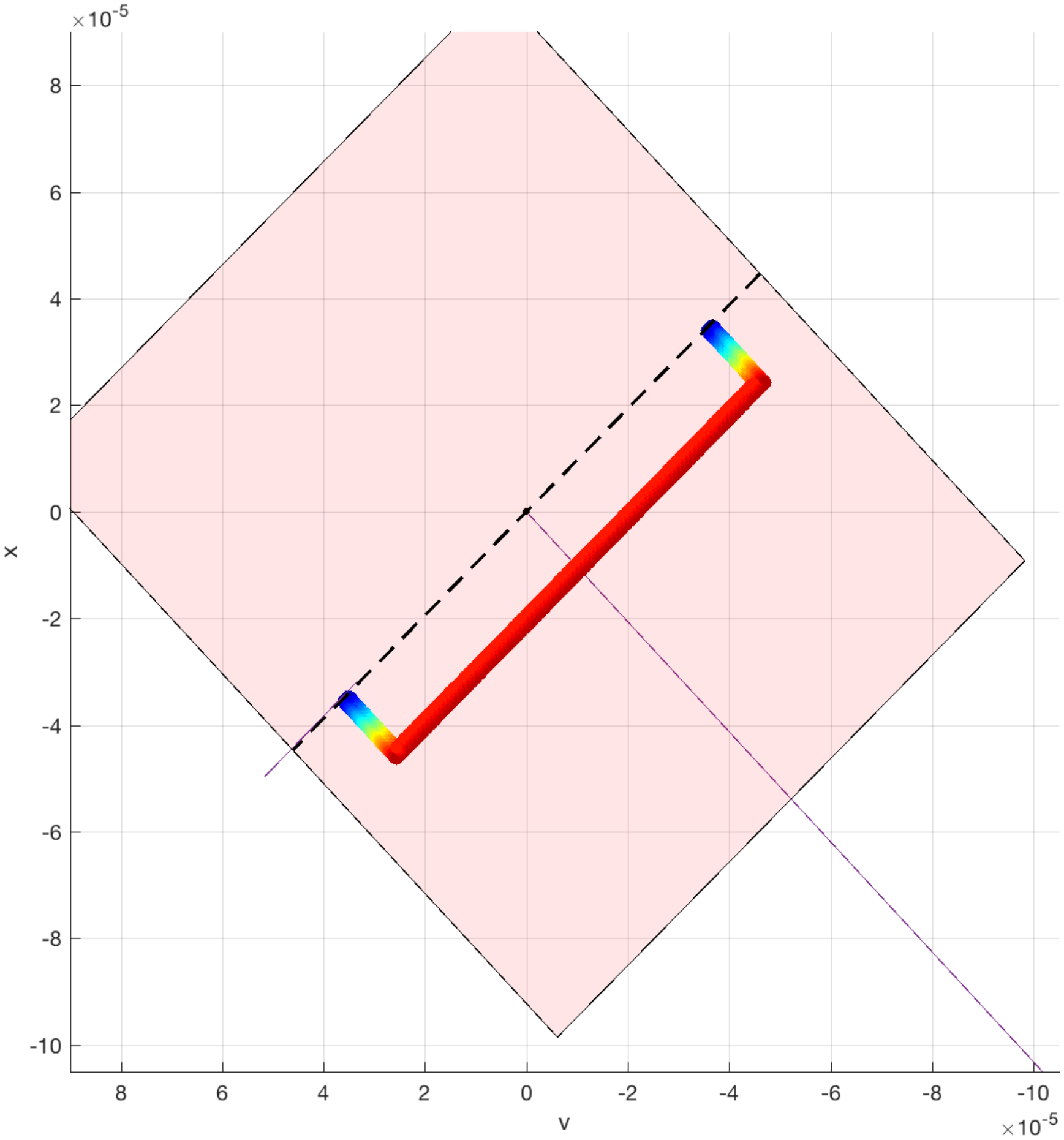}
  }\qquad
  \subfloat[$i = 1$
  ]{
    \includegraphics[width=0.38\textwidth]{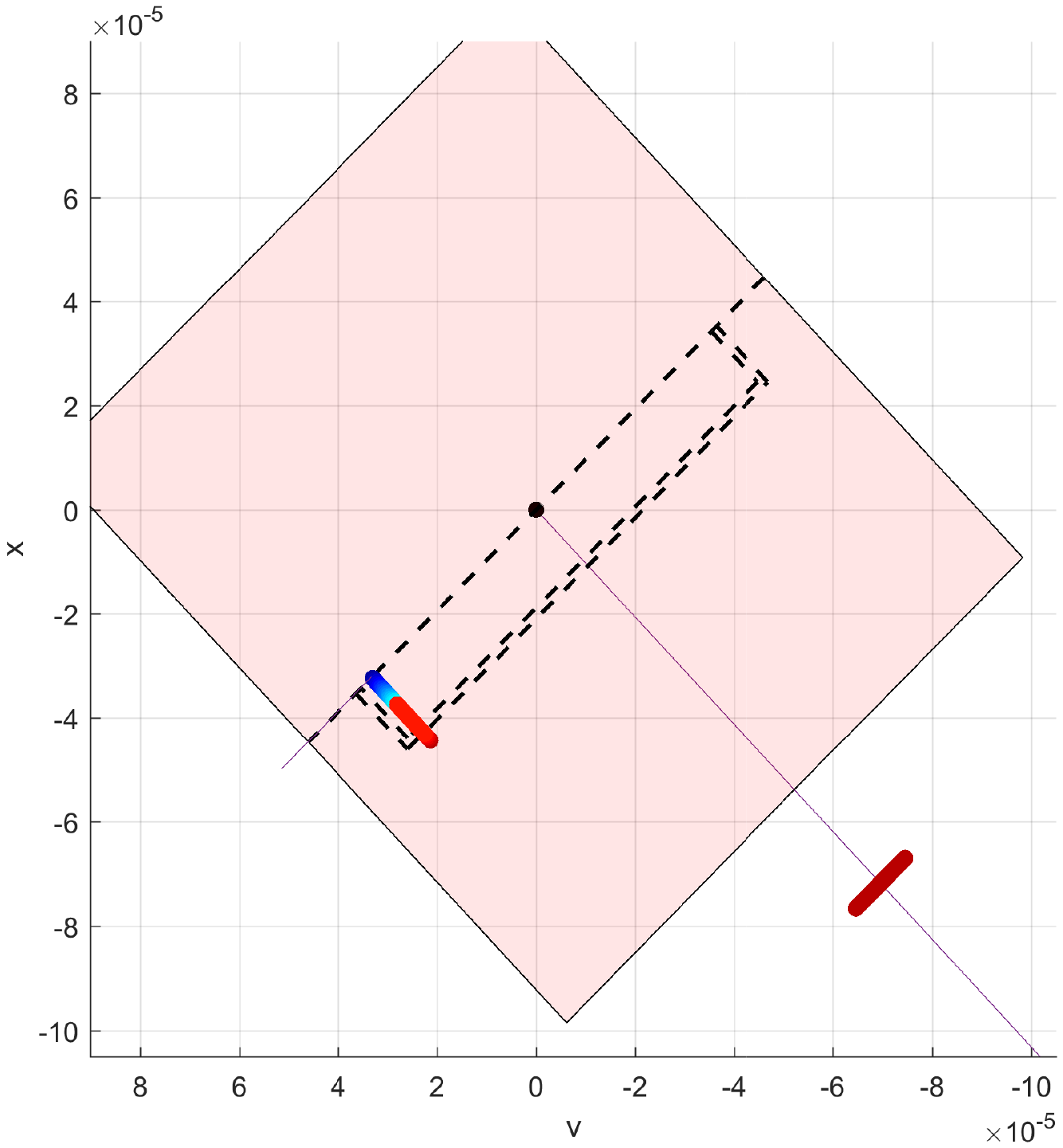}
  }

  \subfloat[$i = 25$
 ]{
    \includegraphics[width=0.38\textwidth]{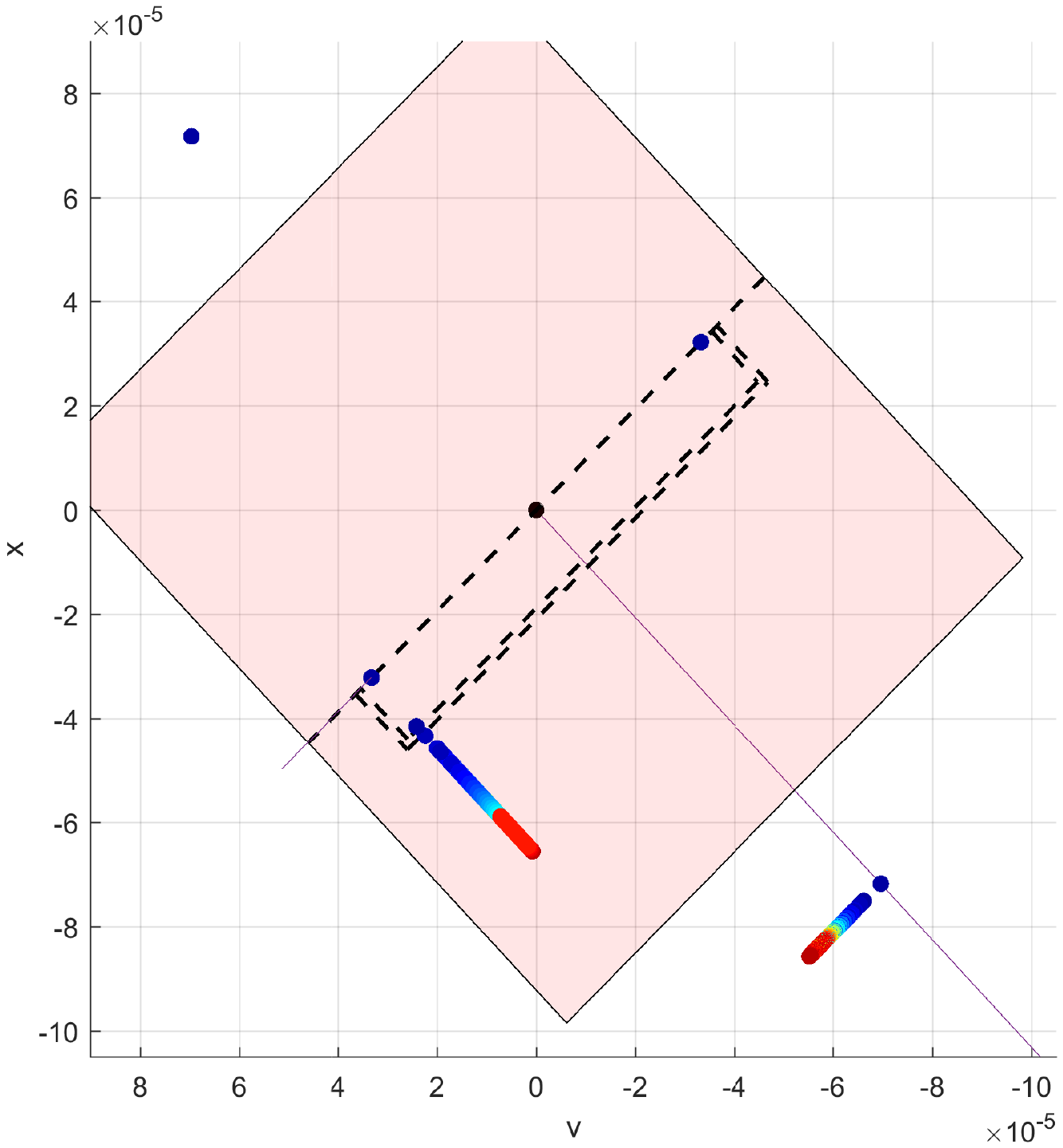}
  }\qquad
  \subfloat[$i = 50$
  ]{
    \includegraphics[width=0.38\textwidth]{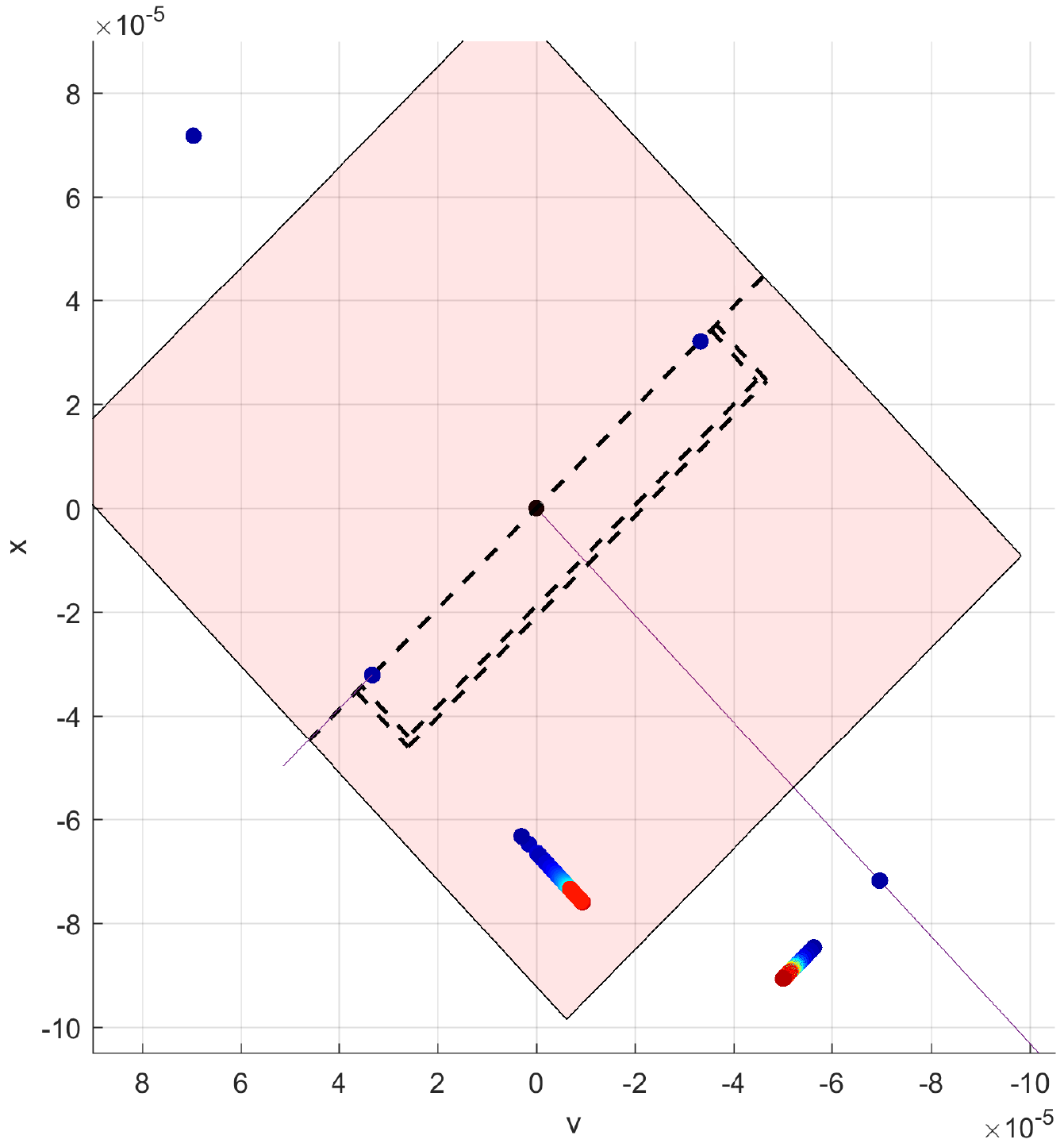}
  }

  \subfloat[$i = 75$
 ]{
    \includegraphics[width=0.38\textwidth]{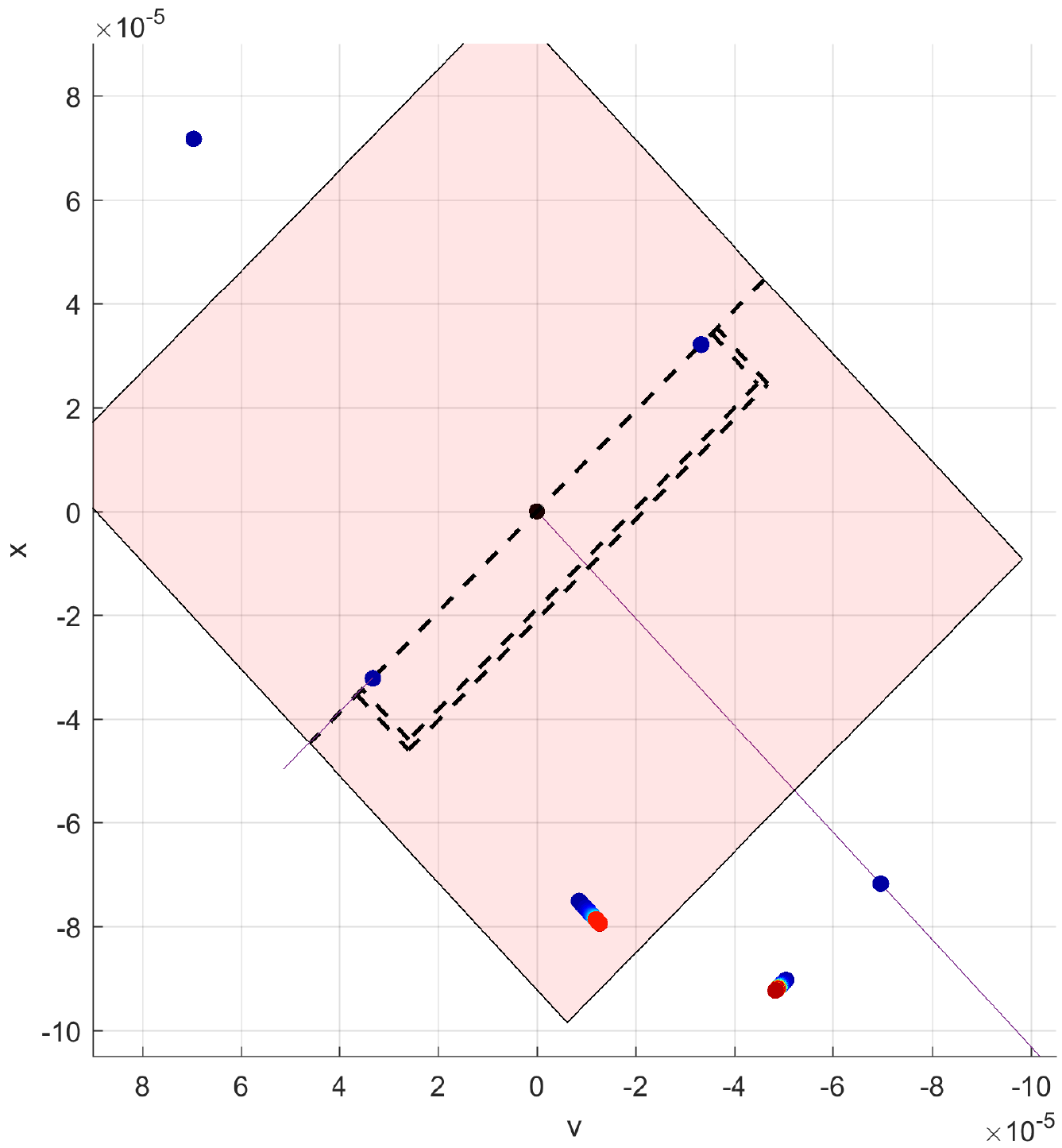}
  }\qquad
  \subfloat[$i = 100$
  ]{
    \includegraphics[width=0.38\textwidth]{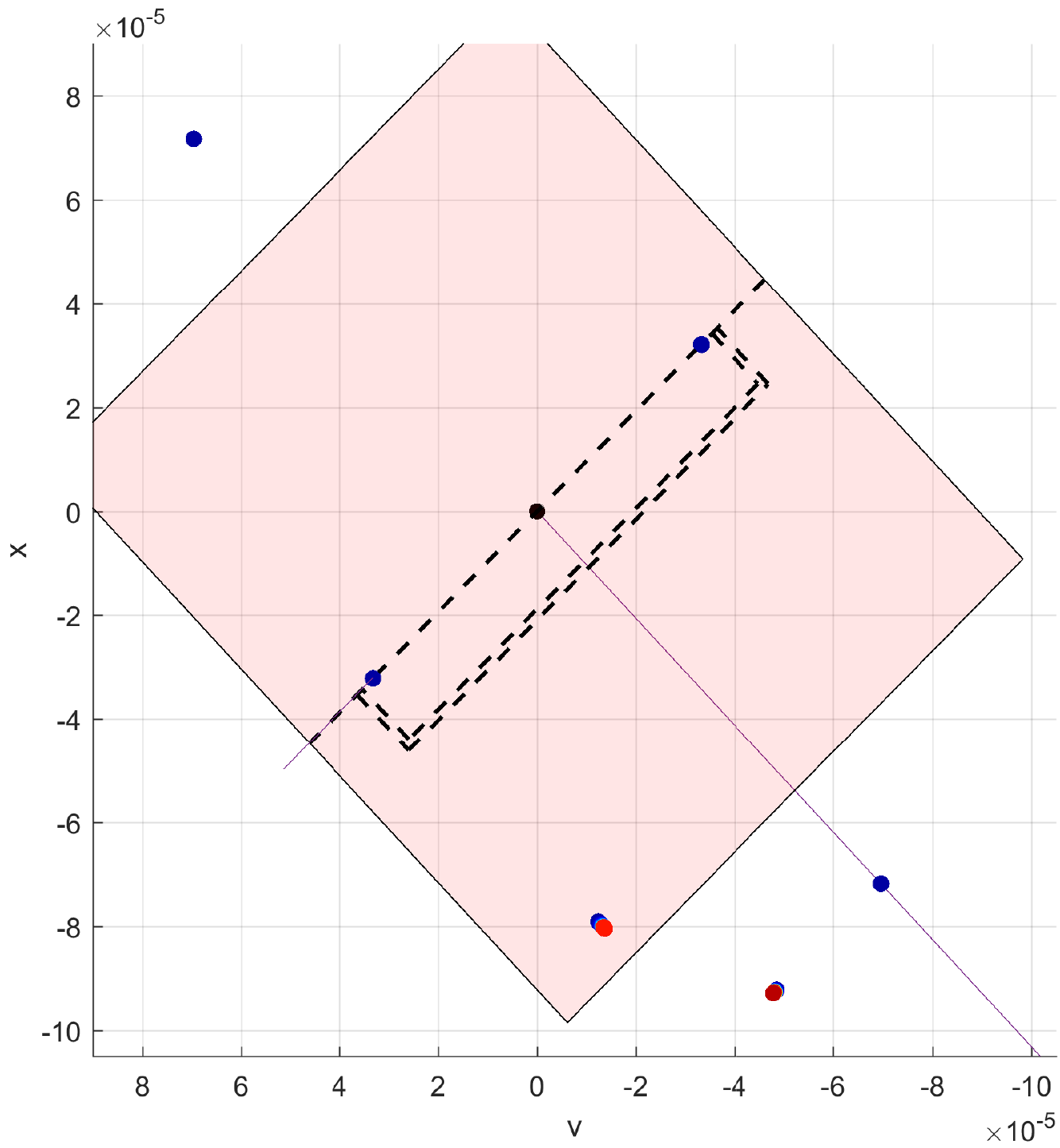}
  }
 \caption{Dynamics of the half-frame of points $\Sigma^{\rm in}_{\rm grid}$
 on the section $\Sigma^{\rm in}$
 under repeated applications of the Poincar\'{e} map
 $\Pi^i : \Sigma^{\rm in} \to \Sigma^{\rm in}$, $i = 1,2, \dots$,
 for $\delta = 0.9$, $\beta = 0.2$, $\text{\underline{s}} = 0.060131460578$
 (before bifurcation).}
 \label{fig:lorenz_like:lc8t_2lc:PM_before}
 \vspace{-1pt}
\end{figure}
\begin{figure}[h!]
 \centering
 \subfloat[$i = 0$
 ]{
    \includegraphics[width=0.38\textwidth]{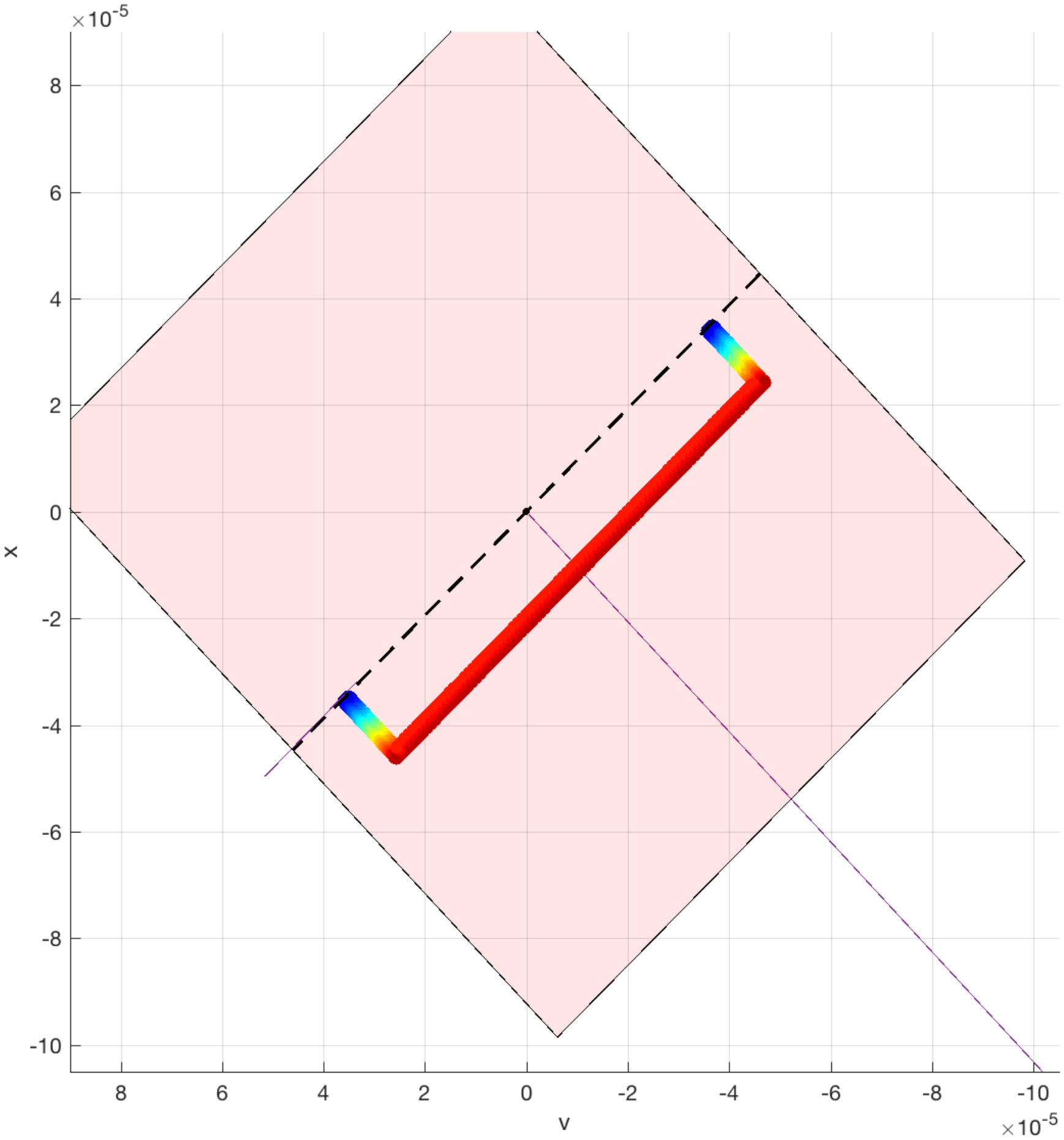}
  }\qquad
  \subfloat[$i = 1$
  ]{
    \includegraphics[width=0.38\textwidth]{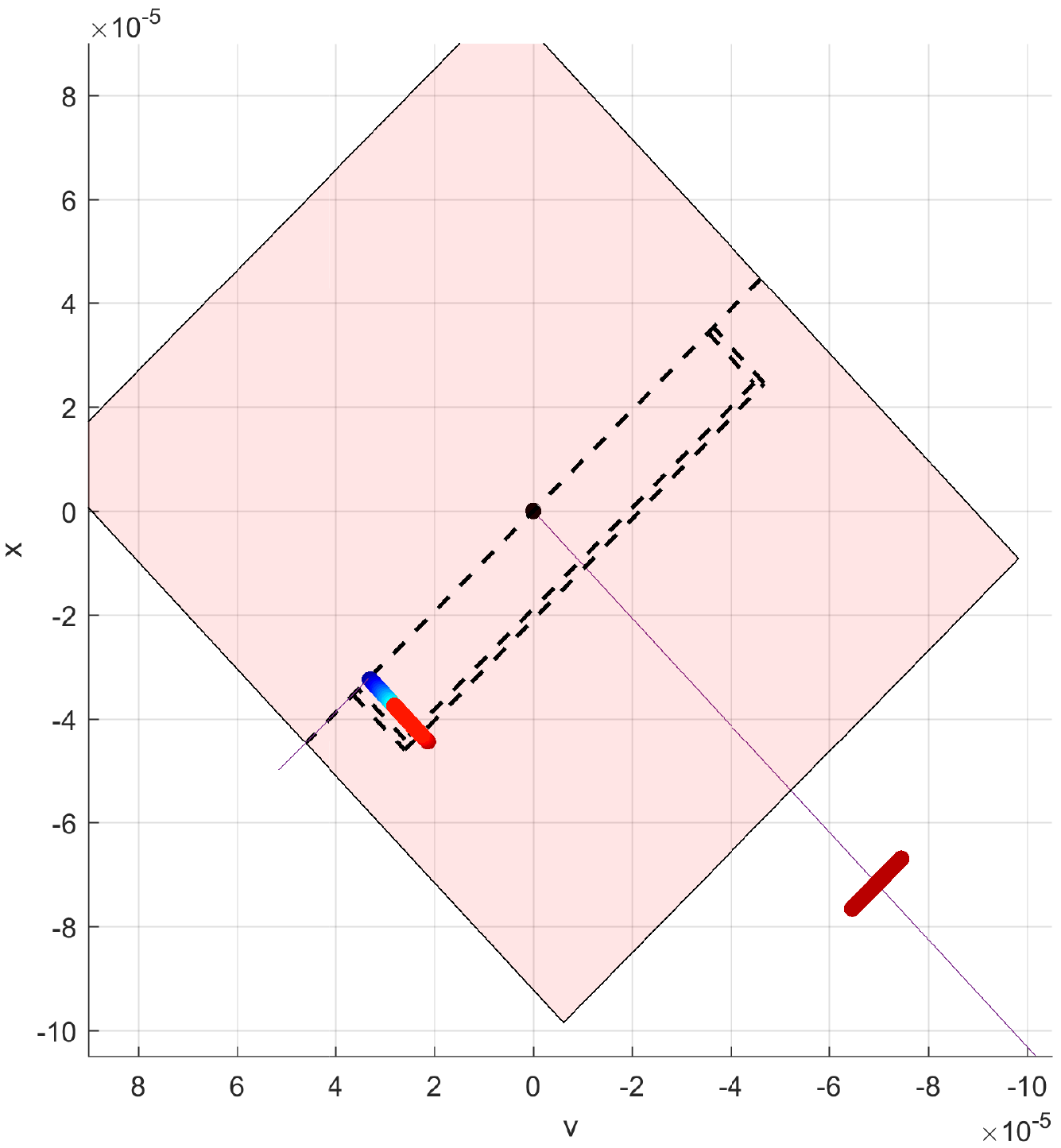}
  }

  \subfloat[$i = 25$
 ]{
    \includegraphics[width=0.38\textwidth]{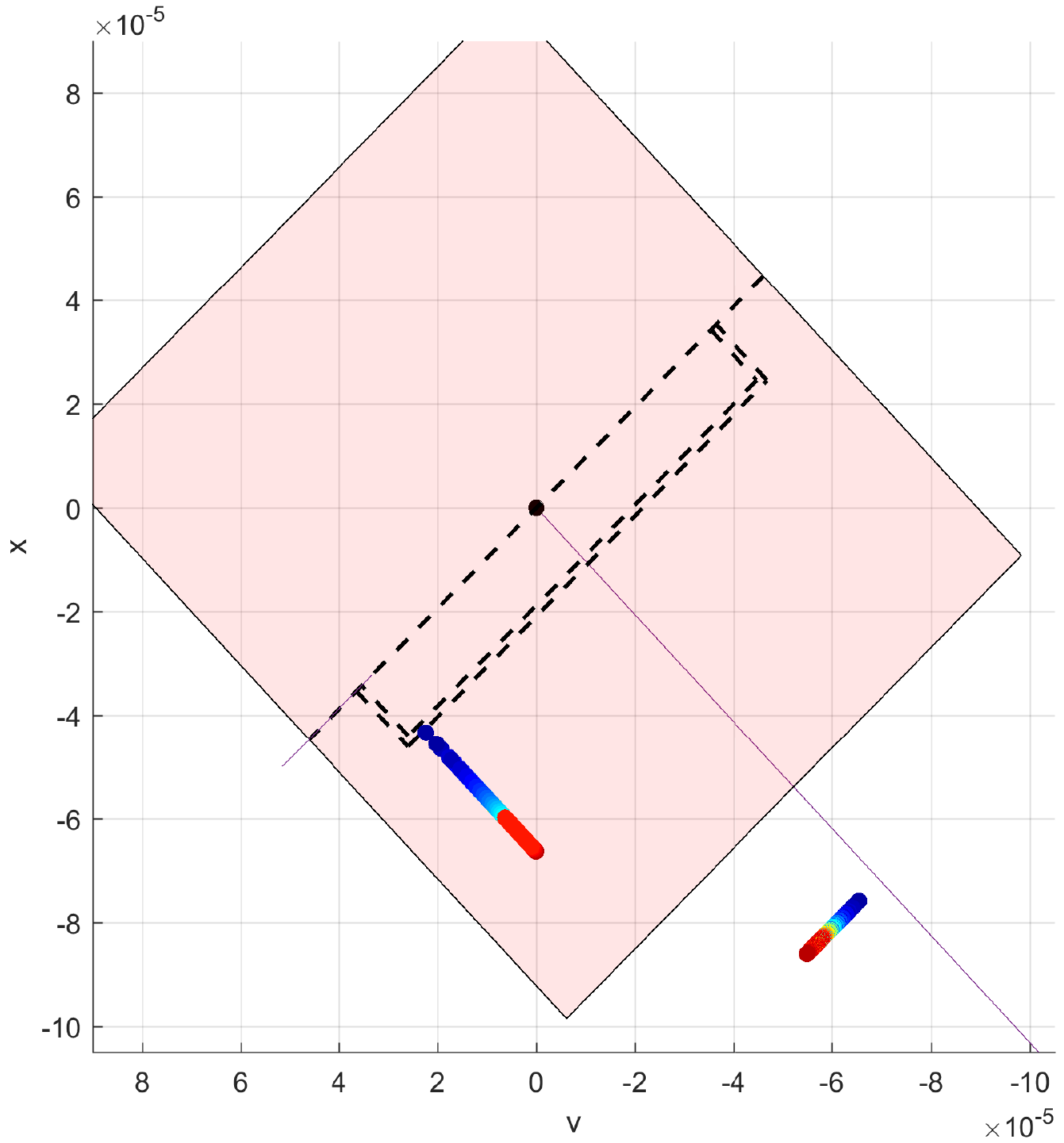}
  }\qquad
  \subfloat[$i = 50$
  ]{
    \includegraphics[width=0.38\textwidth]{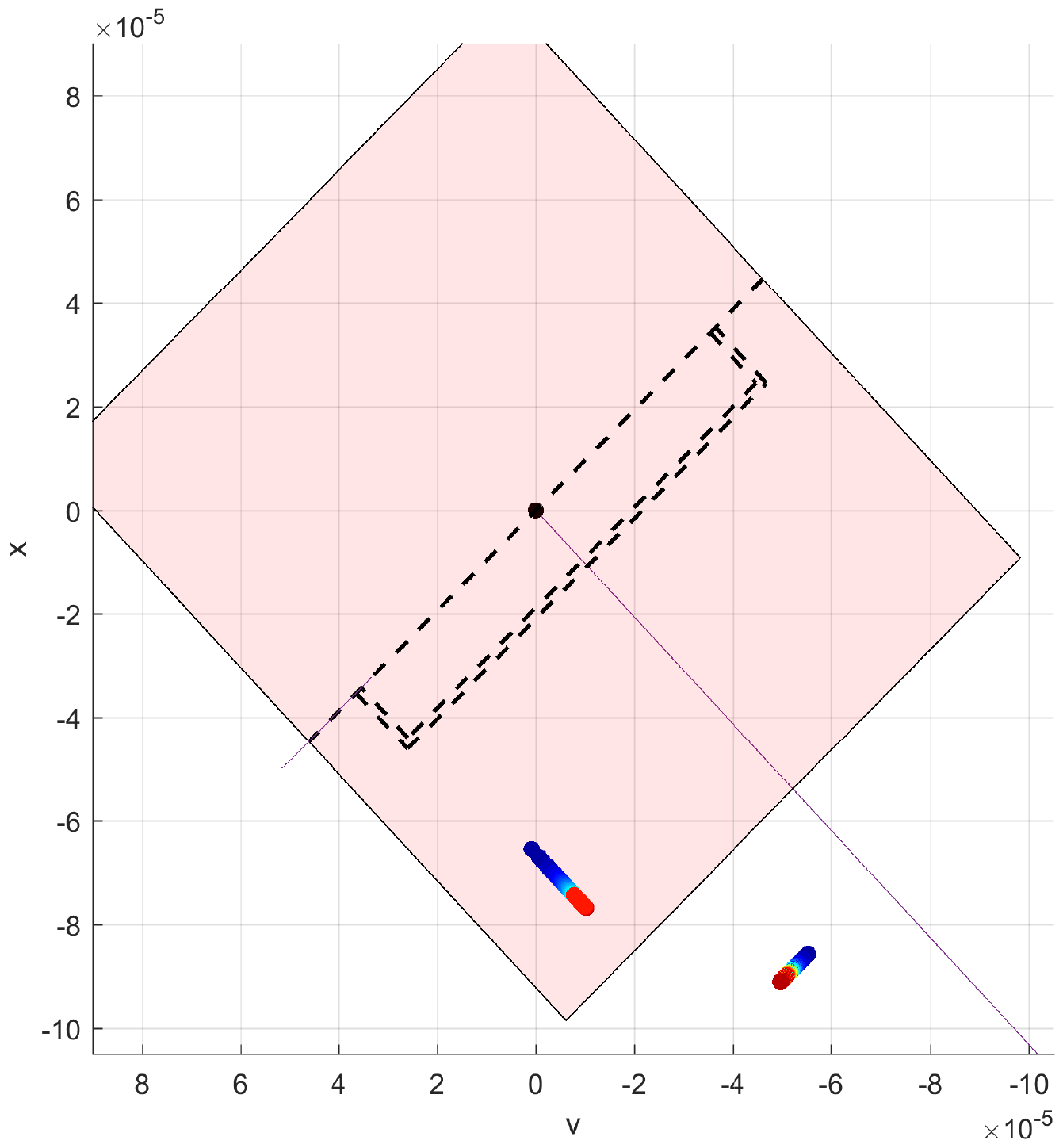}
  }

  \subfloat[$i = 75$
 ]{
    \includegraphics[width=0.38\textwidth]{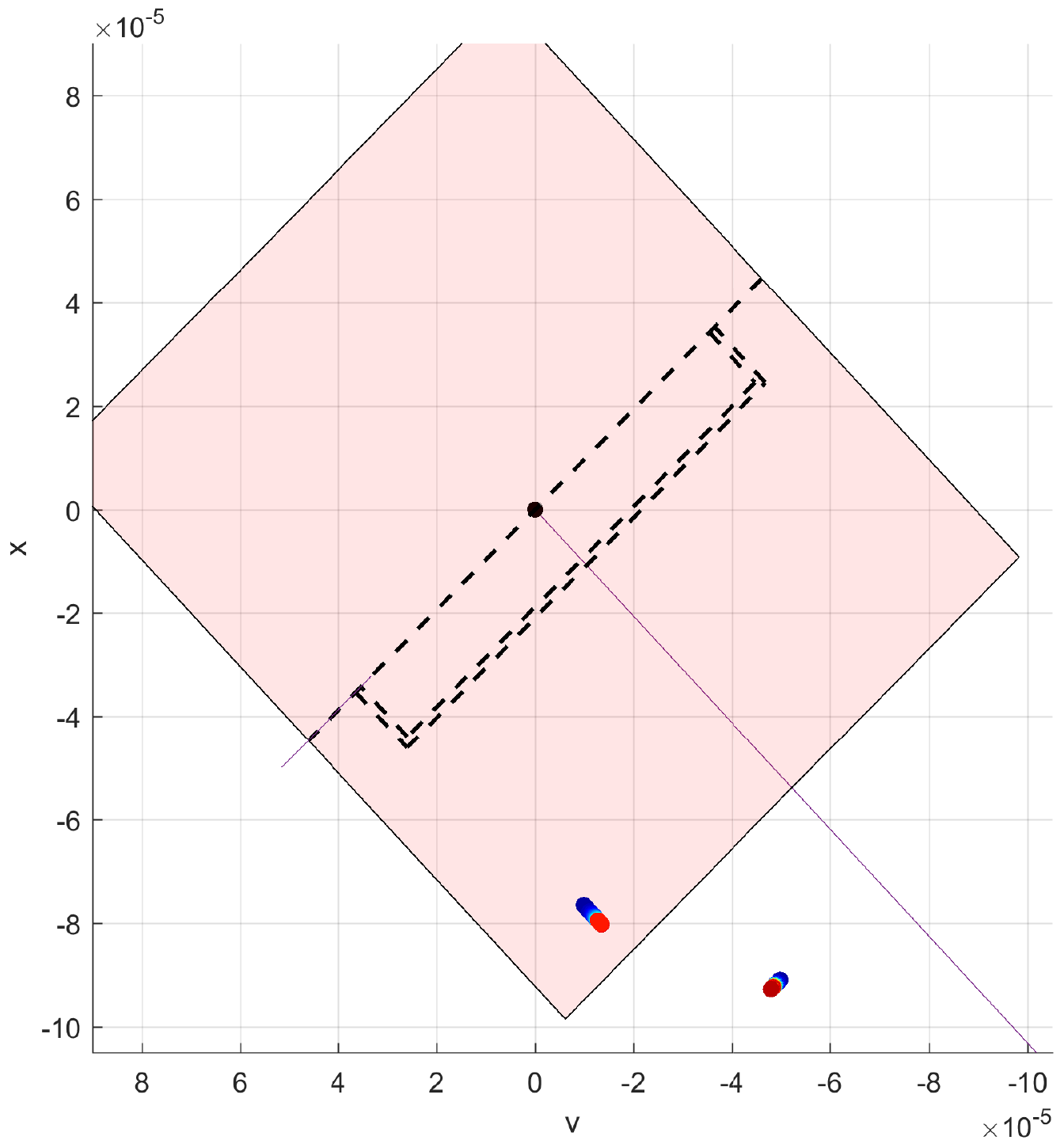}
  }\qquad
  \subfloat[$i = 100$
  ]{
    \includegraphics[width=0.38\textwidth]{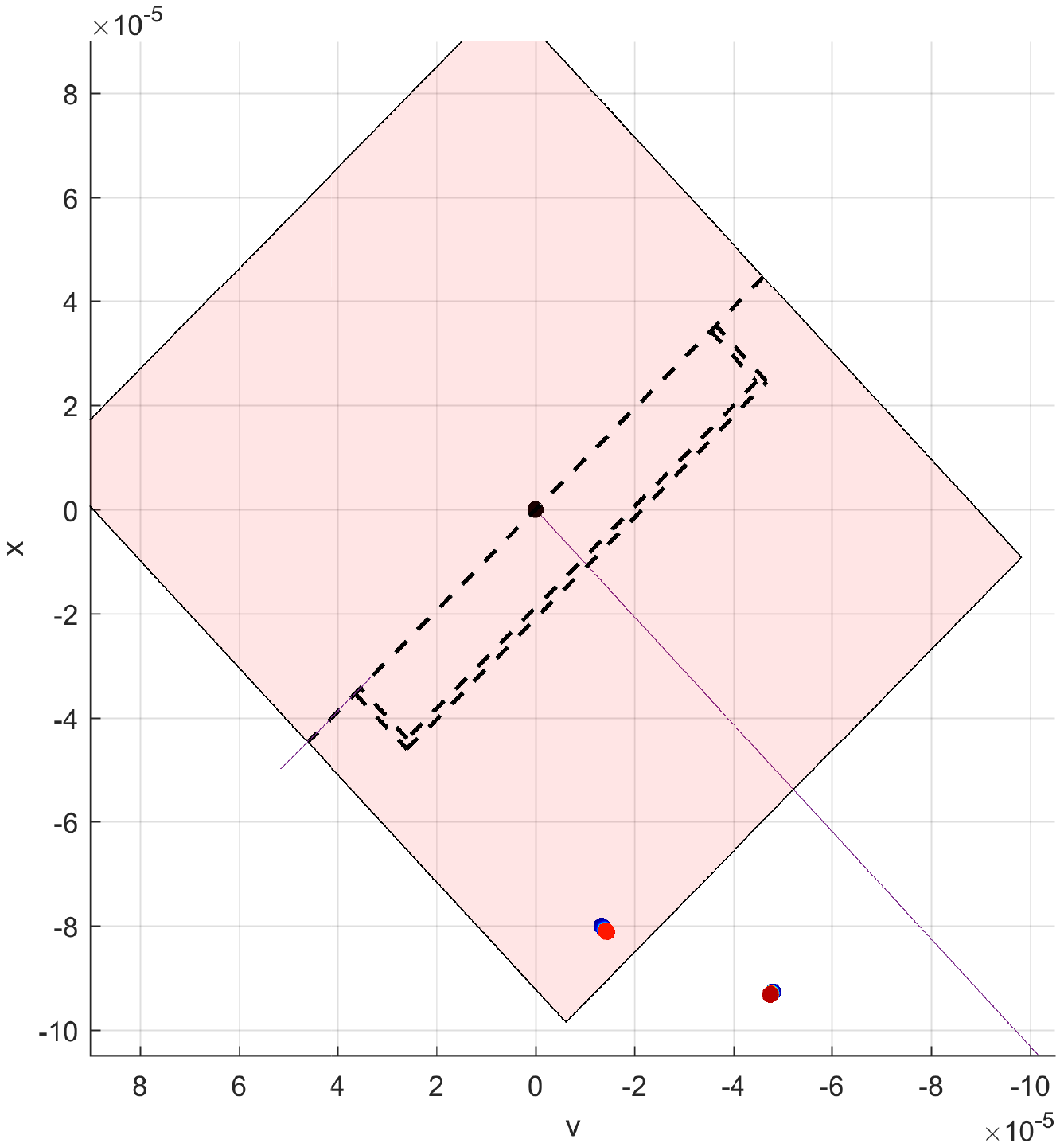}
  }
 \caption{Dynamics of the half-frame of points $\Sigma^{\rm in}_{\rm grid}$
 on the section $\Sigma^{\rm in}$
 under repeated applications of the Poincar\'{e} map
 $\Pi^i : \Sigma^{\rm in} \to \Sigma^{\rm in}$, $i = 1,2, \dots$,
 for $\delta = 0.9$, $\beta = 0.2$, $\overline{s} = 0.060131460581$
 (after bifurcation).}
 \label{fig:lorenz_like:lc8t_2lc:PM_after}
 \vspace{-1pt}
\end{figure}
\begin{figure}[h!]
 \centering
 \subfloat[$i = 0$
 ]{
    \includegraphics[width=0.4\textwidth]{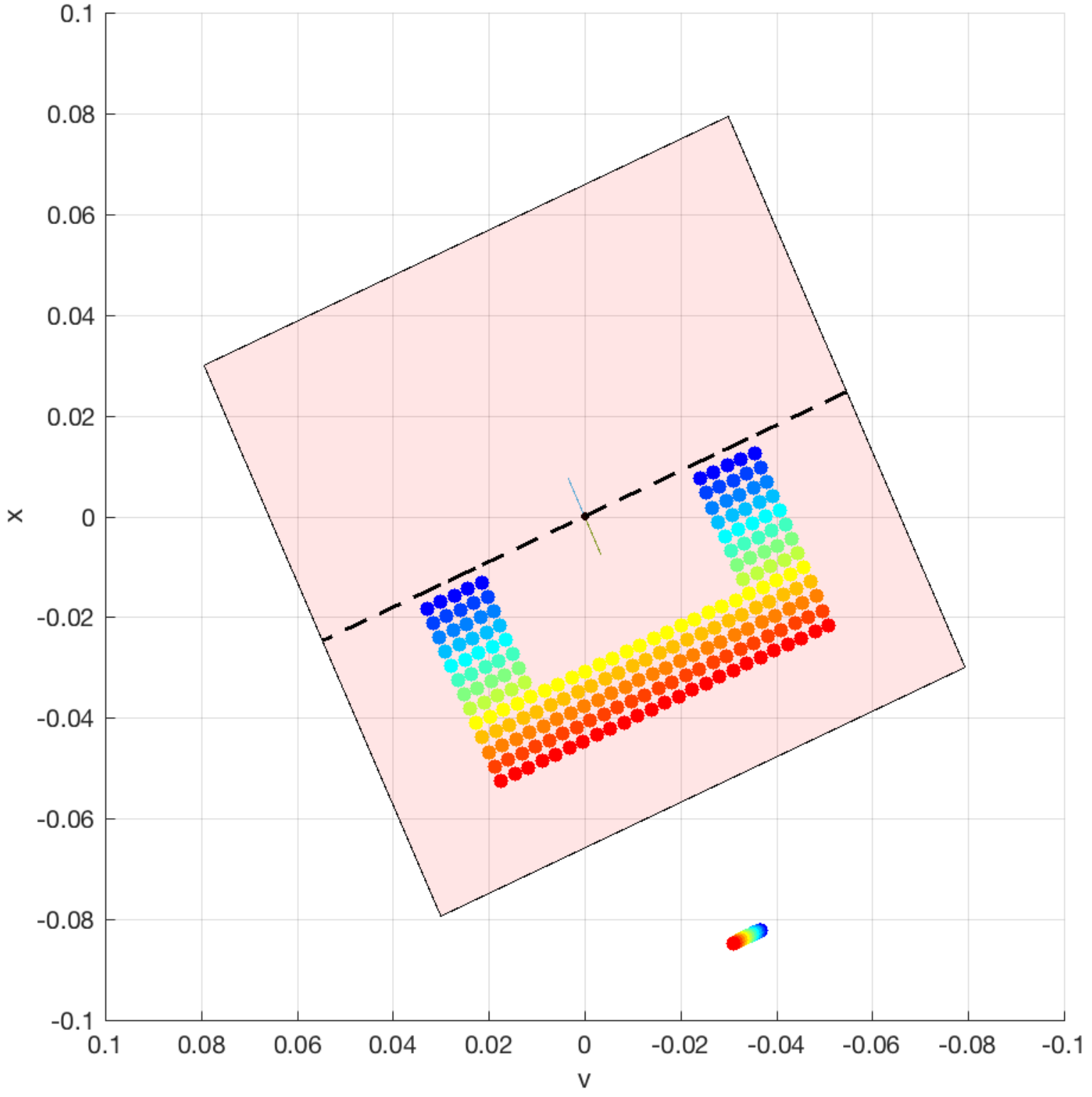}
  }\qquad
  \subfloat[$i = 1$
  ]{
    \includegraphics[width=0.4\textwidth]{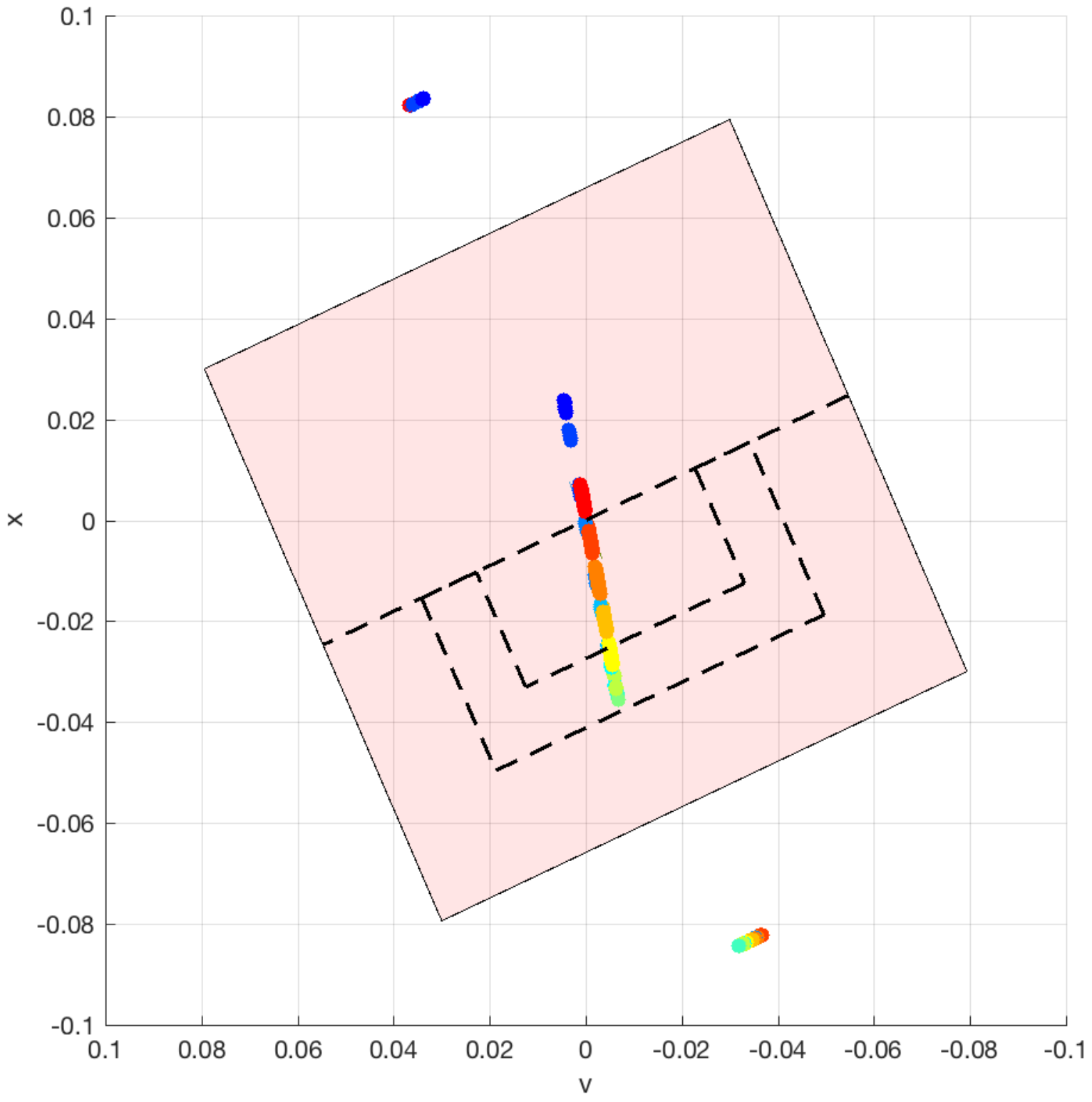}
  }

  \subfloat[$i = 25$
 ]{
    \includegraphics[width=0.4\textwidth]{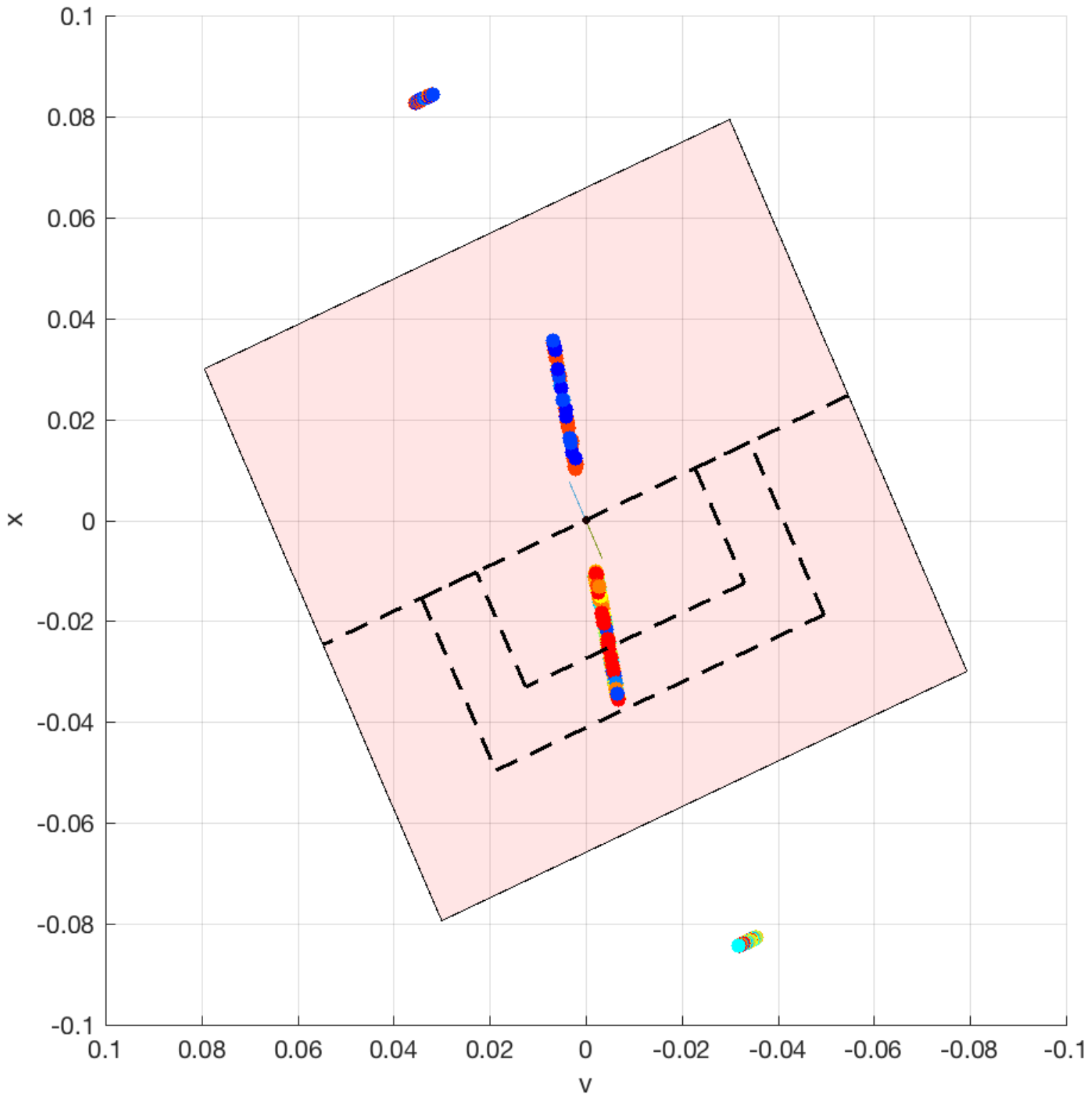}
  }\qquad
  \subfloat[$i = 50$
  ]{
    \includegraphics[width=0.4\textwidth]{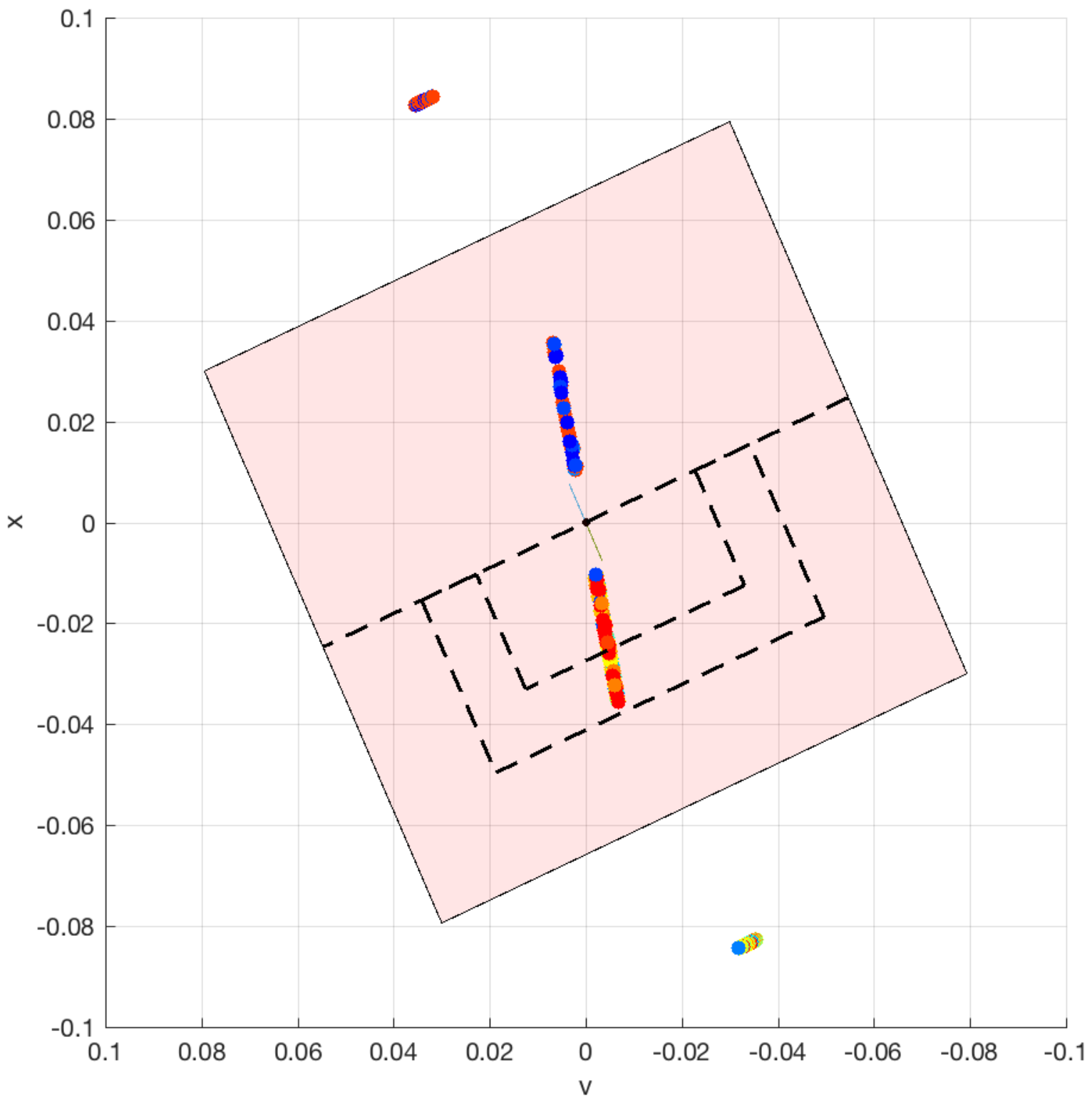}
  }

  \subfloat[$i = 75$
 ]{
    \includegraphics[width=0.4\textwidth]{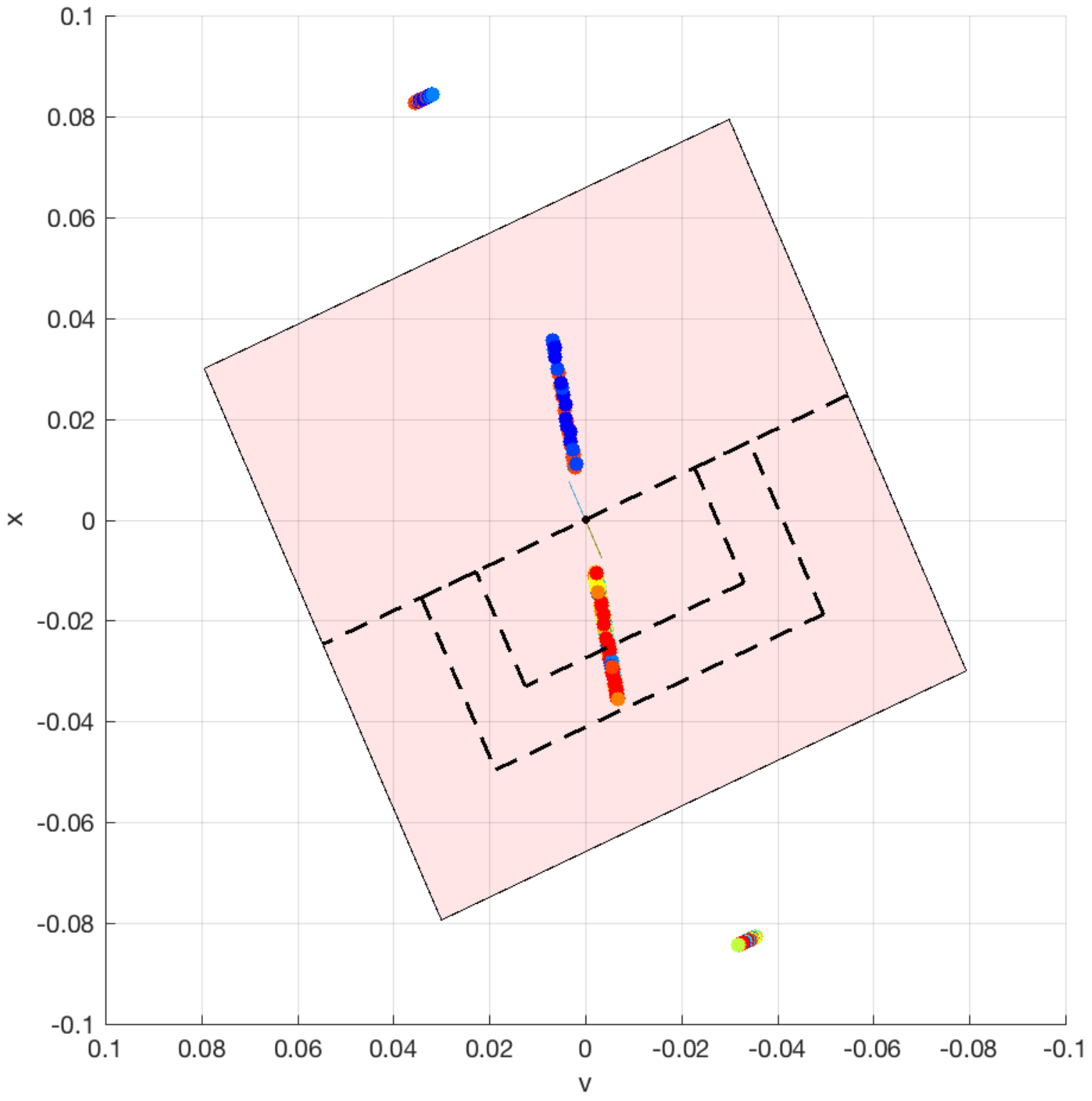}
  }\qquad
  \subfloat[$i = 100$
  ]{
    \includegraphics[width=0.4\textwidth]{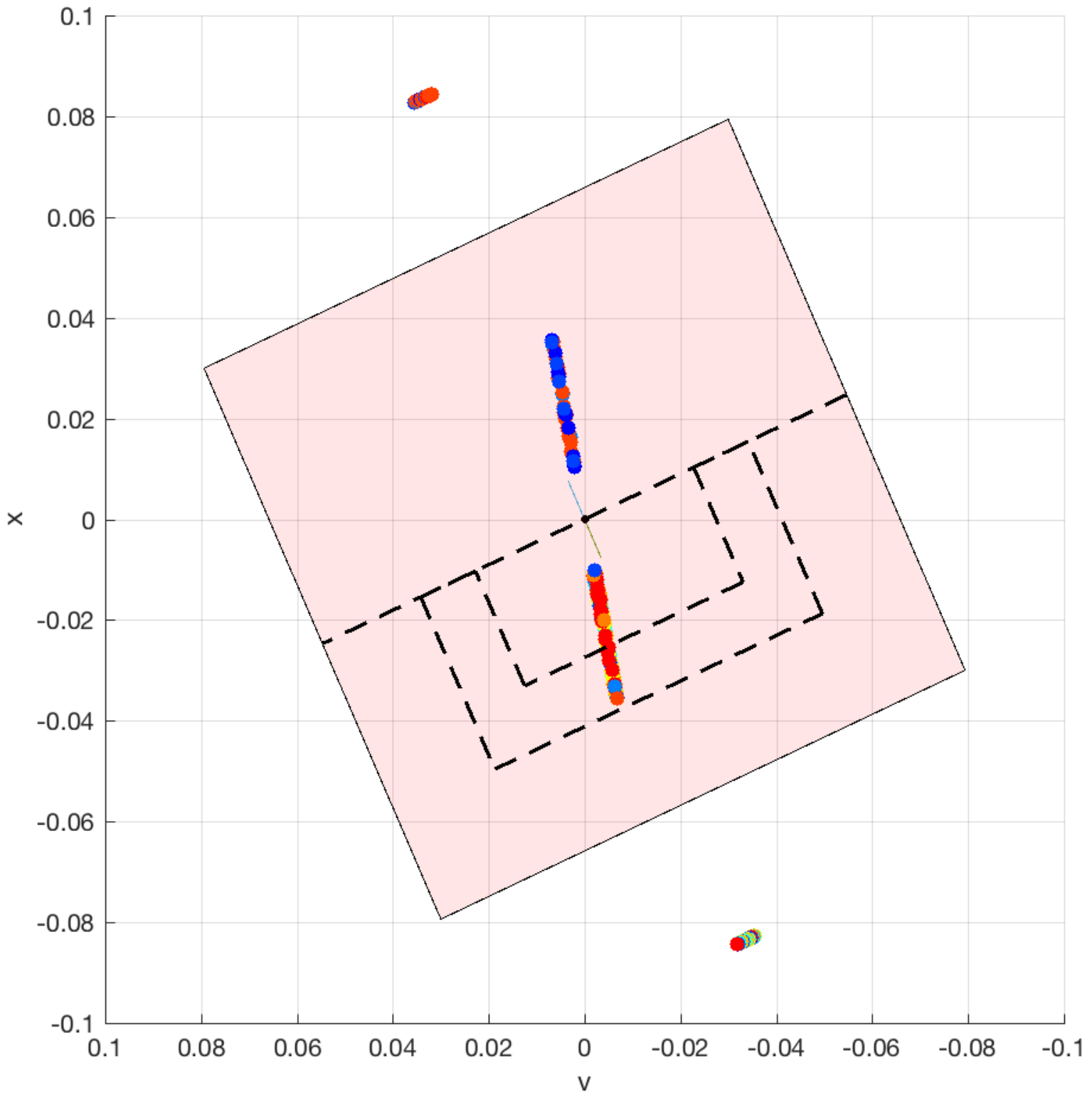}
  }
 \caption{Dynamics of the half-frame of points $\Sigma^{\rm in}_{\rm grid}$
 on the section $\Sigma^{\rm in}$
 under repeated applications of the Poincar\'{e} map
 $\Pi^i : \Sigma^{\rm in} \to \Sigma^{\rm in}$, $i = 1,2, \dots$,
 for $\delta = 0.9$, $\beta = 2.899$, $\text{\underline{s}} = 0.7955$
 (before bifurcation).}
 \label{fig:lorenz_like:unstable_homo:PM_before}
 \vspace{-1pt}
\end{figure}
\begin{figure}[h!]
 \centering
 \subfloat[$i = 0$
 ]{
    \includegraphics[width=0.4\textwidth]{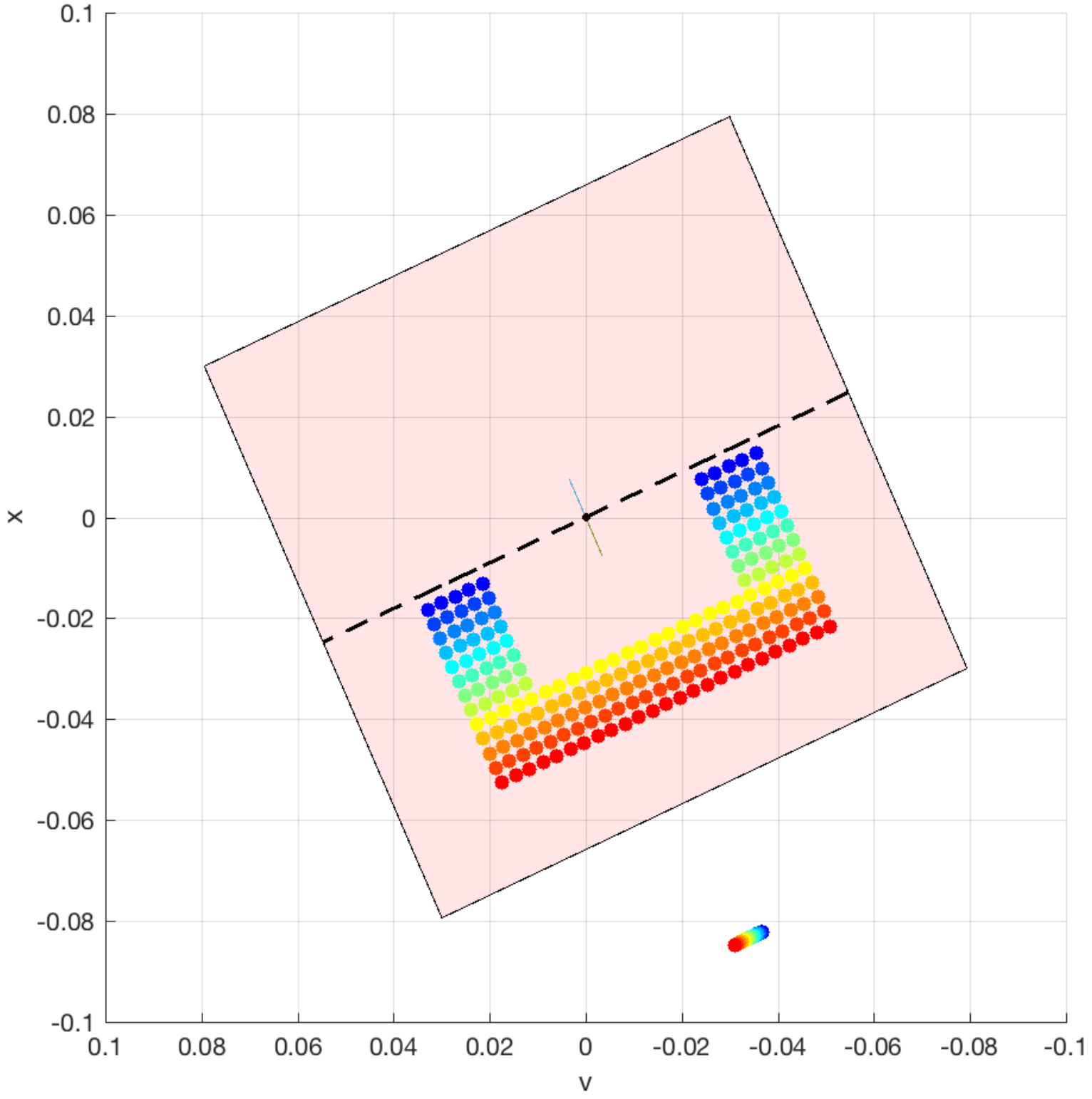}
  }\qquad
  \subfloat[$i = 1$
  ]{
    \includegraphics[width=0.4\textwidth]{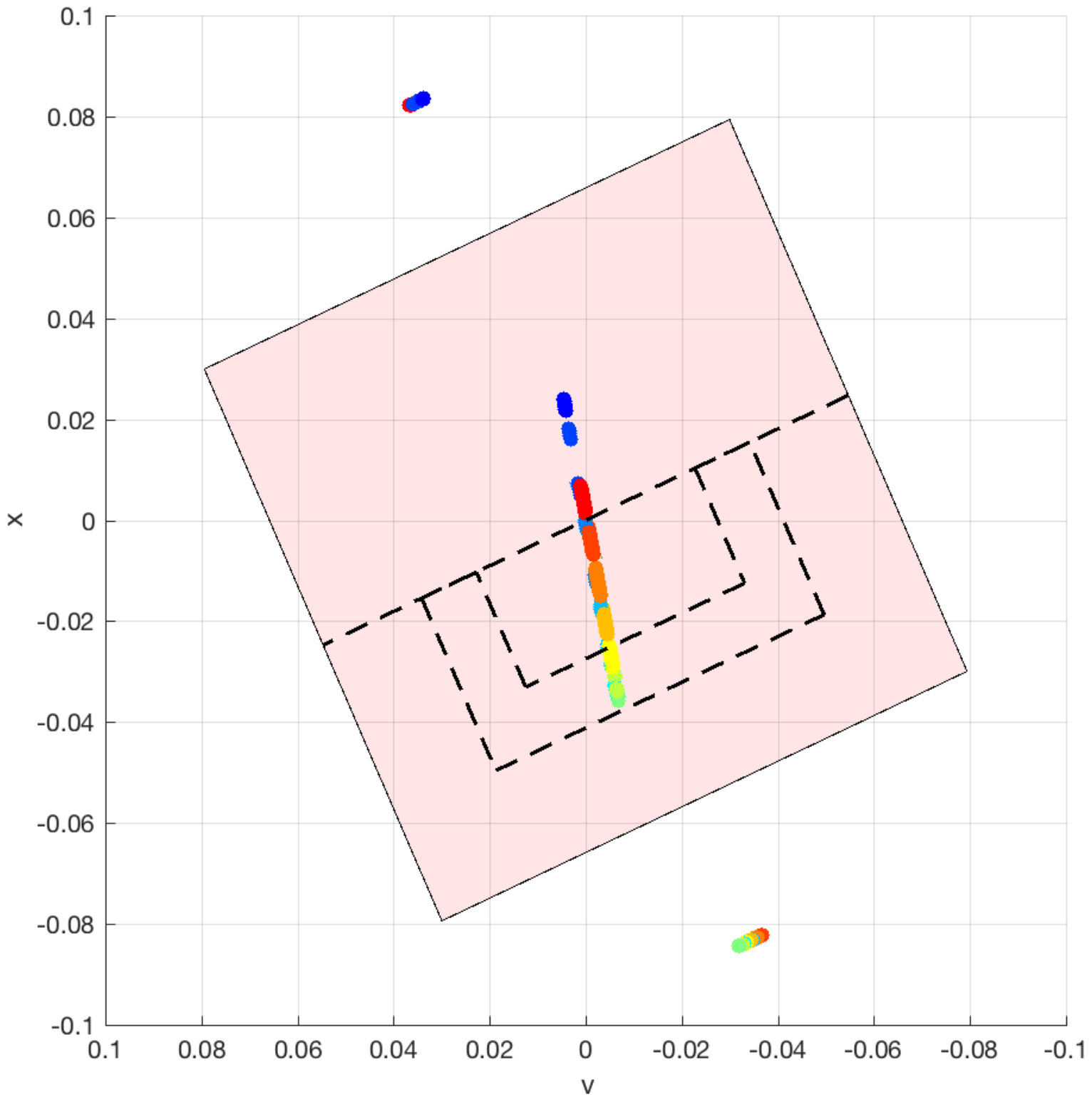}
  }

  \subfloat[$i = 25$
 ]{
    \includegraphics[width=0.4\textwidth]{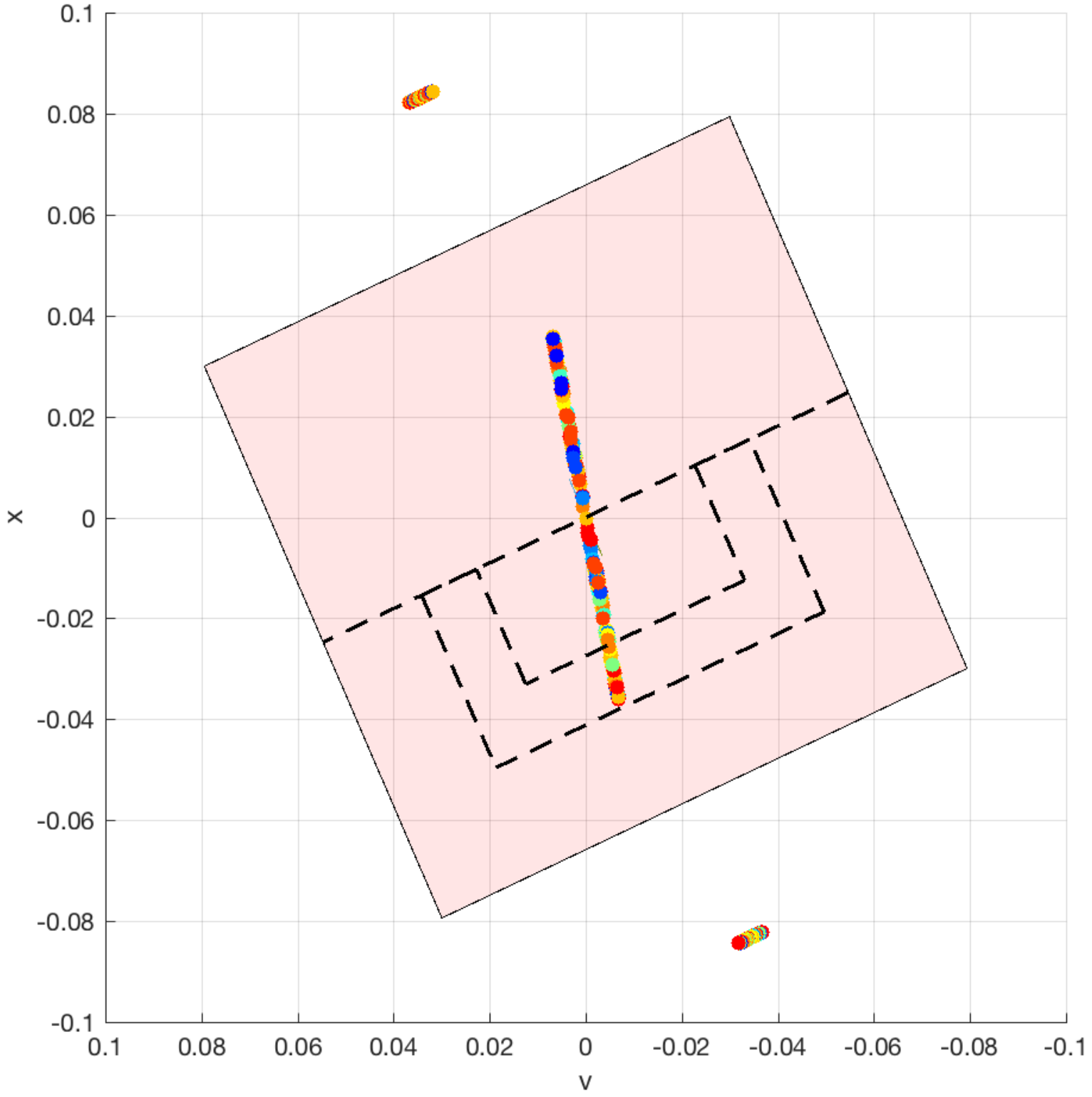}
  }\qquad
  \subfloat[$i = 50$
  ]{
    \includegraphics[width=0.4\textwidth]{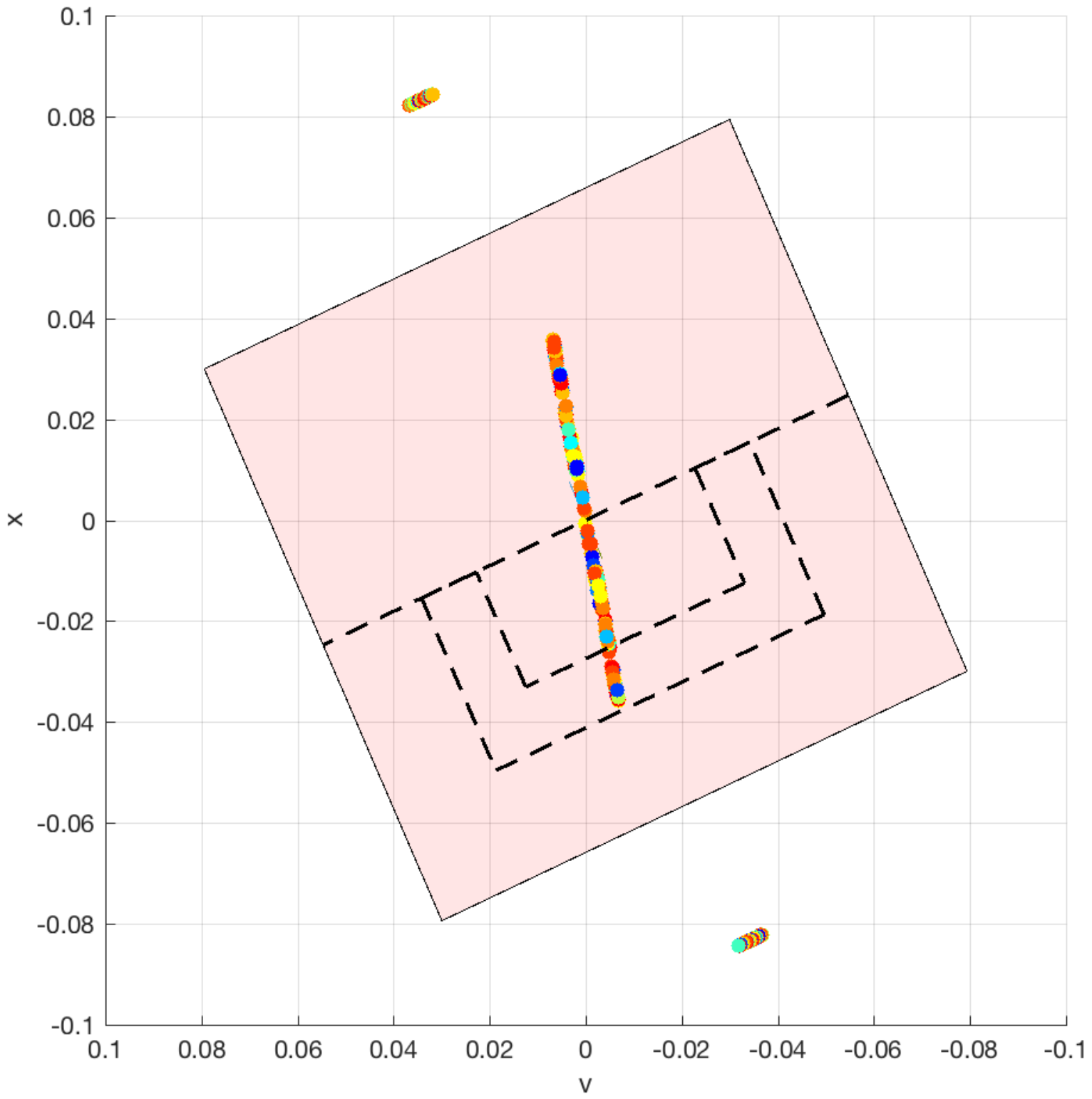}
  }

  \subfloat[$i = 75$
  ]{
    \includegraphics[width=0.4\textwidth]{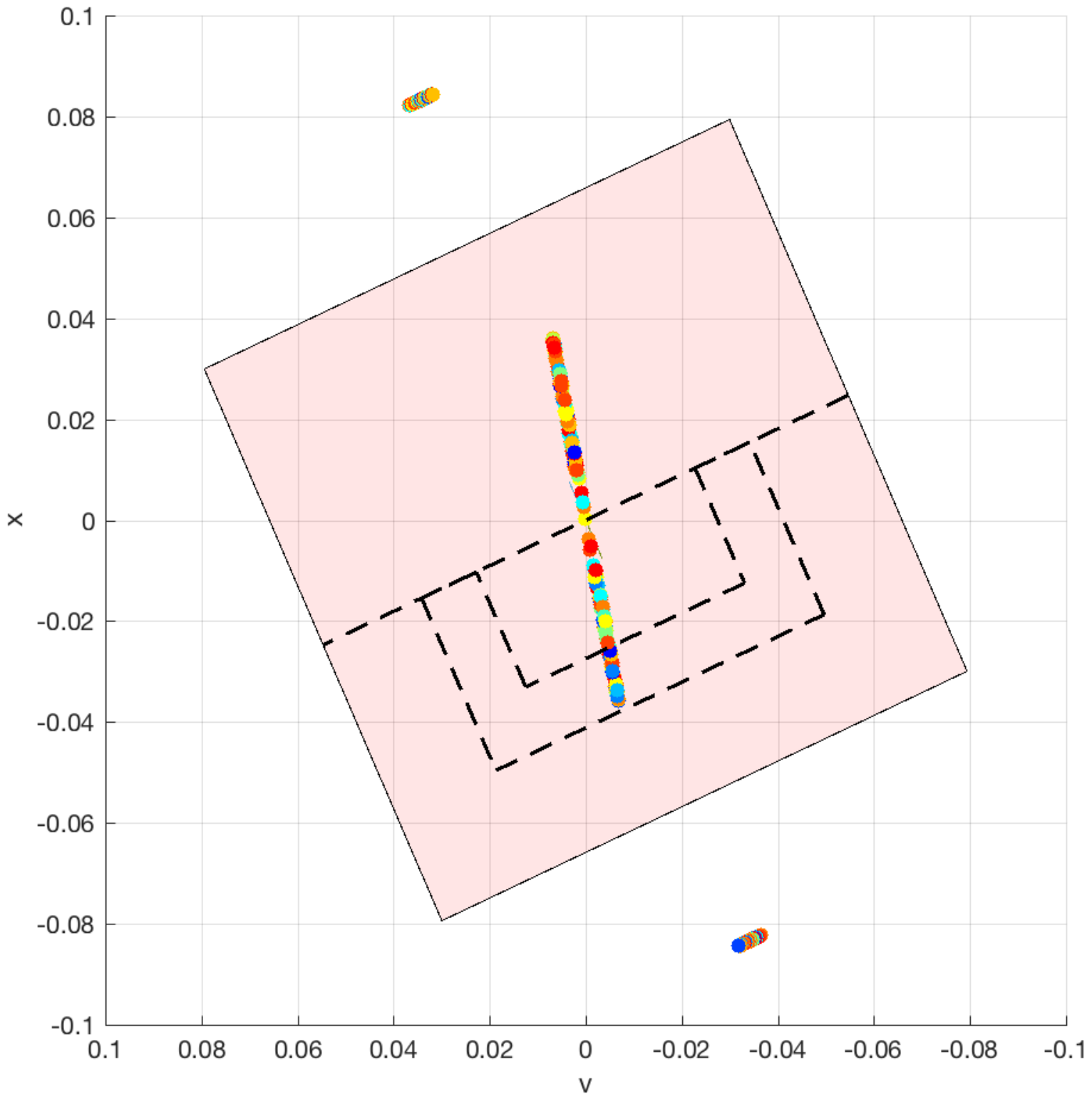}
  }\qquad
  \subfloat[$i = 100$
  ]{
    \includegraphics[width=0.4\textwidth]{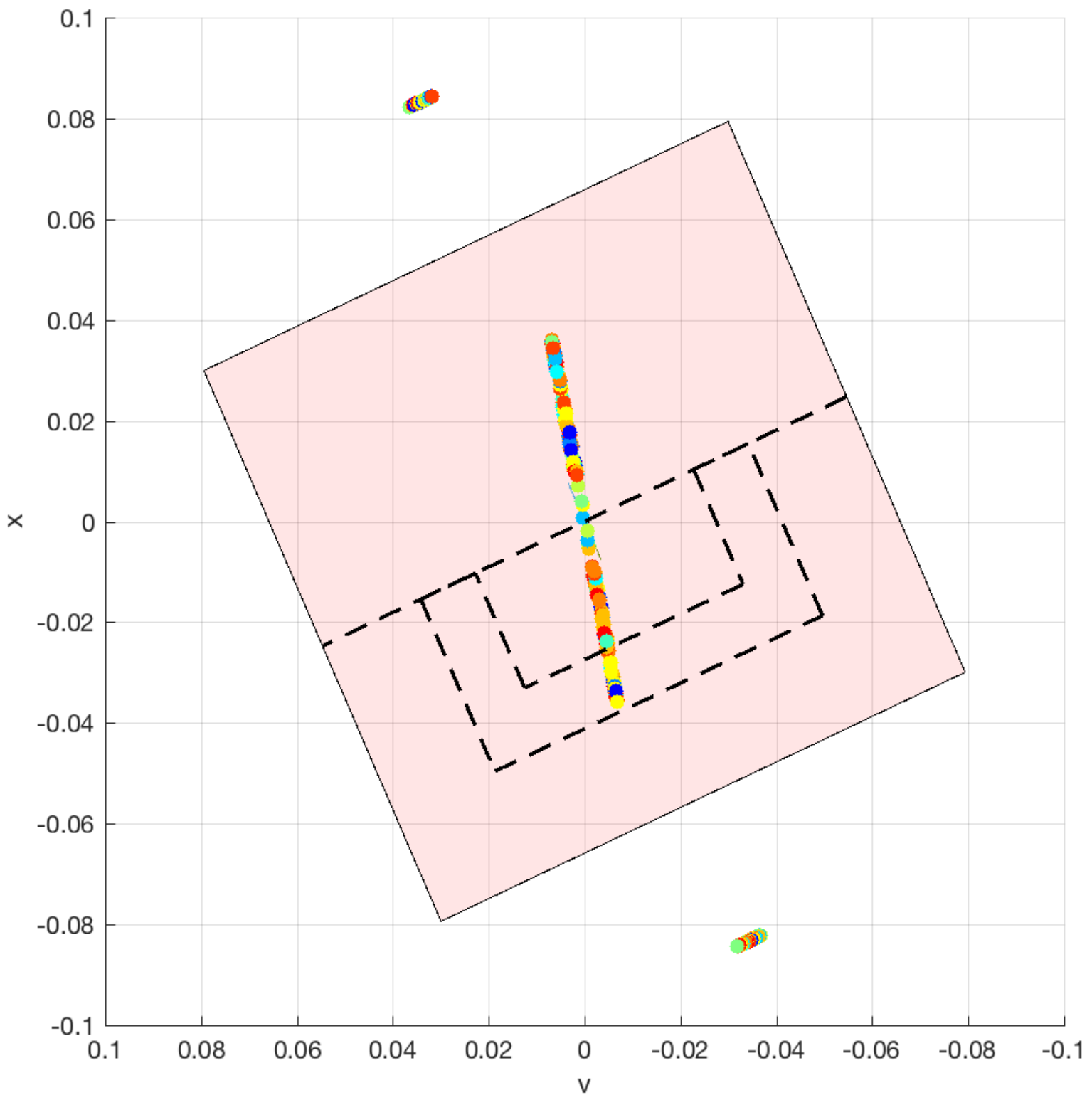}
  }
 \caption{Dynamics of the half-frame of points $\Sigma^{\rm in}_{\rm grid}$
 on the section $\Sigma^{\rm in}$
 under repeated applications of the Poincar\'{e} map
 $\Pi^i : \Sigma^{\rm in} \to \Sigma^{\rm in}$, $i = 1,2, \dots$,
 for $\delta = 0.9$, $\beta = 2.899$, $\overline{s} = 0.7958$
 (after bifurcation).}
 \label{fig:lorenz_like:unstable_homo:PM_after}
 \vspace{-1pt}
\end{figure}
\end{document}